\documentclass[12pt,a4paper,leqno]{amsart}
\usepackage[utf8]{inputenc}
\pdfoutput=1

\usepackage{amssymb,amsfonts,mathtools,mathrsfs}   
\usepackage{graphicx}
\usepackage{placeins}
\usepackage{enumerate}
\usepackage{tikz}
\usepackage{setspace}
\usepackage{xcolor}
\usepackage{hyperref}
\usepackage{theoremref}

\usepackage[
  margin=1.5cm,
  includefoot,
  footskip=30pt,
]{geometry}

\usepackage[
backend=biber,
style=alphabetic,
maxbibnames=99, 
]{biblatex}
\renewcommand{\epsilon}{\varepsilon}

\renewcommand{\le}{\leqslant}

\renewcommand{\ge}{\geqslant}

\newtheorem{thm}{Theorem}[section]
\newtheorem{cor}[thm]{Corollary} 
\newtheorem{lem}[thm]{Lemma} 
\newtheorem{prop}[thm]{Proposition}

\theoremstyle{definition} 
\newtheorem{defn}[thm]{Definition}

\theoremstyle{remark}
\newtheorem{remarkx}[thm]{Remark}
\newenvironment{remark}
  {\pushQED{\qed}\remarkx}
  {\popQED\endremarkx} 

\numberwithin{equation}{section}

\newcommand{\R}{\mathbb{R}}
 
\newcommand{\Z}{\mathbb{Z}}

\newcommand{\Sph}{\mathbb{S}}

\newcommand{\diam}{\operatorname{diam}}
\newcommand{\dist}{\operatorname{dist}}

\newcommand{\PV}{\mathrm{P.V.}}

\DeclareMathOperator*{\essinf}{ess\,inf}
\newcommand{\dd}{\mathop{}\!d} 

\def\XXint#1#2#3{{\setbox0=\hbox{$#1{#2#3}{\int}$}
\vcenter{\hbox{$#2#3$}}\kern-.5\wd0}}

\newenvironment{PDE}
	{ \left \{
	\begin{array}{r@{ \ }l @{\quad \: \;} l}
	}
	{
	\end{array} \right . 
	}

\begin{document}

\title[Quantitative stability for the nonlocal parallel surface problem]{Quantitative stability \\ for overdetermined nonlocal problems \\ with parallel surfaces \\ and investigation of the stability exponents}

\author[S. Dipierro]{Serena Dipierro}
\address{Serena Dipierro: Department of Mathematics and Statistics, The University of Western Australia, 35 Stirling Highway, Crawley, Perth, WA 6009, Australia}
\email{serena.dipierro@uwa.edu.au}

\author[G. Poggesi]{Giorgio Poggesi}
\address{Giorgio Poggesi: Department of Mathematics and Statistics, The University of Western Australia, 35 Stirling Highway, Crawley, Perth, WA 6009, Australia}
\email{giorgio.poggesi@uwa.edu.au}

\author[J. Thompson]{Jack Thompson}
\address{Jack Thompson: Department of Mathematics and Statistics, The University of Western Australia, 35 Stirling Highway, Crawley, Perth, WA 6009, Australia}
\email{jack.thompson@research.uwa.edu.au}

\author[E. Valdinoci]{Enrico Valdinoci}
\address{Enrico Valdinoci: Department of Mathematics and Statistics, The University of Western Australia, 35 Stirling Highway, Crawley, Perth, WA 6009, Australia}
\email{enrico.valdinoci@uwa.edu.au}

\subjclass[2020]{35R11, 47G20, 35N25, 35B50} 

\date{\today}

\dedicatory{}

\maketitle

\begin{abstract}
In this article, we analyze the stability of the parallel surface problem for semilinear equations driven by the fractional Laplacian. 
We prove a quantitative stability result that goes beyond that previously obtained in \cite{MR4577340}. 

Moreover, we discuss in detail several techniques and challenges in obtaining the optimal exponent in this stability result. In particular, this includes an upper bound on the exponent via an explicit computation involving a family of ellipsoids. 
We also sharply investigate a technique that was proposed in \cite{MR3836150} to obtain the optimal stability exponent in the quantitative estimate for the nonlocal Alexandrov's soap bubble theorem,
obtaining accurate estimates to be compared with a new, explicit example.
\end{abstract}

\section{Introduction and main results} \label{tH4AETfY}

\subsection{The long-standing tradition of overdetermined problems}
Overdetermined problems are a broad class of partial differential equations (\textrm{PDEs}) where `too many conditions' are imposed on the solution. Since not every region will admit a solution which satisfies all the conditions, the objective in the study of overdetermined problems is to classify those regions which do admit solutions. 

Often, overdetermined problems arise naturally in applications as a combination of a well-posed partial differential equation (\textsc{PDE}) which describes the dynamics of a given physical system, as well as an extra condition, often referred to as the overdetermined condition, which describes a property or an optimal quality you would like the solution to possess. As such, overdetermined problems have a close relationship with optimization, free boundary problems and calculus of variations, particularly shape optimization, as well as many other applications including fluid mechanics, solid mechanics, thermodynamics, electrostatics, see \cite{MR333220,MR2436831,MR3791463,MR4230553}.

The study of overdetermined problems began in the early 1970's with the celebrated paper of Serrin~\cite{MR333220}. In this influential paper, Serrin proved that, given a bounded domain \(\Omega \subset \R^n\) with~\(C^2\) boundary and \(f\in C^{0,1}_{\text{loc}}(\R)\), if there exists a positive solution \(u\in C^2(\overline \Omega )\) that satisfies the Dirichlet boundary value problem \begin{align}
\begin{PDE} 
-\Delta u &= f(u) &\text{in } \Omega, \\
u&=0 &\text{on } \partial \Omega,
\end{PDE} \label{570WVdsY}
\end{align} as well as the overdetermined condition \begin{align}
\partial_\nu u &= c \qquad  \text{on } \partial \Omega \label{gdBJQCnF}
\end{align} for some constant \(c\in \R \setminus \{0\}\), then~\(\Omega\) must be a ball. Here~\(\nu\) is the unit outward pointing normal to~\(\partial\Omega\) and~\(\partial_\nu u = \nabla u \cdot \nu\).
The proof relies on a powerful technique, now known as the \emph{method of moving planes}, a refinement of a reflection principle conceived by Alexandrov in \cite{MR0150710} to prove the so-called \emph{soap bubble theorem}, which states that the only connected closed hypersurfaces embedded in a space form with everywhere constant mean curvature are spheres.

Since the work of Alexandrov and Serrin, the analysis of overdetermined problems and the method of moving planes has seen intense research activity. Some of the literature on the method of moving planes, overdetermined problems, and symmetry in \textsc{PDE} include: \begin{itemize}
\item \emph{Alternate/incomplete overdetermined conditions:} One can consider elliptic equations such~\eqref{570WVdsY} or otherwise with an alternate overdetermined condition to~\eqref{gdBJQCnF},
see~\cite{MR1616562, MR2436831, MR3231971, MR4230553}. The parallel surface problem fits into this category, see below for references. 
\item \emph{Integral identities:} Integral identities such as the Pohozaev identity are often employed in the analysis of overdetermined problems, providing an alternative approach to the method of moving planes. Such an alternative approach was pioneered in \cite{MR333221} and then further developed in, e.g., \cite{payne1989duality, MR980297, MR2448319, cianchi2009overdetermined, MR3663321, MR3959271, MR4054869, MR4124125}. 
\item \emph{Nonlinear elliptic equations:} The original paper of Serrin \cite{MR333220} dealt with uniformly elliptic quasilinear equations including equations of mean curvature type. Other (possibly degenerate) quasilinear and fully nonlinear equations have been analyzed \cite{MR980297, MR2293958, MR2366129, MR2448319, MR2764863, MR3040677, MR3385189}. 
\item \emph{Symmetry:} The method of moving planes was famously used to prove symmetry results for solutions to semilinear \textsc{PDE} in domains with symmetry in \cite{MR544879,MR634248}. For this type of results, an alternative approach which combines Pohozaev identity and isoperimetric-type inequalities was pioneered in \cite{MR653200} and then further developed in \cite{MR1382205, MR3003296, MR4380032}. 
\item \emph{Nonlocal operators:} For overdetermined problems and symmetry for nonlocal equations, see \cite{MR3395749,MR3881478,MR3836150,MR3937999,RoleAntisym2022,MR4577340}.
\end{itemize}

\subsection{The nonlocal parallel surface problem and the
main results of this paper}
In this paper, we are concerned with the nonlocal parallel surface problem. The context of this problem is as follows: Let \(n\) be a positive integer and \(s\in (0,1)\). Suppose that \(G\) is an open bounded subset of~\(\R^n\) and 
let~\(\Omega  = G+B_R\) for some~\(R>0\), where~\(A+B\) is the Minkowski sum of sets defined by \begin{align*}
A+B = \{ a+b \text{ s.t. } a\in A, b \in B\}. 
\end{align*} Moreover, let \((-\Delta)^s\) denote the fractional Laplacian defined by \begin{align*}
(-\Delta)^s u(x) &= c_{n,s}\PV \int_{\R^n} \frac{u(x)-u(y)}{\vert x - y \vert^{n+2s}} \dd y 
\end{align*} where \(c_{n,s}\) is a positive normalization constant and \(\PV\) denotes the Cauchy principle value. Also, assume that \(f:\R \to \R\) is locally Lipschitz and satisfies \(f(0)\geqslant 0\).  Then, the nonlocal parallel surface problem asks: if there exists a function \(u\) that satisfies (in an appropriate sense) the equation \begin{align}
\begin{PDE}\label{problem00}
(-\Delta)^s u &=f(u) &\text{in } \Omega ,\\
u&=0 &\text{in } \R^n \setminus \Omega, \\
u&\geqslant 0 &\text{in } \R^n ,
\end{PDE}
\end{align} as well as the overdetermined condition \begin{align}
u = \text{constant} \qquad \text{on } \partial G, \label{k2Zr4V52}
\end{align} then is \(\Omega\) necessarily a ball? 

This question is referred to as the rigidity problem for the parallel surface problem. Furthermore, one can ask about the stability of this problem, that is, heuristically, if \(u\) satisfies~\eqref{problem00} and `almost' satisfies~\eqref{k2Zr4V52} then is \(\Omega\) `almost' a ball? This is the subject of the current article. Of course, one must be precise by what one means by `almost'---this will be made clear in the proceeding paragraphs as we describe the literature and our main results. 

The local analogue of the parallel surface problem (i.e. the case \(s=1\)) was introduced in \cite{MR2916825} as a discrete analogue of the original Serrin's problem. Moreover, it also appears in \cite{MR2629887,MR3420522} in the context of invariant isothermic surfaces of a nonlinear non-degenerate fast diffusion equation.

To see why the parallel surface problem can be
viewed as a discrete analogue of Serrin's problem, consider that~\(u\in C^2(\overline \Omega)\) satisfies~\eqref{570WVdsY} as well as~\(u= c_k>0\) on an countably infinite family of parallel surfaces~\(\Gamma_k\) that are a distance~\(1/k\) from the boundary of~\(\Omega\). Then~\(kc_k\) necessarily converges to some~\(c\) by regularity assumptions on~\(u\) and~\(u\)  satisfies~\eqref{gdBJQCnF}. Consequently, \(\Omega\) must be a ball.

To reiterate, in the parallel surface problem there is only a single parallel surface (not a family as just described); regardless, this was enough to prove in \cite{MR2916825,MR3420522} the rigidity result: if there exists a solution to~\eqref{570WVdsY} satisfying~\eqref{k2Zr4V52} then \(\Omega\) must be a ball. 

Subsequently, in \cite{MR3481178}, the stability of~\eqref{problem00} for the local problem was addressed. In that article, the authors used the shape functional \begin{align*}
\rho(\Omega) = \inf \big\{  R-r \text{ s.t. } B_r(x) \subset \Omega \subset B_R(x) \text{ for some } x\in \Omega \big\} 
\end{align*} to quantify how close \(\Omega\) is to being a ball and the semi-norm \begin{align*}
[u]_{\partial G} = \sup_{\substack{x,y\in \partial G \\ x\neq y}} \bigg \{ \frac{\vert u(x)-u(y) \vert}{\vert x - y \vert } \bigg \} 
\end{align*} to quantify how close \(u\) is to being constant on \(\partial \Omega\). They showed, under reasonable assumptions on~\(\Omega\), that \begin{align}
\rho(\Omega) \leqslant C [u]_{\partial G} . \label{dKQORoXN}
\end{align}

Moreover, for the nonlocal parallel surface problem, the rigidity problem was answered affirmatively in \cite{MR4577340} for the case \(f=1\) and \cite{RoleAntisym2022} for general \(f\). Moreover, \cite{MR4577340} also addressed the stability problem (still for the case \(f=1\)) and showed that \begin{align}
\rho(\Omega) \leqslant C [u]_{\partial G}^{\frac 1 {s+2}} . \label{svfQTEHq}
\end{align} It is interesting to observe that~\eqref{svfQTEHq} is sub-optimal in the sense that it does not recover the estimate~\eqref{dKQORoXN} when \(s=1\). This is due to the nonlocality of the fractional Laplacian which caused contributions from mass `far away' to have a significant effect on the analysis. 

This leads us to our main result. 
%
%
\begin{thm} \thlabel{oAZAv7vy}
Let \(G\) be an open bounded subset of~\( \R^n\) and \(\Omega = G+B_R\) for some \(R>0\). Furthermore, let \(\Omega\) and \(G\) have \(C^1\) boundary,
%
%
and let~\(f \in C^{0,1}_{\mathrm{loc}}(\R)\) be such that \(f(0)\geqslant 0\). Suppose that \(u\) satisfies~\eqref{problem00} in the weak sense. 

Then, \begin{align}
\rho(\Omega) &\leqslant C [u]_{\partial G}^{\frac 1 {s+2}} \label{srHxBYlS}
\end{align} with 
\begin{multline*}
C := C(n,s) \left[
\frac{ ( \max \{ 1, \diam \Omega \} )^{n+2s} \big( 1+ ( \diam \Omega )^{2s} [f]_{C^{0,1}([0,\|u\|_{L^\infty(\Omega)}])} \big) }{ R^s \big( f(0)+\|u\|_{L_s(\R^n)} \big) } 
+ (\diam \Omega)^{n-1} + \frac{ \vert \Omega \vert }{ R } \right]
\\ 
\times
\frac{(\diam \Omega)^{n+2s+3}}{R^{2s} \vert \Omega \vert }
\left(
1+R^{2s} [f]_{C^{0,1}([0,\|u\|_{L^\infty(\Omega)}])} \right) .
\end{multline*}
\end{thm}

For the precise definition of a solution satisfying~\eqref{problem00} in the weak sense, see Section~\ref{xxtYVHnv}. 

\thref{oAZAv7vy} is a direct extension of \cite{MR4577340} to the case \(f \neq 1\). Moreover, \thref{oAZAv7vy} relaxes some of the regularity assumptions on \(\Omega\) appearing in \cite{MR4577340}.
It is clear that the dependence on the volume $| \Omega |$ in the constant of Theorem \ref{oAZAv7vy}  can be
removed by means of the following bounds
$$
|B_1| R^n \le | \Omega | \le |B_1| ( \diam \Omega)^n
,
\text{ where } | B_1 | \text{ is the volume of the unit ball in } \R^n
,
$$
which hold true in light of the monotonicity of the volume with respect to inclusion.

Currently, an important open problem for the nonlocal parallel surface problem is to understand the optimality of \thref{oAZAv7vy}. Indeed, let \(\overline\beta(s)\) be the optimal exponent in~\eqref{srHxBYlS} defined as the supremum over \(\beta \in \R\) such that if \(u=u_f\) is a weak solution of~\eqref{problem00} for some \(f\) satisfying the assumptions of \thref{oAZAv7vy} then \(\rho(\Omega) \leqslant C [u]_{\partial G}^{\beta}\). In this framework, an explicit expression for~\(\overline\beta\) as a function of~\(s\) is still unknown. \thref{oAZAv7vy} implies that~\(\overline\beta(s) \geqslant \frac 1 {s+2}\) and \cite{MR2916825,MR3420522} establish that~\(\overline\beta(1)=1\),
therefore we believe it is an interesting problem to detect optimal stability exponents in the nonlocal setting, also to recover, whenever possible, the classical exponent of the local cases in the limit.

In Section~\ref{z1jzKMXt}, we explicitly construct a family of domains \(G_\varepsilon\) that are small perturbations of a ball and corresponding solutions \(u_\varepsilon\) to \eqref{problem00} when \(f=1\) which satisfy \([u_\varepsilon]_{\partial G_\varepsilon} \simeq C \rho(\Omega_\varepsilon)\) as \(\varepsilon \to 0^+\). This entails the following result: 

\begin{thm} \thlabel{O4zCF33o}
Let \(G\) be an open bounded subset of \(\R^n\) and \(\Omega = G+B_R\) for some \(R>0\). Furthermore, let \(\Omega\)
and \(G\) have \(C^1\) boundary.

Then, we have that~\(\overline\beta(s) \leqslant 1\). 
\end{thm}
 
As far as the authors are aware, these are the only known estimates for~\(\overline\beta(s)\). Furthermore, by considering the case in which~\(\Omega = \Omega_\varepsilon\) is a small perturbation of a ball and exploting
interior regularity for the fractional Laplacian, one would expect that~\(\rho(\Omega_\varepsilon) \simeq [u_\varepsilon]_{\partial G_\varepsilon} \simeq \varepsilon\) (up to constant) as~\(\varepsilon \to 0^+\) which suggests that~\(\overline\beta(s)=1\) for all~\(s\in (0,1]\).  In Section~\ref{fEsBEcuv}, we give a broad discussion on some of the challenges that
the nonlocality of the fractional Laplacian presents in obtaining this result. In particular, by way of an example via the Poisson representation formula, we show that estimates for a singular integral involving the reciprocal of the distance to the boundary function play a key role in obtaining the anticipated optimal result. This suggests, surprisingly, that fine geometric estimates for the distance function close to the boundary are required to obtain the optimal exponent. 

\subsection{Scrutiny of the stability exponent} In the recent literature,
some inventive methods have been introduced to improve
the stability results obtained via the moving plane method.
A remarkable one was put forth in~\cite{MR3836150}.
Roughly speaking, the setting considered in~\cite[Proposition~3.1(b)]{MR3836150} focused
on the critical hyperplane~$\pi_\lambda=\{x_n=\lambda\}$ for the moving
plane technique for a set~$\Omega$ which contains the ball~$B_r$ and is contained in the ball~$B_R$, looked at the symmetric difference between $\Omega$ and its reflection~$\Omega'$ and obtained a bound on the measure of the set\begin{equation}\label{pijwdled-2}
\big\{ x\in \Omega \triangle \Omega ' \text{ s.t. } \dist(x,\pi_\lambda) \leqslant \gamma \big\} \end{equation}
which was linear in both~$\gamma$ and~$R-r$.

This constituted a fundamental ingredient in~\cite{MR3836150} to achieve an optimal stability exponent.

Unfortunately, we believe that, in the very broad generality in which the result is stated in~\cite{MR3836150}, this statement may not be true, and we present an explicit counter-example.

Nevertheless, in our opinion, a weaker version of~\cite[Proposition 3.1(b)]{MR3836150} does hold true. We state and prove this new result and check its optimality against the counter-example mentioned above. This construction plays a decisive role in improving the stability exponent.

Though we refer the reader to Section~\ref{c8w8u7Hn} for full details
of this strategy, let us anticipate here some important details.

More specifically, in the forthcoming Theorem~\ref{m9q9X2hE}, under the additional assumption that the set~$\Omega$ is of class~$C^\alpha$, with~$\alpha>1$, we will bound the measure of the set in~\eqref{pijwdled-2}, up to constants,
by\begin{equation}\label{piqfk0uwefhiovwegui}
\gamma (R-r)^{1-\frac 1 \alpha }.\end{equation}
We stress that this bound formally recovers exactly the one stated in~\cite[Proposition 3.1(b)]{MR3836150} when~$\alpha=+\infty$.

However, we believe that the bound in~\eqref{piqfk0uwefhiovwegui} is optimal and cannot be improved (in particular, while the dependence in~$\gamma$ of the estimate above is linear, the dependence in~$R-r$, surprisingly, is not!). This will be shown by an explicit counter-example put forth in Theorem~\ref{example61}---while the analytical details
of this counter-example are very delicate, the foundational idea behind
it is sketched in Figure~\ref{Fig1} (roughly speaking, the example
is obtained by a very small and localized modification of a ball  to induce
a critical situation at the maximal location allowed by the regularity of the set; the construction is technically demanding since the constraint for the set of ``being between two balls'' induces two different scales, in the horizontal and in the vertical directions, which in principle could provide different contributions).

As a consequence of the bound~\eqref{piqfk0uwefhiovwegui} for the measure of the set in~\eqref{pijwdled-2}, as given in \thref{m9q9X2hE}, the stability exponent~$1/(s+2)$ of \thref{oAZAv7vy} can be improved to~$\alpha/( 1 + \alpha (s+1))$ provided that~$\Omega$ is of class~$C^\alpha$ for~$\alpha >1$: we refer to Section~\ref{subsec:new improvement exponent} and \thref{thm:improvement} for details.

\subsection{Organization of paper}

The paper is organized as follows. In Section~\ref{xxtYVHnv}, we summarize the notation and basic definitions used throughout the article. In Section~\ref{S7DGIjUf}, we give several quantitative maximum principles---in both the non-antisymmetric and antisymmetric situations---that are required in the proof of \thref{oAZAv7vy} and, in Section~\ref{sec:stabest}, we give the proof of \thref{oAZAv7vy}.

The remaining sections are broadly focused on the techniques and challenges in obtaining the optimal stability exponent.
In Section~\ref{fEsBEcuv}, we discuss the surprising role that
fine geometric estimates for the distance function close to the boundary play in the attainment of the optimal exponent. In Section~\ref{c8w8u7Hn}, we accurately discuss the possibility of obtaining the optimal exponent
and comment about some criticalities in the existing literature,
and, in Section~\ref{z1jzKMXt}, we construct an explicit family of solutions which implies \thref{O4zCF33o}.

\section*{Acknowledgements} 

All the authors are members of AustMS.
GP is supported by
the Australian Research Council (ARC) Discovery Early Career Researcher Award (DECRA) DE230100954 ``Partial Differential Equations: geometric aspects and applications'' and is member of the Gruppo Nazionale Analisi Matematica Probabilit\`a e Applicazioni (GNAMPA) of the Istituto Nazionale di Alta Matematica (INdAM).
JT is supported by an Australian Government Research Training Program Scholarship.
EV is supported by the Australian Laureate Fellowship FL190100081
``Minimal surfaces, free boundaries and partial differential equations''.

\section{Preliminaries and notation} \label{xxtYVHnv}

In this section we fix the notation that
we will use throughout the article and give some relevant definitions. Let~\(n\geqslant 1\) and~\(s\in (0,1)\). The fractional Sobolev space~\(H^s(\R^n)\) is defined as \begin{align*}
H^s(\R^n)  = \big\{ u \in L^2(\R^n) \text{ such that } [u]_{H^s(\R^n)} <+\infty \big\} 
\end{align*} where \([\cdot ]_{H^s(\R^n)}\) is the Gagliardo semi-norm given by \begin{align*}
[u ]_{H^s(\R^n)} = c_{n,s}\int_{\R^n} \int_{\R^n} \frac{\vert u(x) - u(y) \vert^2}{\vert x-y\vert^{n+2s}} \dd y \dd x 
\end{align*} where \(c_{n,s}\) is the same constant appearing in the definition of the fractional Laplacian. Functions in~\(H^s(\R^n)\) that are equal almost everywhere are identified. Moreover, the bilinear form associated with~\([\cdot ]_{H^s(\R^n)}\) is denoted by~\(\mathcal E : H^s(\R^n) \times H^s(\R^n) \to \R\) and is given by \begin{align*}
\mathcal E(u,v) = \frac{c_{n,s}}{2} \int_{\R^n} \int_{\R^n} \frac{( u(x) - u(y) )(v(x)-y(y) }{\vert x-y\vert^{n+2s}} \dd y \dd x .
\end{align*} We also define the following weighted \(L^1\) norm via \begin{align*}
\| u\|_{L_s(\R^n)}  = \int_{\R^n} \frac{\vert u(x) \vert}{1+\vert x \vert^{n+2s} } \dd x
\end{align*} and the space \(L_s(\R^n)\) as \begin{align*}
L_s(\R^n)  = \big\{ L^1_{\mathrm{loc}} (\R^n) \text{ such that } \| u\|_{L_s(\R^n)} <+\infty \big\}. 
\end{align*}

Now, suppose that \(\Omega\) is an open, bounded subset of \(\R^n\) and define the space \(\mathcal H^s_0(\Omega)\) by \begin{align*}
\mathcal H^s_0(\Omega) &= \big\{ u \in H^s(\R^n) \text{ such that } u =0 \text{ in } \R^n \setminus \Omega \big\}. 
\end{align*} Let \(c:\Omega \to \R\) be a measurable function such that \(c\in L^\infty(\Omega)\) and \(g\in L^2(\Omega)\). We say that
a function~\(u\in H^s(\R^n) \cap L_s(\R^n) \) satisfies~\((-\Delta)^s u + c  u \geqslant g\) (respectively, \(\leqslant\)) in the \emph{weak sense} if \begin{align*}
\mathcal E(u,v) + \int_\Omega c uv \dd x \geqslant \int_\Omega g v \dd x  \qquad (\text{respectively, } \leqslant )
\end{align*} for all \(v\in \mathcal H^s_0(\Omega)\) with \(v\geqslant 0\). In this case, we also say that~\(u\) is a supersolution (respectively,
subsolution) of~\((-\Delta)^s u + c  u = g\) in the weak sense. Moreover, we say a function~\(u\in H^s(\R^n) \cap L_s(\R^n) \) satisfies~\((-\Delta)^s u + c  u = g\) in the \emph{weak sense} if \begin{align*}
\mathcal E(u,v) + \int_\Omega c uv \dd x = \int_\Omega g v \dd x
\end{align*} for all \(v\in \mathcal H^s_0(\Omega)\). This is equivalent to~\(u\) being both a weak supersolution and a weak subsolution. Note that we are assuming \emph{a priori} that weak solutions are in~\(L_s(\R^n)\).  

\medskip

Next, we describe some notation regarding antisymmetric functions. This is closely related to the notation of the method of moving planes; however, we will defer our explanation of the method of moving planes until Section~\ref{sec:stabest} in the interests of simplicity. 

A function \(v: \R^n \to \R\) is said to be \emph{antisymmetric} with respect to a plane \(T\) if \begin{align*}
v(Q_T(x)) = -v(x) \qquad \text{for all } x\in \R^n
\end{align*} where \(Q_T : \R^n \to \R^n\) is the function that reflects~\(x\) across~\(T\). Often it will suffice to consider the case~\(T=\{x_1=0\}\), in which case~\(Q_T(x) = x-2x_1e_1\). 

For simplicity, we will refer to~\(v\) as \emph{antisymmetric} if it is antisymmetric with respect to the plane~\(\{x_1=0\}\). Moreover, we sometimes write~\(x'\) to denote~\(Q_T(x)\) when it is clear from context what~\(T\) is.

\medskip

Finally, some other notation we will employ through the article is:
given a set~\(A\subset \R^n\), the function~\(\chi_A : \R^n \to \R \) denotes the characteristic function of~\(A\), given by \begin{align*}
\chi_A(x) &= \begin{cases}
1, &\text{if } x\in A ,\\
0, &\text{if } x\not \in A,
\end{cases}
\end{align*} and \(\delta_A : \R^n \to [0,+\infty]\) denotes the distance function to \(A\), given by \begin{align*}
\delta_A(x) = \inf_{y \in A} \vert x-y\vert. 
\end{align*}

We will denote by~$\R^n_+:=\{x=(x_1,\dots,x_n)\in\R^n\;{\mbox{s.t.}}\; x_1>0\}$ and, given~$A\subseteq\R^n$, we will denote by~$A^+:=
A\cap\R^n_+$.

Moreover, if \(A\) is open and bounded with sufficiently regular boundary (we only use the case in which~\(A\) has smooth boundary), we will denote by \(\psi_A\) the (unique) function in \(C^\infty(A) \cap C^s(\R^n)\) such that \(\psi_A\) satisfies \begin{align} \label{bNuXHq1g}
\begin{PDE}
(-\Delta)^s \psi_A &=1 &\text{in } A ,\\
\psi_A&=0 &\text{in } \R^n \setminus A. 
\end{PDE}
\end{align} Also, we will use \(\lambda_1(A)\) to denote the first Dirichlet eigenvalue of the fractional Laplacian. 

\section{Quantitative maximum principles up to the boundary} \label{S7DGIjUf}

The basic idea of the proof of \thref{oAZAv7vy} is to apply the method of moving planes, as in the proof of the analogous rigidity result \cite[Theorem 1.4]{RoleAntisym2022}, but replace `qualitative' maximum principles, such as the strong maximum principles, with `quantitative' maximum principles, such as the Harnack inequality.

The purpose of this section is to prove several such quantitative maximum principles both in the non-antisymmetric and the antisymmetric setting. In particular, we require that these maximum principles hold up to the boundary of regions with very little boundary regularity. Moreover, we are careful to keep track of precisely how constants depend on relevant quantities since we believe that this may be useful in future analyses of the nonlocal parallel surface problem. 

The section is split into two parts: maximum principles for equations without any antisymmetry assumptions and maximum principles for equations with antisymmetry assumptions. 

\subsection{Equations without antisymmetry}

In this subsection, we will give several maximum principles for linear equations with zero-th order terms without any antisymmetry assumptions. This will culminate in a quantitative analogue of the Hopf lemma for non-negative supersolutions of general semilinear equations (see \thref{zAPw0npM} below). 

Our first result is as follows: 

\begin{prop} \thlabel{YCl8lL0I}
Let \(\Omega\) be an bounded open subset of \(\R^n\), \(c\in L^\infty(\Omega)\), and \(g\in L^2(\Omega)\) with \(g\geqslant 0\). Suppose that \(u\) satisfies \begin{align*}
\begin{PDE}
(-\Delta)^s u+cu &\geqslant g &\text{in } \Omega, \\
u&\geqslant 0 &\text{in } \R^n
\end{PDE}
\end{align*} in the weak sense. 

Then,
\begin{align*}
u(x) &\geqslant  C\Big(\essinf_\Omega g + \| u\|_{L_s(\R^n)}\Big) \delta_{\partial \Omega}^{2s}(x) \qquad \text{for a.e. } x\in \Omega
\end{align*} with \(C:=C(n,s)(\max\{1, \diam \Omega \} )^{-n-2s}\big(1+(\diam \Omega)^{2s}\|c^+\|_{L^\infty(\Omega)}\big)^{-1}\).
\end{prop}

The proof of \thref{YCl8lL0I} essentially follows by, at each point \(x\in \Omega\), `touching' the solution~\(u\) from below by~\(\psi_B\) (recall the notation in~\eqref{bNuXHq1g}) where~\(B\) is the largest ball contained in~\(\Omega\) and centred at~\(x\). This is similar to the proof of \cite[Theorem 2.2]{MR4023466} where a nonlinear interior analogue of \thref{YCl8lL0I} was proven. To prove \thref{YCl8lL0I}, however, some care must be taken since the touching point may occur on the boundary of \(\Omega\) (unlike in
the interior result of~\cite[Theorem~2.2]{MR4023466}) where the PDE does not necessarily hold. Moreover, \thref{YCl8lL0I} is stated in the context of weak solutions where pointwise techniques no longer make sense, so a simple mollification argument needs to be made. We require the following lemma.

\begin{lem} \thlabel{Voueel7o}
Let \(\rho>0 \), \(c\in L^\infty(B_\rho)\), and \(g\in C^\infty_0(\R^n)\) be such that \(g\geqslant 0\) in \(B_\rho\). Suppose that~\(u\in C^\infty_0(\R^n)\) satisfies \begin{align*}
\begin{PDE}
(-\Delta)^s u +cu &\geqslant g &\text{in } B_\rho, \\
u&\geqslant 0 &\text{in } \R^n .
\end{PDE}
\end{align*}

Then, \begin{align*}
u(x) \geqslant C\Big(\inf_{B_\rho} g + \| u\|_{L_s(\R^n)}\Big)\psi_{B_\rho}(x) \qquad \text{for all } x\in B_\rho
\end{align*} with \(C:=C(n,s)(\max\{1, \rho\} )^{-n-2s}\big(1+\rho^{2s}\|c^+\|_{L^\infty(B_\rho)}\big)^{-1}\).
\end{lem}

We observe that Lemma~\ref{Voueel7o} can be seen as a quantitative version of a strong maximum principle (which can be proved directly):
in particular, it entails that if~$u(x_0)=0$ for some $x_0\in\Omega$, then~$u$ vanishes identically.

\begin{proof}[Proof of Lemma~\ref{Voueel7o}]
We will begin by proving the case \(\rho=1\). Fix \(\varepsilon>0\) small and let \(\tau \geqslant 0\) be the largest value such that \(u\geqslant \tau \psi_{B_{1-\varepsilon}}\) in \(\R^n\). Now, by the definition of \(\tau\), there exists \(x_0\in B_{1-\varepsilon}\) such that \(u(x_0)=\tau \psi_{B_{1-\varepsilon}}(x_0)\). On one hand, \begin{align*}
(-\Delta)^s(u-\tau \psi_{B_{1-\varepsilon}}) (x_0)+c(x_0) (u-\tau \psi_{B_{1-\varepsilon}}) (x_0) &\geqslant \inf_{B_1} g  - \tau \big (  1 +c(x_0) \psi_{B_{1-\varepsilon}} (x_0) \big ) \\
&\geqslant \inf_{B_1} g -C(n,s)\big(1+ \|c^+\|_{L^\infty(B_1)}\big) \tau . 
\end{align*} On the other hand, since \(\vert x_0 - y \vert^{-n-2s} \geqslant C(n,s)(1+\vert y \vert^{n+2s} )^{-1}\), we have that  \begin{align*}
(-\Delta)^s(u-\tau\psi_{B_{1-\varepsilon}}) (x_0)+c(x_0) (u-\tau \psi_{B_{1-\varepsilon}}) (x_0) &= - \int_{\R^n} \frac{(u-\tau \psi_{B_{1-\varepsilon}}) (y)}{\vert x_0 - y \vert^{n+2s}} \dd y \\
&\leqslant  -C(n,s) \int_{\R^n} \frac{(u-\tau \psi_{B_{1-\varepsilon}}) (y)}{1+\vert y \vert^{n+2s}} \dd y \\
&\leqslant - C(n,s)  \big( \|u\|_{L_s(\R^n)} - \tau \big) . 
\end{align*} Rearranging gives \(\tau \geqslant C(n,s) \big(1+\|c^+\|_{L^\infty(B_1)}\big)^{-1}  \big(\inf_{B_1} g + \| u\|_{L_s(\R^n)}\big)\), which implies that \begin{align*}
u(x) &\geqslant  C(n,s) \big(1+\|c^+\|_{L^\infty(B_1)}\big)^{-1} \Big(\inf_{B_1} g + \|u\|_{L_s(\R^n)}\Big) \psi_{B_{1-\varepsilon}}.
\end{align*} Sending \(\varepsilon \to 0^+\) gives the result when \(\rho=1\). 

Now, for the case that \(\rho\) is not necessarily 1, let \(u_\rho(x):= u(\rho x)\). Then \begin{align*}
(-\Delta)^su_\rho + \rho^{2s} c_\rho &\geqslant \rho^{2s} g_\rho \qquad \text{in }B_1
\end{align*} where \(c_\rho(x):=c(\rho x)\) and \(g_\rho(x):=g(\rho x)\), so, for all \(x\in B_\rho\), we have that \begin{align*}
u(x) &= u_\rho(x/\rho) \\
&\geqslant C(n,s) \big(1+\rho^{2s}\|c^+_\rho\|_{L^\infty(B_1)}\big)^{-1} \Big(\rho^{2s}\inf_{B_1} g_\rho + \|u_\rho\|_{L_s(\R^n)}\Big) \psi_{B_1}(x/\rho) \\
&= C(n,s) \rho^{-2s}\big(1+\rho^{2s}\|c^+\|_{L^\infty(B_\rho)}\big)^{-1} \Big(\rho^{2s}\inf_{B_\rho} g + \|u_\rho\|_{L_s(\R^n)}\Big) \psi_{B_\rho}(x).
\end{align*} Finally, \begin{align*}
\|u_\rho\|_{L_s(\R^n)} &= \int_{\R^n} \frac{\vert u(\rho x) \vert }{1+\vert x \vert ^{n+2s} } \dd x = \rho^{2s}\int_{\R^n} \frac{ \vert u(x) \vert }{\rho^{n+2s}+\vert x \vert ^{n+2s} } \dd x \geqslant \rho^{2s}(\max\{1, \rho\} )^{-n-2s} \| u\|_{L_s(\R^n)}
\end{align*} which completes the proof. 
\end{proof}

We can now give the proof of \thref{YCl8lL0I}. 

\begin{proof}[Proof of \thref{YCl8lL0I}]
If \(u,g\in C^\infty_0(\R^n)\) then the proof follows immediately from \thref{Voueel7o}. Indeed, let \(x\in \Omega\) be arbitrary. Then, applying \thref{Voueel7o} in the ball \(B_\rho(x)\) with \(\rho := \delta_{\partial \Omega}(x)\), we obtain that \begin{align*}
u(y) &\geqslant C  \Big(\essinf_{\Omega} g + \| u\|_{L_s(\R^n)}\Big)\psi_{B_\rho}(y) \qquad \text{for all } y \in B_\rho(x) 
\end{align*} with \(C=C(n,s)(\max\{1, \diam \Omega \} )^{-n-2s}\big(1+(\diam \Omega)^{2s}\|c^+\|_{L^\infty(\Omega)}\big)^{-1}\). Substituting \(y=x\) completes the proof.

In the general case that \(u\in H^s(\R^n)\), if \(u_\varepsilon\) and \(g_\varepsilon\) are mollifications of \(u\) and \(g\) respectively then \((-\Delta)^s u_\varepsilon +\|c^+\|_{L^\infty(\Omega) } u_\varepsilon \geqslant g_\varepsilon\) in \(\Omega_\varepsilon := \Omega \cap \{ \delta_{\partial \Omega} >\varepsilon\}\). Then, applying the result for smooth compactly supported functions, we have that \(u_\varepsilon \geqslant C \delta^{2s}_{\partial\Omega_\varepsilon}\) with \(C\) as in the statement of Proposition~\ref{YCl8lL0I} (and, in particular, uniformly bounded in \(\varepsilon\)). Moreover, given \(x\in \Omega_\varepsilon\), if \(y\in \partial \Omega_\varepsilon\) is the closest point to \(x\) on \( \partial \Omega_\varepsilon\) and \(z\in \partial \Omega\) is the closest point to \(y\) on \(\partial \Omega\) then \begin{align*}
\delta_{\partial \Omega}(x) \leqslant \vert x - z\vert \leqslant \vert x-y \vert + \vert y-z\vert \leqslant \delta_{\partial\Omega_\varepsilon}(x)+\varepsilon,
\end{align*}so \(u_\varepsilon \geqslant C (\delta_{\partial\Omega}-\varepsilon)^{2s}\) in \(\Omega_\varepsilon\). Then \(u_\varepsilon \to u\) a.e., so sending \(\varepsilon \to 0^+\) gives the result. 
\end{proof}

Next, we obtain a refined version of \thref{YCl8lL0I} when \(\Omega\) satisfies the uniform interior ball condition. 
As customary, we say that a bounded domain $\Omega \subset \R^n$ satisfies the {\it uniform interior ball condition
with radius} $r_\Omega>0$ if for every point $x_0 \in \partial \Omega$ there exists a ball $B \subset \Omega$ of radius
$r_\Omega$ such that its closure intersects $\partial \Omega$ only at $x_0$.

\begin{prop} \thlabel{1kapKUbs}
Let \(\Omega\) be an open bounded subset of \(\R^n\) with $C^1$ boundary and satisfying the uniform interior ball condition with radius \(r_\Omega>0\), \(c\in L^\infty(\Omega)\), and \(g\in L^2(\Omega)\) with \(g\geqslant 0\). Suppose that \(u\) satisfies \begin{align*}
\begin{PDE}
(-\Delta)^s u+cu &\geqslant g &\text{in } \Omega, \\
u&\geqslant 0 &\text{in } \R^n
\end{PDE}
\end{align*} in the weak sense. 

Then,  \begin{align*}
u(x) &\geqslant  C\Big(\essinf_\Omega g + \| u\|_{L_s(\R^n)}\Big) \delta_{\partial \Omega}^{s}(x) \qquad \text{for a.e. } x\in \Omega
\end{align*} with \(C:=C(n,s)r_\Omega^s (\max\{1, \diam \Omega \})^{-n-2s}\big(1+(\diam \Omega)^{2s}\|c^+\|_{L^\infty(\Omega)}\big)^{-1}\).
\end{prop}

\begin{proof}
By an analogous argument to the one in the proof of \thref{YCl8lL0I}, we may assume that \(u,g\in C^\infty_0(\R^n)\). Let \(x\in \Omega\) be arbitrary. If \(\delta_{\partial \Omega}(x) \geqslant r_\Omega\) then we are done by \thref{YCl8lL0I}, so let \(\delta_{\partial \Omega}(x)<r_\Omega\). If \(\bar x \) denotes the closest point to \(x\) on \(\partial \Omega\) then there exists a ball~\(B=B_{r_\Omega}(x_0)\) such that~\(x\in B\), \(B \subset \Omega\), and~\(B\) touches~\(\partial \Omega\) at \(\bar x\). Then, applying \thref{Voueel7o} in \(B\), we have that \(u(y) \geqslant C\big(\essinf_\Omega g + \| u\|_{L_s(\R^n)}\big)\psi_{B_{r_\Omega}}(y)\)
with \(C=C(n,s)(\max\{1, \diam \Omega \} )^{-n-2s}\big(1+(\diam \Omega)^{2s}\|c^+\|_{L^\infty(\Omega)}\big)^{-1}\). Substituting \(y=x\), we obtain \begin{align*}
u(x) &\geqslant C\Big(\essinf_\Omega g + \| u\|_{L_s(\R^n)}\Big)(r_\Omega^2-\vert x-x_0\vert^2)^s \\
&= C\Big(\essinf_\Omega g + \| u\|_{L_s(\R^n)}\Big)(r_\Omega+\vert x-x_0\vert)^s(r_\Omega-\vert x-x_0\vert)^s \\
&\geqslant Cr_\Omega^s\Big(\essinf_\Omega g + \| u\|_{L_s(\R^n)}\Big)(r_\Omega-\vert x-x_0\vert)^s.
\end{align*} Observing that \(r_\Omega-\vert x-x_0\vert = \delta_{\partial \Omega}(x)\) completes the proof. 
\end{proof}

From \thref{1kapKUbs}, we immediately obtain the following corollary.

\begin{cor} \thlabel{zAPw0npM}
Let \(\Omega\) be an open bounded subset of \(\R^n\) with $C^1$ boundary and satisfying the uniform interior ball condition with radius \(r_\Omega>0\). Let \(f \in C^{0,1}_{\mathrm{loc}}(\overline \Omega \times \mathbb R)\) satisfy~\(f_0:=\inf_{x \in \Omega} f(x,0) \geqslant 0\). Suppose that~\(u\) satisfies \begin{align*}
\begin{PDE}
(-\Delta)^s u &\geqslant f(x,u) &\text{in } \Omega ,\\
u&\geqslant 0 &\text{in } \R^n 
\end{PDE}
\end{align*} in the weak sense. 

Then, \begin{align*}
u(x) &\geqslant  C_\ast\big(f_0 + \| u\|_{L_s(\R^n)}\big) \delta_{\partial \Omega}^{s}(x) \qquad \text{for a.e. } x\in \Omega
\end{align*} with \(C_\ast:=C(n,s)r_\Omega^s (\max\{1, \diam \Omega \} )^{-n-2s}\big(1+(\diam \Omega)^{2s} [f]_{C^{0,1}(\overline \Omega \times [0,\|u\|_{L^\infty(\Omega)}])} \big)^{-1}\).
\end{cor}

\begin{proof}
Since \(f=f(x,z)\) is locally lipschitz, \(\partial_z f\) exists a.e., so we may define \begin{align*}
c(x):= -\int_0^1 \partial_z f (x, t u(x)) \dd t .
\end{align*} Observe that \(\| c\|_{L^\infty(\Omega)} \leqslant [f]_{C^{0,1}(\overline \Omega \times [0,\| u\|_{L^\infty(\Omega)}])}\). Then, \((-\Delta)^s u+c u \geqslant f(\cdot, 0)\) in \(\Omega\) in the weak sense, so the result follows by \thref{1kapKUbs}. 
\end{proof}

\begin{remark}
One could also obtain an analogous result to \thref{zAPw0npM}
without assuming that $\Omega$ has $C^1$ boundary and satisfies the uniform interior ball condition; this can be achieved by applying \thref{YCl8lL0I} instead of \thref{1kapKUbs}. 
\end{remark}

\subsection{Equations with antisymmetry}

The purpose of this subsection is to prove the following proposition. 

\begin{prop} \thlabel{TV1cTSyn}
Let \(H\subset \R^n\) be a halfspace, \(U \) be an open subset of \(H\), and~\(a\in H\) and~\(\rho>0\) such that~\(B_\rho(a)\cap H \subset U\). Moreover, let \(c\in L^\infty(U)\) be such that \begin{equation}\label{ed0395v04ncrhrhfwejfvwejhfvqjhkfvqejhk}
\|c^+\|_{L^\infty(U)} < \lambda_1(B_\rho(a)\cap H).\end{equation} 
Suppose that \(v\) is antisymmetric with respect to \(\partial H\) and satisfies \begin{align*}
\begin{PDE}
(-\Delta)^s v +c v &\geqslant 0 &\text{in } U  ,\\
v&\geqslant 0 & \text{in } H 
\end{PDE}
\end{align*} in the weak sense. 

If \(K\subset H\) is a non-empty open set that is disjoint from \(B_\rho(a)\) and 
$$ \inf_{{x\in K}\atop{y\in B_\rho(a)\cap H}} \vert Q_{\partial H}(x)-y\vert^{-1} \geqslant M\geqslant 0 $$ then \begin{align}
v(x) &\geqslant C \| \delta_{\partial H} v\|_{L^1(K)} \delta_{\partial H}(x) \qquad \text{for a.e. } x\in B_{\rho/2}(a)\cap H \label{tQIcIMb5}
\end{align}  with \begin{align*}
C := C(n,s) \rho^{2s} M^{n+2s+2} \Big(1+\rho^{2s}\|c^+\|_{L^\infty(U)} \Big)^{-1}.
\end{align*}
\end{prop}

\thref{TV1cTSyn} can be viewed as a quantitative version of Proposition 3.3 in \cite{MR3395749}. A similar result was also obtained in \cite{MR4577340} in the case \(c=0\). One advantage that
\thref{TV1cTSyn} has over both of the results of \cite{MR3395749} and \cite{MR4577340} is that it allows the ball~\(B_\rho(a)\) to
go right up to and, indeed, overlap the plane of symmetry \(\partial H\). To allow for this possibility required the construction of an antisymmetric barrier given in \thref{gmjWA2gy} below. 

However, a disadvantage of \thref{TV1cTSyn} is that it is not a boundary estimate, in the sense that~\eqref{tQIcIMb5} holds in \(B_{\rho/2}(a) \cap H\) and not \(B_\rho(a) \cap H\), so it does not give any information up to \(\partial U\setminus \partial H\). This is also a by-product of the barrier given in \thref{gmjWA2gy}. In theory, one should be able to fix this issue by adjusting the barrier to behave like distance to the power \(s\) close to the boundary of its support, but this was unnecessary for the purposes of our results. 

To prove \thref{TV1cTSyn}, we require two lemmata.

\begin{lem} \thlabel{xwrqefNE}
Let \(v \in C^\infty_0(\R^n)\) be antisymmetric with respect to \(\{x_1=0\}\). 

Then, \((-\Delta)^sv\) is a smooth function in \(\R^n\) that is antisymmetric with respect to \(\{x_1=0\}\). Furthermore, if \(w(x):=v(x)/x_1\) then \begin{align}
\vert (-\Delta)^s v(x)\vert &\leqslant C \Big( \|w \|_{L^\infty(\R^n)}+\|x_1^{-1} \partial_{1} w \|_{L^\infty(\R^n)}+\|D^2w \|_{L^\infty(\R^n)} \Big)x_1 \label{v5SMVDf3}
\end{align} for all \(x\in \R^n_+\). The constant \(C\) depends only on \(n\) and \(s\). 
\end{lem}

\begin{proof}
The proof that \((-\Delta)^sv\) is a smooth function and antisymmetric follows from standard properties of the fractional Laplacian, so we will omit it and focus on the proof of~\eqref{v5SMVDf3}. For all \(\tilde x \in \R^{n+2}\), define \(\tilde w (\tilde x ) := w \big(\sqrt{\tilde x_1^2+\tilde x_2^2+\tilde x_3^2 }, \tilde x_4, \dots , \tilde x_{n+2}\big)\). By Bochner's relation, we have that \begin{align}
(-\Delta)^sv(x) &= x_1 (-\Delta)^s_{\R^{n+2}} \tilde w (x_1,0,0,x_2, \dots ,x_n) \label{0I9upPuf}
\end{align} where \((-\Delta)^s_{\R^{n+2}} \) is the fractional Laplacian in \(\R^{n+2}\) (and \((-\Delta)^s \) still refers to the fractional Laplacian in \(\R^n\)). For more details regarding Bochner's relation and a proof of~\eqref{0I9upPuf}, we refer the interested reader to the upcoming note~\cite{MR4411363}. 

Thus, applying standard estimates for the fractional Laplacian, we have that, \begin{align*}
\vert (-\Delta)^s v(x)\vert &\leqslant C\Big( \| D^2 \tilde w \|_{L^\infty(\R^{n+2})} + \|  \tilde w \|_{L^\infty(\R^{n+2})}  \Big)x_1 
\end{align*}for all \(x\in \R^n_+\) and \(C>0\) depending on~$n$ and~$s$. Clearly, \(\|  \tilde w \|_{L^\infty(\R^{n+2})} \leqslant \| w \|_{L^\infty(\R^n)}\). Moreover, by a direct computation, one can check that
\begin{eqnarray*}
\partial_{ij} \tilde w( \tilde x) &=&\partial_{11}w\Big(\sqrt{\tilde x_1^2+\tilde x_2^2+\tilde x_3^2 }, \tilde x_4, \dots , \tilde x_{n+2}\Big)
\frac{\tilde x_i\tilde x_j}{\tilde x_1^2+\tilde x_2^2+\tilde x_3^2}\\&&+
\partial_1 w\Big(\sqrt{\tilde x_1^2+\tilde x_2^2+\tilde x_3^2 }, \tilde x_4, \dots , \tilde x_{n+2}\Big)\left[\frac{\delta_{ij}}{\sqrt{\tilde x_1^2+\tilde x_2^2+\tilde x_3^2 }}-\frac{\tilde x_i\tilde x_j}{(\tilde x_1^2+\tilde x_2^2+\tilde x_3^2 )^{3/2}}
\right] \\
&&\qquad{\mbox{ if }} i,j\in\{1,2,3\},\\
\partial_{ij} \tilde w( \tilde x) &=&\partial_{1j-2}w\Big(\sqrt{\tilde x_1^2+\tilde x_2^2+\tilde x_3^2 }, \tilde x_4, \dots , \tilde x_{n+2}\Big)
\frac{\tilde x_i}{\sqrt{\tilde x_1^2+\tilde x_2^2+\tilde x_3^2 }}\\
&&\qquad {\mbox{ if }} i\in\{1,2,3\} \;{\mbox{ and }}\; j\in\{4, \dots, n+2\},\\
\partial_{ij} \tilde w( \tilde x) &=&\partial_{i-2j-2}w\Big(\sqrt{\tilde x_1^2+\tilde x_2^2+\tilde x_3^2 }, \tilde x_4, \dots , \tilde x_{n+2}\Big)
\qquad {\mbox{ if }} i,j\in\{4,\dots, n+2\},
\end{eqnarray*}
and therefore
\begin{align*}
\| D^2 \tilde w\|_{L^\infty(\R^{n+2})} &\leqslant C \Big(\|x_1^{-1} \partial_1 w \|_{L^\infty(\R^n)}+\|D^2w \|_{L^\infty(\R^n)} \Big)
\end{align*} for some universal constant \(C>0\), which implies the desired result. 
\end{proof}

We now construct the barrier that will be essential to allow \(B_\rho(a)\) in \thref{TV1cTSyn} to come up to \(\partial H\). 

\begin{lem} \thlabel{gmjWA2gy} Let \(a \in \overline{\R^n_+}\) and \(\rho>0\). There exists a function \(\varphi \in C^\infty_0(B_\rho(a)\cup B_\rho(a'))\) such that \(\varphi\) is antisymmetric with respect to \(\partial \R^n_+\), \(\rho^{2s}\chi_{B_{\rho/2}(a)}(x)x_1 \leqslant \varphi(x) \leqslant C\rho^{2s} x_1\) in \(\R^n_+\), and \begin{equation}\label{formula}
\vert (-\Delta)^s \varphi(x) \vert \leqslant C  x_1 \qquad \text{in } B_\rho^+(a).
\end{equation} The constant \(C\) depends only on \(n\) and \(s\). 
\end{lem}

Recall that~$ B_\rho^+(a)=B_\rho(a)\cap \R^n_+$, and notice that~$
B_\rho^+(a)$ in Lemma~\ref{gmjWA2gy} coincides with~$B_\rho(a)$
if~$\rho\in(0,a_1)$ with~$a=(a_1,\dots,a_n)$.

\begin{proof}[Proof of Lemma~\ref{gmjWA2gy}]
Let \(\eta \in C^\infty_0(B_1)\) be a radial function such that~\(\eta =1\) in~\(B_{1/2}\) and~\(0\leqslant \eta \leqslant 1\) in~\(\R^n\). Define \begin{align*}
\varphi(x) &:= \rho^{2s} x_1 \bigg ( \eta \bigg ( \frac{x-a} \rho \bigg ) + \eta \bigg ( \frac{x-a'} \rho \bigg ) \bigg ) .
\end{align*} Then \(\varphi \in C^\infty_0(B_\rho(a)\cup B_\rho(a'))\), it is antisymmetric, and satisfies \(\rho^{2s}\chi_{B_{\rho/2}(a)}(x)x_1\leqslant \varphi \leqslant 2 \rho^{2s} x_1\) in~\(\R^n_+\). 

Moreover, let 
$$\bar \varphi(x):= x_1 \left( \eta \left(x - \frac{a}\rho\right) + \eta \left(x-
\frac{a'}\rho\right) \right) $$ so that \(\varphi(x) = \rho^{2s+1}\bar \varphi(x/\rho)\). From \thref{xwrqefNE}, it follows that \( \vert (-\Delta)^s \bar \varphi \vert \leqslant C x_1\) for some \(C>0\) depending on~$n$ and~$s$, which implies that \begin{align*}
\vert (-\Delta)^s \varphi(x) \vert &= \rho \vert (-\Delta)^s \bar \varphi(x/\rho ) \vert \leqslant C x_1. \qedhere
\end{align*}
\end{proof}

Now, we will give the proof of \thref{TV1cTSyn}.

\begin{proof}[Proof of \thref{TV1cTSyn}]
Without loss of generality, we may assume that \(H = \R^n_+\). 
We also denote by~$B=B_\rho(a)$.
Recall that, given \(A\subset \R^n\), we use the notation \(A^+ = A \cap \R^n_+\). Let \(\tau \geqslant 0\) be a constant to be chosen later and \(w := \tau\varphi +  (\chi_K+\chi_{K'}) v\) where \(\varphi\) is as in \thref{gmjWA2gy} and \(K':=Q_{\partial \R^n_+}(K)\). Furthermore, let~\(\xi \in \mathcal H^s_0(B^+)\) with~\(\xi \geqslant 0\). By formula~\eqref{formula} in \thref{gmjWA2gy}, we have that \begin{align*}
\mathcal E (w, \xi )& = \tau \mathcal E (\varphi, \xi )+ \mathcal E (\chi_K v , \xi )+ \mathcal E (\chi_{K'}v , \xi ) \\
&\leqslant C  \tau \int_{B^+} x_1 \xi(x) \dd x -c_{n,s}  \int_{B^+ } \int_{K} \frac{\xi(x) v(y)}{\vert x-y\vert^{n+2s}} \dd y \dd x -c_{n,s}  \int_{B^+ } \int_{K'} \frac{\xi(x) v(y)}{\vert x-y\vert^{n+2s}} \dd y \dd x \\
&= \int_{B^+ } \bigg [ C\tau x_1 - c_{n,s}  \int_{K}  \bigg ( \frac 1{\vert x-y\vert^{n+2s}} - \frac 1 {\vert x'- y \vert^{n+2s}} \bigg ) v(y) \dd y \bigg ] \xi(x) \dd x 
\end{align*} where \(x':= Q_{\partial H}(x)\).

Since, for all \(x\in B^+\) and \(y\in K\), we have that \begin{align*}
\frac 1{\vert x-y\vert^{n+2s}} - \frac 1 {\vert x'- y \vert^{n+2s}}  &= \frac{n+2s}2 \int_{\vert x- y \vert^2}^{\vert x'- y \vert^2} t^{-\frac{n+2s+2}2} \dd t \\
&\geqslant \frac{n+2s}2  \;\frac{
\vert x' - y \vert^2 - \vert x- y \vert^2 }{ \vert x'- y \vert^{n+2s+2}} \\
&= 2(n+2s)  \frac{x_1y_1}{\vert x ' -y \vert^{n+2s+2} } \\
&\geqslant 2(n+2s) M^{n+2s+2} x_1y_1,
\end{align*} it follows that \begin{align*}
\mathcal E(w,\xi) &\leqslant C \Big ( \tau- \tilde C M^{n+2s+2} \| y_1 v \|_{L^1(K)} \Big)\int_{B^+ }   x_1 \xi(x) \dd x 
\end{align*} with \(C\) and \(\tilde C\) depending only on \(n\) and \(s\).

Hence, using that \(w = \tau \varphi \leqslant C \rho^{2s} \tau x_1\) in \(B^+\) by \thref{gmjWA2gy}, we have that \begin{align*}
\mathcal E(w,\xi) + \int_{B^+} c(x) w(x) \xi(x) \dd x  &\leqslant C \Big[ \tau\big(1  + \rho^{2s}\|c^+\|_{L^\infty(U)}\big)  - \tilde C M^{n+2s+2}  \| y_1 v \|_{L^1(K)} \Big]\int_{B^+ }   x_1 \xi(x) \dd x 
\end{align*} (after possibly relabeling \(C\) and \(\tilde C\), but still depending only on \(n\) and \(s\)). Choosing \begin{align*}
\tau &:=  \frac12 \tilde C M^{n+2s+2}  \| y_1 v \|_{L^1(K)}\big(
1  + \rho^{2s}\|c^+\|_{L^\infty(U)} \big)^{-1}
\end{align*} we obtain that \begin{align*}
(-\Delta)^s w + cw &\leqslant 0 \qquad \text{in } B^+. 
\end{align*} Moreover, in \(\R^n_+ \setminus B^+\), we have that~\(w = v \chi_K \leqslant  v\), so,
recalling also~\eqref{ed0395v04ncrhrhfwejfvwejhfvqjhkfvqejhk}, \cite[Proposition 3.1]{MR3395749} 
implies that,
in~\(B^+\), \begin{align*}
v(x) &\geqslant w(x)  =  \frac12 \tilde C M^{n+2s+2} \big(
1+\rho^{2s}\|c^+\|_{L^\infty(U)} \big)^{-1} \| y_1 v \|_{L^1(K)} \varphi(x). 
\end{align*} Recalling that \(\varphi \geqslant \rho^{2s} x_1 \chi_{B_{\rho/2}^+(a)}\) 
(by Lemma~\ref{gmjWA2gy}),
we obtain the final result. 
\end{proof}

\section{The stability estimate and proof of Theorem~\ref{oAZAv7vy} }
\label{sec:stabest}

The proof of \thref{oAZAv7vy} makes use of the method of moving planes. Before we begin our discussion of this technique and give the proof of \thref{oAZAv7vy}, we must fix some notation. Let~\(\mu \in \R\),  \(e\in \Sph^{n-1}\), and~\(A \subset \R^n\). Then we have the following standard definitions: 
\begin{align*}
\pi_\mu&=\{ x\in \R^n \text{ s.t. } x\cdot e =\mu \} && \text{a hyperplane orthogonal to }e \\
H_\mu&=\{x\in \R^n  \text{ s.t. } x\cdot e>\mu \} &&\text{the right-hand half space with respect to } \pi_\mu  \\
H_\mu'&=\{x\in \R^n  \text{ s.t. } x\cdot e<\mu \} &&\text{the left-hand half space with respect to } \pi_\mu\\
A_\mu &= A \cap H_\mu &&\text{the portion of }A \text{ on the right-hand side of } \pi_\mu \\
x_\mu' &= x-2(x\cdot e -\mu) e && \text{the reflection of } x \text{ across } \pi_\mu \\ 
A_\mu' &= \{x\in \R^n \text{ s.t. } x'_\mu\in A_\mu \} && \text{the reflection of } A_\mu \text{ across } \pi_\mu
\end{align*} Note that in some articles such as \cite{MR3836150} \(A'_\mu\) is used to denote the reflection of~\(A\) (instead of~\(A_\mu\)) across~\(\pi_\mu\).

The method of moving planes works as follows. Fix a direction \(e\in \Sph^{n-1}\) and suppose that \(\Omega\) is a bounded open subset of \(\R^n\) with \(C^1\) boundary. Since \(\Omega\) is bounded, for \(\mu\) sufficiently large the hyperplane \(\pi_\mu\) does not intersect \(\Omega\). Furthermore, by decreasing the value of \(\mu\), at some point \(\pi_\mu\) will intersect \(\overline \Omega\). We denote the value of \(\mu\) at this point by \begin{align*}
\Lambda = \Lambda_e :=\sup\{x \cdot e \text{ s.t. } x\in \Omega \}.
\end{align*} From here, we continue to decrease the value of \(\mu\). Initially, since \(\partial \Omega\) is \(C^1\), the reflection of \(\Omega_\mu\) across \(\pi_\mu\) will be contained within \(\Omega\), that is \(\Omega_\mu'\subset \Omega\) for \(\mu < \Lambda\) but with \(\mu\) sufficiently close to \(\Lambda\).
Eventually, as we continue to make \(\mu\) smaller, there will come a point when this is no longer the case.
More precisely, there exists \(\lambda =\lambda_e\in \R\) such that for all \(\mu \in [\lambda,\Lambda)\), it occurs that~\(\Omega'_\mu \subset \Omega\), but~\(\Omega'_\mu \not\subset \Omega\) for~\(\mu<\lambda\). We may write~\(\lambda\) more explicitly as \begin{align}
\lambda := \inf \Big\{ \tilde \mu \in \R \text{ s.t. } \Omega'_\mu \subset \Omega \text{ for all } \mu \in (\tilde \mu , \Lambda) \Big\} . \label{IAhmMsFq}
\end{align} When \(\mu = \lambda\), geometrically speaking, there are two possibilities that can occur: \label{pagemov}

\emph{Case 1:} The boundary of \(\Omega'_\lambda\) is internally tangent to~\(\partial\Omega\) at some point not on~\(\pi_\lambda\), that is, there exists~\(p\in (\partial \Omega \cap \partial \Omega'_\lambda) \setminus \pi_\lambda \); or 

\emph{Case 2:} The critical hyperplane \(\pi_\lambda\) is orthogonal to the~\(\partial \Omega\) at some point, that is, there exists~\(p\in \partial\Omega \cap \pi_\mu\) such that the normal of~\(\partial \Omega\) at~\(p\) is contained in the plane~\(\pi_\mu\). 

At this stage, let us introduce the function \begin{align}
v_\mu(x) := u(x) - u( x_\mu' ), \qquad \text{for all } \mu \in \R \text{ and }  x\in \R^n . \label{4LML9uGd}
\end{align} It follows that, in \(\Omega_\mu'\), \begin{align*}
(-\Delta)^s v_\mu (x) &= f(u(x)) - f(u(x_\mu')) = -c_\mu(x) v_\mu (x) 
\end{align*} where \begin{align*}
c_\mu(x) &= \int_0^1 f'((1-t)u(x) +tu(x_\mu')) \dd t .
\end{align*} Note that \begin{align*}
\| c_\mu \|_{L^\infty(\Omega_\mu ')} \leqslant [f]_{C^{0,1}([0,\| u\|_{L^\infty(\Omega)}])}.
\end{align*} Hence, \(v_\mu\) is an antisymmetric function that satisfies \begin{align*}
\begin{PDE}
(-\Delta)^s v_\mu  +c_\mu v_\mu &= 0 &\text{in } \Omega_\mu ' ,\\
v_\mu &= u &\text{in } (\Omega\cap H_\mu') \setminus \Omega_\mu' ,\\
v_\mu &= 0 &\text{in } H_\mu' \setminus \Omega ,
\end{PDE}
\end{align*} with \(c_\mu \in L^\infty (\Omega_\mu ')\).

In the situations where one expects that~\(\Omega\) should be a ball, the goal of the method of moving planes is to prove that~\(v_\lambda \equiv 0\) i.e. \(u\) is even with respect to reflections across the critical hyperplane~\(\pi_\lambda\). Since the direction~\(e\) was arbitrary, one can then deduce that~\(u\) must be radial with respect to some point. The proof that~\(v_\lambda \equiv 0\) is achieved through repeated applications of the maximum principle.

\begin{remark} \thlabel{VbKq6OVB}
From the preceding exposition, it is clear that in several instances we will need to evaluate \(u\) and \(v_\mu\) at a single point. This is technically an issue since, in \thref{oAZAv7vy}, \(u\) is only assumed to be in \(H^s(\R^n)\cap L^\infty(\R^n)\). However, by standard regularity theory, we have that\footnote{Here we are using the notation that for \(\alpha>0\) with \(\alpha\) not an integer, \(C^\alpha(\Omega):=C^{k,\beta}(\Omega)\) where \(k\) is the integer part of \(\alpha\) and \(\beta=\alpha-k\in (0,1)\).} \(u \in C^{2s+1-\varepsilon}(\Omega)\) for all \(\varepsilon>0\) such that \(2s+1-\varepsilon\) is not an integer. In particular, this implies that~\(u\in C^1(\Omega)\) which will be essential for the proof of the theorem. Indeed, using that \(f\) is locally Lipschitz, we have that \((-\Delta)^s u = -cu+f(0)\) with \(c(x) = - \int_0^1 f'(tu(x))\dd t\), so \((-\Delta)^su \in L^\infty(\Omega)\). Hence, it follows that \(u\in C^{2s-\varepsilon}(\Omega)\) for all \(\varepsilon\in (0,2s)\), \(2s-\varepsilon \not\in \Z\), see \cite[Proposition 2.3]{MR3168912}. Then, it follows that~\((-\Delta)^su = f(u) \in C^{\min \{1-\varepsilon,2s-\varepsilon \}} (\Omega)\), so by \cite[Proposition~2.2]{MR3168912} and a bootstrapping argument (if necessary), we obtain that~\(u \in C^{2s+1-\varepsilon}(\Omega)\). 
\end{remark}

\subsection{Uniform stability in each direction}

In this subsection, we will use the maximum principle of Section~\ref{S7DGIjUf} to prove uniform stability for each direction~\(e\in \Sph^{n-1}\), that is, we will show, for each~\(e\in \Sph^{n-1}\), that~\(\Omega\) is almost symmetric with respect to~\(e\). This is stated precisely in \thref{7TQmUHhl}. We will repeatedly use that fact that~\(v_\lambda\) is~\(C^1\) which follows from \thref{VbKq6OVB}. Before proving \thref{7TQmUHhl}, we have two lemmata. 

\begin{lem} \thlabel{JFhLgBx4}
Let \(\Omega\) be a bounded open set with \(C^1\) boundary, \(e\in \Sph^{n-1}\), and \(v_\mu\) as in~\eqref{4LML9uGd}. 

Then, for all \(\mu \in [\lambda , \Lambda]\), we have that~\(v_\mu \geqslant 0\) in \(\Omega_\mu'\).   
\end{lem}

The proof of \thref{JFhLgBx4} is given as part of the proof of Theorem 1.4 in \cite{RoleAntisym2022}, so we will not include it again here. However, we would like to emphasise that, even though \cite[Theorem 1.4]{RoleAntisym2022} assumes the solution \(u\) of~\eqref{problem00} is constant on \(\partial G\) (i.e. \([u]_{\partial G}=0\)), this assumption was unnecessary to obtain the much weaker result of \thref{JFhLgBx4}.

We now give the second lemma.

\begin{lem} \thlabel{goNy8kYt}
Let \(\Omega\) and~$G$ be open bounded sets with~\(C^1\) boundary
such that~$\Omega=G+B_R$
for some~$R>0$. 
Let~\(f \in C^{0,1}_{\mathrm{loc}}(\R)\) be such that \(f(0)\geqslant 0\)
and~\(u\in  H^s (\R^n)\) be a solution of~\eqref{problem00}.

Then, for each \(e\in \Sph^{n-1}\), we have that \begin{align}
 \int_{(\Omega\cap H'_\lambda) \setminus \Omega_\lambda'} \delta_{\pi_\lambda}(x) u (x) \dd x \leqslant C  [u]_{\partial G} \label{HqEzyGyW},
\end{align} where \begin{align}
C := C(n,s)  R^{-2s} (\diam \Omega)^{n+2s+2} \Big(1+ R^{2s}[f]_{C^{0,1}([0,\|u\|_{L^\infty(\Omega)}])} \Big). \label{KW9pZNdx}
\end{align}
\end{lem}

\begin{proof}
Without loss of generality, take \(e=-e_1\) and \(\lambda =0\). 
We now apply the method of moving planes to~$G$.
First, suppose that we are in the first case,
namely, the boundary of \(G'_\lambda\) is internally tangent to~\(\partial G\) at some point not on~\(\{x_1=0\}\),
and let~\(p\in (\partial G \cap \partial G_\lambda') \setminus \{x_1=0\}\).

By \thref{TV1cTSyn} with~$U:=\Omega'_\lambda$, \(K := (\Omega \cap H_\lambda')\setminus \Omega_\lambda'\) and~\(B_\rho(a):=B_{R/2}(p)\)
(notice that condition~\eqref{ed0395v04ncrhrhfwejfvwejhfvqjhkfvqejhk}
is satisfied, possibly taking a smaller ball centered at~$p$), we have that \begin{align}
\frac{v_\lambda(p)}{p_1} &\geqslant C \int_{(\Omega\cap H_\lambda') \setminus \Omega_\lambda '} y_1 v_\lambda (y) \dd y = C \int_{(\Omega\cap H_\lambda') \setminus \Omega_\lambda '} y_1 u(y) \dd y. \label{bYXi8urM}
\end{align}
Note that, since \(x_\lambda'\) belongs to the reflection of~$\Omega$ across~\(\{x_1=0\}\)
for each \(x\in(\Omega\cap H_\lambda') \setminus \Omega_\lambda'\), we have that  \begin{align*}
\inf_{{x\in( \Omega\cap H'_\lambda) \setminus \Omega_\lambda'}\atop{y\in B_{R/2}^+(p)}} \vert x_\lambda'- y \vert^{-1} &\geqslant (\diam \Omega)^{-1}, 
\end{align*} so \thref{TV1cTSyn} implies that the constant in~\eqref{bYXi8urM} is given by \begin{align*}
C = C(n,s) R^{2s} (\diam \Omega)^{-n-2s-2} \big(1+R^{2s}\|c^+\|_{L^\infty(\Omega_\lambda')} \big)^{-1}.
\end{align*} Moreover, we have that \begin{align*}
\frac{v_\lambda(p)}{p_1} = \frac{u(p)-u(p_\lambda')}{p_1} = \frac{2(u(p)-u(p_\lambda'))}{\vert p_1-(p_\lambda')_1\vert} \leqslant 2 [u]_{\partial G},
\end{align*} which, along with~\eqref{bYXi8urM}, gives~\eqref{HqEzyGyW}. 

Now, let us suppose that we are in the second case and let~\(p\in \partial G \cap \{x_1=0\}\) be
such that the normal of~\(\partial G\) at~\(p\) is contained in~\(\{x_1=0\}\).
Proceeding in a similar fashion as the first
case, we apply \thref{TV1cTSyn} with~$U:=\Omega'_\lambda$,
\(K := (\Omega\cap H_\lambda') \setminus \Omega_\lambda'\) and~\(B_\rho(a):=B_{R/2}(p+he_1)\) with~\(h>0\) very small,
to obtain that \begin{align*}
\int_{(\Omega\cap H_\lambda') \setminus \Omega_\lambda'} x_1 u (x) \dd x \leqslant C \frac{v_\lambda(p+he_1)}h
\end{align*} with \(C\) in the same form as in~\eqref{KW9pZNdx} and, in particular, independent of \(h\). Sending \(h\to 0^+\), we obtain that \begin{align*}
\int_{(\Omega\cap H_\lambda') \setminus \Omega_\lambda'} x_1 u (x) \dd x \leqslant C \partial_1 v_\lambda(p) \leqslant C [u]_{\partial G} 
\end{align*} which gives~\eqref{HqEzyGyW} in this case as well. 
\end{proof}

We are now able to obtain uniform stability for each direction in \thref{7TQmUHhl} below. 

\begin{prop}\thlabel{7TQmUHhl}
Let \(\Omega\) be an open bounded set with~\(C^1\) boundary and satisfying the uniform interior ball condition with radius~\(r_\Omega >0\)
and~$G$ be an open bounded set with~\(C^1\) boundary
such that~$\Omega=G+B_R$
for some~$R>0$. Let~\(f \in C^{0,1}_{\mathrm{loc}}(\R)\) be such that \(f(0)\geqslant 0\)
and~\(u\in  H^s (\R^n)\) be a solution of~\eqref{problem00}.

For \(e\in \Sph^{n-1}\), let \(\Omega '\) denote the reflection of \(\Omega\) with respect to the critical hyperplane \(\pi_\lambda\).

Then, \begin{align}
\vert \Omega \triangle \Omega' \vert \leqslant C_\star [u]_{\partial G}^{\frac1{s+2}}, \label{D8UMJ0i3}
\end{align} where \begin{eqnarray*}
C_\star &:=& C(n,s) \Big( C_\ast^{-1}
\big( f(0)+\|u\|_{L_s(\R^n)}\big)^{-1}
+ (\diam \Omega)^{n-1} + r_\Omega^{-1}\vert \Omega \vert \Big) \\
&&\qquad\times R^{-2s} (\diam \Omega)^{n+2s+2} \Big(
1+R^{2s} [f]_{C^{0,1}([0,\|u\|_{L^\infty(\Omega)}])} \Big)
\end{eqnarray*} and \(C_\ast\) is as in \thref{zAPw0npM}.
\end{prop}

\begin{proof}
Without loss of generality, take \(e=-e_1\) and \(\lambda =0\). By \thref{zAPw0npM}, we have that \begin{align}
\int_{(\Omega\cap H_\lambda') \setminus \Omega_\lambda'} x_1 u (x) \dd x &\geqslant C_\ast \big( f(0)+\|u\|_{L_s(\R^n)}\big)
 \int_{(\Omega\cap H_\lambda') \setminus \Omega_\lambda'} x_1 \delta_{\partial \Omega}^s(x) \dd x. \label{JfjtK29b}
\end{align}
Fix \(\gamma >0\). By Chebyshev's inequality, \eqref{JfjtK29b} and \thref{goNy8kYt}, we have that  \begin{equation}\begin{split}\label{ido4985vb6tdfghwafrikewytoiw}
\Big| \big\{ x \in (\Omega\cap H_\lambda') \setminus \Omega_\lambda' \text{ s.t. } x_1 \delta_{\partial \Omega}^s(x) >\gamma \big\} \Big| \leqslant\;& \frac1{\gamma}
\int_{(\Omega\cap H_\lambda') \setminus \Omega_\lambda'} x_1 \delta_{\partial \Omega}^s (x) \dd x \\ \leqslant\;& \frac {CC_\ast^{-1}\big( f(0)+\|u\|_{L_s(\R^n)}\big)^{-1}} \gamma [u]_{\partial G}.
\end{split}\end{equation}  Moreover, \begin{eqnarray*}&&
\Big| \big\{ x \in (\Omega\cap H_\lambda') \setminus \Omega_\lambda' \text{ s.t. } x_1 \delta_{\partial \Omega}^s(x) \leqslant \gamma \big\} \Big| \\&=&
\Big| \big\{ x \in (\Omega\cap H_\lambda') \setminus \Omega_\lambda' \text{ s.t. } x_1 \delta_{\partial \Omega}^s(x) \leqslant \gamma, x_1< \gamma^{\frac 1{s+1}} \big\} \Big| \\&&\qquad + \Big| \big\{ x \in (\Omega \cap H_\lambda')
\setminus \Omega_\lambda' \text{ s.t. } x_1 \delta_{\partial \Omega}^s(x) \leqslant \gamma , x_1 \geqslant \gamma^{\frac 1{s+1}} \big\} \Big| \\
&\leqslant& \Big| \big\{ x \in \Omega^+ \text{ s.t. } x_1< \gamma^{\frac 1{s+1}} \big\} \Big| 
+  \Big| \big\{ x \in \Omega \text{ s.t. } \delta_{\partial \Omega}(x) \leqslant \gamma^{\frac 1{s+1}} \big\} \Big| . 
\end{eqnarray*}

Furthermore, we have the estimate \begin{align*}
 \Big| \big\{ x \in \Omega^+ \text{ s.t. } x_1< \gamma^{\frac 1{s+1}} \big\} \Big| &\leqslant (\diam \Omega)^{n-1} \gamma^{\frac 1{s+1}}
\end{align*} and, by \cite[Lemma 5.2]{MR4577340} in the case that \(\partial \Omega\) is \(C^2\) and more generally in \cite{MR483992}, we have that \begin{align*}
 \Big| \big\{ x \in \Omega \text{ s.t. } \delta_{\partial \Omega}(x) \leqslant \gamma^{\frac 1{s+1}} \big\} \Big|
 &\leqslant \frac{2n \vert \Omega \vert }{r_\Omega} \gamma^{\frac 1{s+1}}.
\end{align*} Thus,
\begin{eqnarray}\label{eq:NUOVAPERIMPROVEMENT}
\Big| \big\{ x \in (\Omega\cap H_\lambda') \setminus \Omega_\lambda' \text{ s.t. } x_1 \delta_{\partial \Omega}^s(x) \leqslant \gamma \big\} \Big|
\le \left[ (\diam \Omega)^{n-1} +\frac{2n \vert \Omega \vert }{r_\Omega} \right] \gamma^{\frac 1{s+1}}.
\end{eqnarray}
{F}rom this and~\eqref{ido4985vb6tdfghwafrikewytoiw}, we deduce that
\begin{eqnarray*}
\vert (\Omega\cap H_\lambda') \setminus \Omega'_\lambda \vert \leqslant
\frac {CC_\ast^{-1}\big( f(0)+\|u\|_{L_s(\R^n)}\big)^{-1}} \gamma [u]_{\partial G}
+\left[ (\diam \Omega)^{n-1} +\frac{2n \vert \Omega \vert }{r_\Omega} \right] \gamma^{\frac 1{s+1}}.
\end{eqnarray*}
Hence,
 \begin{align*}
\vert \Omega \triangle \Omega '\vert &= 2 \vert (\Omega\cap H_\lambda') \setminus \Omega'_\lambda \vert \leqslant \frac {2CC_\ast^{-1}\big( f(0)+\|u\|_{L_s(\R^n)}\big)^{-1}} {\gamma} [u]_{\partial G}
+2\left[ (\diam \Omega)^{n-1} +\frac{2n \vert \Omega \vert }{r_\Omega} \right] \gamma^{\frac 1{s+1}}
\end{align*} for all \(\gamma >0\). Choosing \(\gamma = [u]_{\partial G}^{\frac{s+1}{s+2}}\),
we obtain that
$$
\vert \Omega \triangle \Omega '\vert \leqslant \left[ 2CC_\ast^{-1}\big( f(0)+\|u\|_{L_s(\R^n)}\big)^{-1}
+2\left( (\diam \Omega)^{n-1} +\frac{2n \vert \Omega \vert }{r_\Omega} \right) \right][u]_{\partial G}^{\frac 1{s+2}},$$
which gives~\eqref{D8UMJ0i3}. 
\end{proof}

Note, setting \([u]_{\partial G}=0\) in \thref{7TQmUHhl} implies that \(\Omega\) (and, therefore, \(G\)) must be symmetric with respect to the critical hyperplane for every direction. This implies that they are both balls, thereby recovering the main result of \cite{RoleAntisym2022}. 

\subsection{Proof of Theorem~\ref{oAZAv7vy}} 

We will now use the results of the previous subsection to prove \thref{oAZAv7vy}. In fact, \thref{oAZAv7vy} follows almost immediately from the following result, \thref{k4VRS3Fx}. The idea of \thref{k4VRS3Fx} is to choose a centre for \(\Omega\) by looking at the intersection of \(n\) orthogonal critical hyperplanes, then to prove that every other critical hyperplane is quantifiably close to this centre in terms of the semi-norm \([u]_{\partial G}\). The precise statement is as follows. 

\begin{prop} \thlabel{k4VRS3Fx} Let \(\Omega\) be an open bounded set with \(C^1\) boundary and satisfying the uniform interior ball condition with radius \(r_\Omega >0\) and suppose that the critical planes \(\pi_{e_i}\) with respect to the coordinate directions \(e_i\) coincide with \(\{x_i=0\}\) for every \(i=1,\dots, n\). Also, given \(e\in \Sph^{n-1}\), denote by \(\lambda_e\) the critical value associated with~\(e\) as in~\eqref{IAhmMsFq}. 

Assume that \begin{align}
[u]^{\frac 1{s+2} }_{\partial G} \leqslant \frac {\vert \Omega \vert } {n C_\star }   \label{MCxhJzBW}
\end{align} where \(C_\star\) is the constant in \thref{7TQmUHhl}.

Then,
\begin{align}\label{eq:NEWNEW OLD lambda estimate}
\vert \lambda_e \vert \leqslant C [u]_{\partial G}^{\frac 1 {s+2} }
\end{align}
for all \(e\in \Sph^{n-1}\) with \begin{eqnarray*}
C &:=& C(n,s) \Big ( C_\ast^{-1} 
\big( f(0)+\|u\|_{L_s(\R^n)}\big)^{-1}+ (\diam \Omega)^{n-1} + r_\Omega^{-1}\vert \Omega \vert \Big ) \\&&\qquad\times R^{-2s} \vert \Omega \vert^{-1} (\diam \Omega)^{n+2s+3} \Big(
1+R^{2s} [f]_{C^{0,1}([0,\|u\|_{L^\infty(\Omega)}])} \Big)
\end{eqnarray*} and \(C_\ast\) is as in \thref{zAPw0npM}.
\end{prop}

\begin{proof}
Define \(\Omega^{\mathbf 0} := \{-x \text{ s.t. } x\in \Omega\}\). Moreover, let \(Q_i:=Q_{\pi_{e_i}}\) be the reflections across each critical plane \(\pi_{e_i}\) and define \(\Omega^{\mathbf 0}\) recursively via \(\Omega^{\mathbf 0}_{i+1}  := Q_{i+1}(\Omega^{\mathbf 0}_i)\) for \(i=1,\dots n-1\) with \(\Omega^{\mathbf 0}_1:=\Omega\). Observe that \( \Omega^{\mathbf 0} = \Omega^{\mathbf 0}_n\). Via the triangle inequality for  symmetric difference, it follows that \begin{align*}
\vert \Omega \triangle \Omega^{\mathbf 0} \vert &\leqslant \vert \Omega \triangle Q_n(\Omega ) \vert  + \vert  Q_n(\Omega ) \triangle  Q_n(\Omega^{\mathbf 0}_{n-1}  )\vert 
\end{align*} Since \( \vert  Q_n(\Omega ) \triangle  Q_n(\Omega^{\mathbf 0}_{n-1}  )\vert = \vert  Q_n(\Omega  \triangle  \Omega^{\mathbf 0}_{n-1}  )\vert = \vert  \Omega  \triangle  \Omega^{\mathbf 0}_{n-1} \vert \), we have that \begin{align*}
\vert \Omega \triangle \Omega^{\mathbf 0} \vert &\leqslant \vert \Omega \triangle Q_n(\Omega ) \vert  + \vert  \Omega  \triangle  \Omega^{\mathbf 0}_{n-1} \vert .
\end{align*} Iterating, we obtain \begin{align}
\vert \Omega \triangle \Omega^{\mathbf 0} \vert &\leqslant \sum_{i=1}^n \vert \Omega \triangle Q_i(\Omega ) \vert \leqslant n C_\star [u]_{\partial G}^{\frac1 {s+2}} \label{1wuBMACC}
\end{align} by~\thref{7TQmUHhl}. 

Next, let us assume that \(\lambda_e>0\) (the case \(\lambda_e<0\) is analogous). If \(\Lambda_e > \diam \Omega\) then \(x\cdot e \geqslant \Lambda_e - \diam \Omega  \geqslant 0\) for all \(x\in \Omega\), so \(\vert \Omega \triangle \Omega^{\mathbf 0} \vert = 2\vert \Omega \vert \). However, this is in contradiction with~\eqref{MCxhJzBW} and~\eqref{1wuBMACC}, so we must have that \(\Lambda_e \leqslant \diam \Omega\). Arguing as in Lemma~4.1 in~\cite{MR3836150} using~\thref{7TQmUHhl} instead of \cite[Proposition 3.1 (a)]{MR3836150} , we find that \begin{align*}
\vert \Omega_{\lambda_e} \vert \lambda_e \leqslant (n+3) C_\star (\diam \Omega)  [u]_{\partial G}^{\frac 1{s+2}} . 
\end{align*} Now, recalling the notation of Section~\ref{sec:stabest}, we have that \begin{align*}
\vert \Omega \triangle \Omega^{\mathbf 0} \vert &= 2 \vert (\Omega \cap H_{\lambda_e}' )\setminus \Omega_{\lambda_e}' \vert = 2 \big ( \vert \Omega \vert - 2 \vert \Omega_{\lambda_e} \vert \big ),
\end{align*} and therefore \begin{align*}
\vert \Omega_{\lambda_e} \vert  =\frac{\vert \Omega \vert } 2 - \frac14 \vert \Omega \triangle \Omega^{\mathbf 0} \vert  \geqslant \frac{\vert \Omega \vert } 2 -\frac14 C_\star [u]_{\partial G}^{\frac1{s+2}} \geqslant \frac{\vert \Omega \vert } 4 
\end{align*} by \thref{7TQmUHhl}  and~\eqref{MCxhJzBW}. Thus, we conclude that \begin{align*}
\vert \lambda_e\vert &\leqslant  \frac{C(n) C_\star  (\diam \Omega)}{\vert \Omega \vert }  [u]_{\partial G}^{\frac 1{s+2}} ,
\end{align*}
as desired.
\end{proof}

Now the proof of \thref{oAZAv7vy} follows almost immediately from \thref{k4VRS3Fx}. 

\begin{proof}[Proof of \thref{oAZAv7vy}] 
The result follows by reasoning as in the proof of \cite[Theorem~1.2]{MR3836150}, but using \thref{k4VRS3Fx} instead of \cite[Lemma 4.1]{MR3836150}.
Notice that the dependence on $r_\Omega$ appearing in \thref{k4VRS3Fx} can be removed (and, in fact, does not appear in \thref{oAZAv7vy}).
In fact, by the definition of $\Omega := G + B_R(0)$, we have that $\Omega$ automatically satisfies the uniform interior sphere condition and we can take, e.g., $r_\Omega := R/2$.
\end{proof}

\section{The role of boundary estimates in the attainment of the optimal exponent} \label{fEsBEcuv}

In this section, we give a broad discussion on some of the challenges the nonlocality of the fractional Laplacian presents in obtaining the optimal exponent. By way of an example, via the Poisson representation formula, we show that estimates for a singular integral involving the reciprocal of the distance to the boundary function play a key role in obtaining the anticipated optimal result. This suggests, surprisingly, that fine geometric estimates for the distance function close to the boundary are required to obtain the optimal exponent. 

Recall the notation at the beginning of Section~\ref{sec:stabest} and also
\begin{eqnarray*}
&&\R^n_+:=\{x=(x_1,\dots,x_n)\in\R^n \;{\mbox{ s.t. }} x_1>0\},\\
&&\R^n_-:=\{x=(x_1,\dots,x_n)\in\R^n \;{\mbox{ s.t. }} x_1<0\},\\
&&B_1^+:=B_1\cap\R^n_+, \qquad B_1^-:=B_1\cap\R^n_-,\\
&&\Omega^+:=\Omega\cap \R^n_+ \qquad
{\mbox{and}} \qquad \Omega^-:=\Omega\cap \R^n_-.
\end{eqnarray*}

Let~$\Omega$ and $G$ be bounded and smooth sets.
For the purposes of this discussion, let us assume that \(\Omega = G+B_{1/2}\) where~\(G\) is such that \(\R^n_- \cap G = B_{1/2}^-\), \(B_{1/2}^+ \subset \R^n_+ \cap \Omega\), and the critical plane~\(\pi_\lambda\) corresponding to running the method of moving planes with~\(e=-e_1\) is equal to~\(\{x_1=0\}\), see Figure~\ref{Fig2}.

\begin{figure}[h]
\centering 
\includegraphics{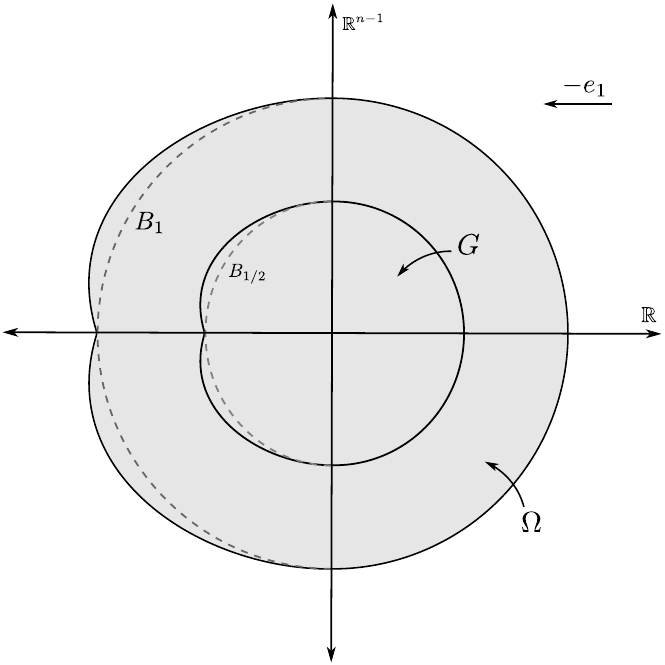}
\caption{Geometry of \(\Omega\) as described in Section~\ref{fEsBEcuv}.}
\label{Fig2}
\end{figure}

Suppose that \(u \in C^2(\Omega )\cap L^\infty(\R^n)\) satisfies the torsion problem \begin{align*}
\begin{PDE}
(-\Delta)^s u &= 1&\text{in } \Omega ,\\
u &=0 &\text{in } \R^n \setminus \Omega . 
\end{PDE}
\end{align*} In this case, if \(v(x) := u(x) - u(-x_1,x_2,\dots,x_n)\), then \(v\) is antisymmetric in \(\R^n\) with respect to~$\{x_1=0\}$
and \(s\)-harmonic in~\(B_1\). Hence, by the Poisson kernel representation, (up to normalization constants) \begin{align*}
v(x) &=  \int_{\R^n \setminus B_1 } \bigg ( \frac{1-\vert x \vert^2}{\vert y \vert^2-1} \bigg )^s \frac{v(y)}{\vert x - y \vert^n} \dd y .
\end{align*} Using the antisymmetry of \(v\), we may rewrite this as \begin{align*}
v(x) &= \int_{\R^n_+ \setminus B_1^+ } \bigg ( \frac{1-\vert x \vert^2}{\vert y \vert^2-1} \bigg )^s \bigg (\frac1{\vert x - y \vert^n} - \frac1{\vert x' - y \vert^n}\bigg )  v(y) \dd y \\
&=  \int_{\Omega^+ \setminus B_1^+ } \bigg ( \frac{1-\vert x \vert^2}{\vert y \vert^2-1} \bigg )^s \bigg (\frac1{\vert x - y \vert^n} - \frac1{\vert x' - y \vert^n}\bigg ) u(y) \dd y,
\end{align*}
where~$x'$ denotes the reflection of the point~$x$ with respect to~$\{x_1=0\}$.

Now, by \thref{oAZAv7vy}, when \([u]_{\partial G}\) is sufficiently small, then~\(\Omega\) is uniformly close to~\(B_1\), so we can suppose that~\(G_\lambda' \subset B_{3/4}\). Therefore, for all \(x\in \partial G\) and \(y \in \R^n_+\setminus \Omega\),  \begin{align*}
(1-\vert x \vert^2)^s \bigg (\frac1{\vert x - y \vert^n} - \frac1{\vert x' - y \vert^n}\bigg ) &\geqslant C x_1y_1.
\end{align*} 
Also, by Corollary~\ref{zAPw0npM},
we have that~\(u \geqslant C \delta^s_{\partial \Omega}\)
(with~$C>0$ not depending on~$u$),
and thus
\begin{align*}
v(x) &\geqslant C x_1 \int_{\Omega^+ \setminus B_1^+ }  \frac {y_1 u(y)}{(\vert y \vert^2-1)^s}  \dd y \geqslant Cx_1 \int_{\Omega^+ \setminus B_1^+ }  \frac {y_1\delta_{\partial \Omega}^s(y)}{\delta_{\partial B_1}^s(y)}  \dd y. 
\end{align*}

Recall the discussion on the moving plane method on page~\pageref{pagemov}
and, for simplicity, assume that we are in the first case (the second
case can be treated similarly), so there exists \(p\in (\partial G \cap \partial G_\lambda')\setminus \{x_1=0\}\). Hence, \begin{align} 
\int_{\Omega^+ \setminus B_1^+ }  y_1\bigg ( \frac {\delta_{\partial \Omega}(y)}{\delta_{\partial B_1}(y)} \bigg )^s  \dd y &\leqslant C \frac{v(p)}{p_1} \leqslant C [u]_{\partial G}. \label{K98q9WSJ}
\end{align}

Now the point of obtaining~\eqref{K98q9WSJ} is that the left hand side is geometric (does not depend on~\(u\)), and~\eqref{K98q9WSJ} is sharp in the sense that the only terms that have been `thrown away'
are bounded away from zero as \([u]_{\partial G} \to 0\). This suggests that if \([u]_{\partial G}\) is of order \(\varepsilon\) then the behaviour of \begin{align*}
\int_{\Omega^+ \setminus B_1^+ }  y_1\bigg ( \frac {\delta_{\partial \Omega}(y)}{\delta_{\partial B_1}(y)} \bigg )^s  \dd y
\end{align*} as a function of \(\varepsilon\) will entirely determine the optimal exponent~\(\overline\beta(s)\) (recall that
the definition of~\(\overline\beta(s)\) is given after the statement of \thref{oAZAv7vy} in Section~\ref{tH4AETfY}). This is surprising, since in the local case, the problem is entirely an interior one in that the proof relies only on interior estimates such as the Harnack inequality while, it appears that in the nonlocal case, the geometry of the boundary may have a significant effect on the value of~\(\overline\beta(s)\). If one hopes to obtain that~\(\overline\beta(s)=1\) from~\eqref{K98q9WSJ} then it would be necessary to show that \begin{align*}
\int_{\Omega^+ \setminus B_1^+ }  y_1\bigg ( \frac {\delta_{\partial \Omega}(y)}{\delta_{\partial B_1}(y)} \bigg )^s  \dd y \geqslant C \vert \Omega \setminus B_1^+ \vert. 
\end{align*} If it were true that \(y_1\delta^s_{\partial \Omega} \geqslant C \delta^s_{\partial B_1}\) then this would follow immediately; however, this is not the case as seen by sending \(y \to \partial \R^n_+ \cap B_{1/2}\). 

Moreover, even though we made several major assumptions on the geometry of \(\Omega\) to obtain~\eqref{K98q9WSJ}, we believe that the inequality is indicative of the more general situation. Indeed, one would expect that, under reasonable assumptions on \(\partial \Omega\), similar estimates to the ones employed above would hold for Poisson kernels in general domains, so one may suspect that an inequality of the form \begin{align*}
\int_{\Omega \triangle \Omega ' } \delta_{\pi_\lambda}(y)  \bigg ( \frac {\delta_{\partial \Omega}(y)}{\delta_{\partial \Omega' }(y)} \bigg )^s  \dd y &\leqslant C [u]_{\partial G}
\end{align*} should hold. It would be interesting in future articles to further explore this methodology to try obtain improvements on the exponent in \thref{oAZAv7vy}.

\section{Sharp investigation of a technique to improve the stability exponent}\label{c8w8u7Hn}

In the papers \cite{MR3881478, MR3836150}, it was proven that the only open bounded sets\footnote{For the precise statements and the regularity assumptions on the boundary of the set, see \cite{MR3836150,MR3881478}.} whose boundary has constant nonlocal mean curvature are balls.
Moreover, in \cite{MR3836150}, the method of moving planes was employed to obtain a stability estimate for sets whose boundary has
\emph{almost} constant nonlocal mean curvature. The stability estimate that
the authors proved also seems to achieve the optimal exponent of~\(1\) which is however
obtained using the following technical result,
see~\cite[Proposition 3.1(b)]{MR3836150}:

\begin{lem} \thlabel{Z08Q7917}
Let \(\Omega \subset \R^n\) be an open bounded set with \(C^{2,\alpha}\) boundary for~\(\alpha >2s\) and~\(s\in (0,1/2)\). Let~\(\pi_\lambda\) be the
critical hyperplane. Suppose that~\(\dist(0,\pi_\lambda) \leqslant 1/ 8\) and~\(B_r \subset \Omega \subset B_R\) with~\(1/2 \leqslant r\leqslant R \leqslant 2\).

Then, \begin{align}
\Big| \big\{ x\in \Omega \triangle \Omega ' \text{ s.t. } \dist(x,\pi_\lambda) \leqslant \gamma \big\} \Big| \leqslant C(n) \gamma (R-r) \label{IAMJKSyF}
\end{align} for all \(0<\gamma \leqslant 1/4\). 
\end{lem} Without this result, the exponent in the stability estimate in~\cite{MR3836150} would have been \(1/2\), so \thref{Z08Q7917} seems to play a major role in~\cite{MR3836150} in the attempt of ``doubling the exponent'' that they obtained in previous estimates. This feature is extremely relevant to our main result, \thref{oAZAv7vy}, since we hoped that a similar argument to \thref{Z08Q7917} would lead to an exponent that was twice the one that we obtained. Unfortunately, we believe that, in the very broad generality in which the result is stated in~\cite{MR3836150}, \thref{Z08Q7917} may not be true, and we present a counter-example
(which may also impact some of the statements in~\cite{MR3836150}) as well as a corrected statement of \thref{Z08Q7917}. 

\subsection{A geometric lemma and a counter-example to \protect\thref{Z08Q7917}} 
\label{6LQMifK3}

The main result of this section is a geometric lemma which may be viewed as a corrected version of \thref{Z08Q7917}. Moreover, we will give an explicit family of domains which demonstrate that our geometric lemma is sharp. This family will also serve as a counter-example to \thref{Z08Q7917}. 

An essential component to the geometric lemma that is not present in the assumptions of \thref{Z08Q7917} is a uniform bound on the boundary regularity of \(\partial \Omega\). To our knowledge, there is no `standard' definition of such a bound, so we will begin by concretely specifying what is meant by this.

Let \(\Omega\) be an open subset of \(\R^n\) with \(C^1\) boundary. For each \(x\in \partial\Omega\), let 
$$ \Pi_x(y):= (y-x) - ((y-x)\cdot \nu(x))\nu(x),$$
which is the projection of \(\R^n\) onto \(T_x\partial \Omega\), that is, the tangent plane of \(\partial \Omega\) at \(x\). We have made a translation in the definition of \(\Pi_x\), so that \(\Pi_x(x)=0\). For simplicity, we will often identify  \(T_x\partial \Omega\) with \(\R^{n-1}\). 

\begin{defn} \thlabel{deZqwlaB}
Let \(\Omega\) be an open subset of \(\R^n\) with \(C^\alpha\) boundary, with~\(\alpha>1\), 
and let~\(\rho\), \(M>0\). We say that \(\partial \Omega \in C^\alpha_{M,\rho}\) if, for all \(x\in \partial \Omega\), there exists \(\psi^{(x)} : B_\rho^{n-1}\to\R\) such that \(\psi^{(x)}\in C^\alpha (B_\rho^{n-1} )\), \(\| \psi^{(x)}\|_{ C^\alpha (B_\rho^{n-1})  } \leqslant M \), and \begin{align*}
B_\rho (x) \cap \Omega = \big\{ y \in B_\rho(x) \text{ s.t. }  y\cdot \nu(x) < \psi^{(x)} (\Pi_x (y))  \big\} .
\end{align*}
\end{defn}

We remark that, if~$\alpha\in \Z$ in Definition~\ref{deZqwlaB},
the notation~$C^\alpha_{M,\rho}$ means~$C^{\alpha-1,1}_{M,\rho}$.

We now give the geometric lemma, the main result of this section.

\begin{thm} \thlabel{m9q9X2hE}
Let \(\Omega \subset \R^n\) be an open bounded set with \(\partial \Omega \in C^\alpha_{M,\rho}\), with~\(\alpha>1\), 
for some~\(M>0\) and~\(\rho \in (0,1/4)\). Moreover, suppose that \(B_r \subset \Omega \subset B_R\) with~\(1/2 \leqslant r\leqslant R \leqslant 2\).
Let~\(e\in \Sph^{n-1}\), denote by~\(\pi_\lambda\) the critical hyperplane with respect to~\(e\), and suppose that~\(\dist(0,\pi_\lambda) \leqslant 1/ 8\).

Then, \begin{align}
\Big| \big\{ x\in \Omega \triangle \Omega ' \text{ s.t. } \dist(x,\pi_\lambda) \leqslant \gamma \big\} \Big| \leqslant C\gamma (R-r)^{1-\frac 1 \alpha } \label{mKBpmgxs}
\end{align} for all \(0<\gamma \leqslant 1/4\). The constant \(C\) depends only on \(n\), \(M\), \(\rho\), and \(\alpha\). 
\end{thm}

In order to prove Theorem~\ref{m9q9X2hE},
we require two preliminary lemmata. The first lemma gives an elementary estimate of the left hand side of~\eqref{mKBpmgxs} in terms of~\(R-r\) and an error term involving~\(\dist (0,\pi_\lambda)\).

\begin{lem} \thlabel{o1pUKCHX}
Let \(\Omega \subset \R^n\) be an open bounded set with \(C^1\) boundary, and \(B_r \subset \Omega \subset B_R\) with~\(1/2 \leqslant r\leqslant R \leqslant 2\). 
Let~\(e\in \Sph^{n-1}\), denote by~\(\pi_\lambda\) the critical hyperplane with respect to~\(e\), and suppose that~\(\dist(0,\pi_\lambda) \leqslant 1/ 8\).

Then, \begin{align*}
\Big| \big\{ x\in \Omega \triangle \Omega ' \text{ s.t. } \vert x_1 -\lambda \vert \leqslant \gamma \big\} \Big| \leqslant C \gamma\big ( R-r + \gamma \vert \lambda \vert \big ) 
\end{align*} for all \(0<\gamma \leqslant 1/4\). The constant \(C\) depends only on \(n\). 
\end{lem}

\begin{proof}
First, observe that \begin{align}
\Omega \triangle \Omega ' \subset ( B_R \cup B_R' ) \setminus (B_r \cap B_r') \label{M3UmmPsi}
\end{align} where \(B_\rho '\) is the reflection of~$B_\rho$ across the critical hyperplane.
Indeed, if \(x\in \Omega \cup \Omega'\) then \(x\in \Omega \subset B_R\) or \(x\in \Omega'\subset B_R'\), so \(x\in B_R\cup B_R'\). Moreover, if \(x\in B_r\cap B_r'\) then \(x\in B_r\subset \Omega\) and \(x\in B_r'\subset \Omega'\), so~\(B_r\cap B_r'\subset \Omega \cap \Omega'\). Then~\eqref{M3UmmPsi} follows from the the fact that~\(\Omega \triangle \Omega' = (\Omega \cup \Omega' ) \setminus (\Omega \cap \Omega' )\). 

From~\eqref{M3UmmPsi}, it immediately follows that \begin{align*}
\Big| \big\{ x\in \Omega \triangle \Omega ' \text{ s.t. } \vert x_1-\lambda \vert \leqslant \gamma \big\} \Big| &\leqslant \Big| \big\{ x\in  ( B_R \cup B_R' ) \setminus (B_r \cap B_r')  \text{ s.t. } \vert x_1-\lambda \vert \leqslant \gamma \big\} \Big| .
\end{align*} Without loss of generality, we may assume that \(\lambda >0\). Then \begin{align*}
\Big| \big\{ x\in  ( B_R \cup B_R' ) &\setminus (B_r \cap B_r')  \text{ s.t. } \vert x_1-\lambda \vert \leqslant \gamma \big\} \Big|  \\
&= 2 \Big| \big\{ x\in   B_R\setminus B_r'  \text{ s.t. } \lambda-\gamma < x_1<\lambda \big\} \Big| \\
&= 2 \Big ( \Big| B_R \cap \{\lambda-\gamma < x_1<\lambda\big\}  \Big| - \Big| B_r' \cap \big\{\lambda-\gamma < x_1<\lambda\big\}  \Big|  \Big ) .
\end{align*}
Hence, via the co-area formula, we obtain that \begin{align*}
\Big| \big\{ x\in  ( B_R \cup B_R' ) &\setminus (B_r \cap B_r')  \text{ s.t. } \vert x_1-\lambda \vert \leqslant \gamma \big\} \Big| \\
&= 2 \int_{\lambda-\gamma}^\lambda \left [ \mathcal H^{n-1} \left( B_{ \sqrt{R^2-t^2}}^{n-1} \right) -  \mathcal H^{n-1} \left( B_{\sqrt{r^2-(2\lambda - t)^2}}^{n-1} \right)  \right] \dd t \\
&= 2 \omega_{n-1} \int_{\lambda-\gamma}^\lambda \left[ 
\big (R^2 - t^2 \big )^{\frac{n-1}2 } - \big (r^2 - (2\lambda - t )^2 \big )^{\frac{n-1}2 }\right] \dd t . 
\end{align*} 

Next, if \(0<a<\tau < b\) then \begin{align*}
\bigg \vert \frac{\dd }{\dd \tau } \tau^{\frac{n-1}2} \bigg \vert &= \frac{n-1} 2 \tau^{\frac{n-3}2 } \leqslant \frac{n-1} 2\begin{cases}
a^{-1/2}, &\text{if } n=2 \\
b^{\frac{n-3}2 }, &\text{if } n \geqslant 3. 
\end{cases}
\end{align*} Since \(\lambda>0\) and \(2\lambda - t \in (\lambda,\lambda+\gamma) \subset (0,3/8)\), we have that \begin{align*}
C^{-1} \leqslant r^2 - (2\lambda - t )^2 \leqslant R^2 - t^2 \leqslant C \qquad \text{ for all } t\in (\lambda-\gamma , \lambda)
\end{align*} with \(C>0\) a universal constant. Hence, \begin{align*}
\Big \vert \big (R^2 - t^2 \big )^{\frac{n-1}2 } - \big (r^2 - (2\lambda - t )^2 \big )^{\frac{n-1}2 }  \Big \vert &\leqslant C \big \vert \big (R^2 - t^2\big ) - \big ( r^2- (2\lambda - t )^2  \big ) \big \vert\\
&\leqslant C \big ( R-r +  \lambda (\lambda - t ) \big ). 
\end{align*} 

Gathering these pieces of information, we conclude that \begin{align*}
\Big| \big\{ x\in \Omega \triangle \Omega ' \text{ s.t. } \vert x_1-\lambda \vert \leqslant \gamma \big\} \Big| &\leqslant C \int_{\lambda-\gamma}^\lambda  \big(R-r +  \lambda (\lambda - t ) \big) \dd t \leqslant C \gamma\big ( R-r + \gamma \lambda \big ),
\end{align*} as required. 
\end{proof}

The purpose of the second lemma will be to reduce the proof of \thref{m9q9X2hE} to the case of graphs of functions. 

\begin{lem} \thlabel{GNZAnsDn}
Let \(\Omega \subset \R^n\) be an open bounded set with \(\partial \Omega \in C^\alpha_{M,\rho}\), with~\(\alpha>1\), 
for some~\(M>0\) and~\(\rho \in (0,1/4)\).

Then, there exists~$\epsilon_0\in(0,1)$ such that, for all~$\epsilon\in(0,\epsilon_0]$, it holds that if~\(B_{1-\varepsilon/2} \subset \Omega \subset B_{1+\varepsilon/2}\) then
there exists \(\psi_\epsilon:=\psi : B_{3/4}^{n-1} \to (0,+\infty)\)
such that~\(\psi \in C^\alpha (B_{3/4}^{n-1})\), \begin{align}
\big \|\psi-\sqrt{1-\vert x' \vert^2} \big \|_{C^1(B_{3/4}^{n-1})} \leqslant C \varepsilon^{1-\frac1\alpha}, \label{WR8AxQVQ}
\end{align}  and \begin{align}
\Omega \cap \big ( B_{3/4}^{n-1} \times (0,+\infty) \big ) = \{ x\in \R^n \text{ s.t. } 0<x_n < \psi(x') \}  . \label{Gug035Xo}
\end{align} The constant \(C\) depends only on \(n\), \(\alpha\), \(\rho\), and \(M\). 
\end{lem}

\begin{proof}
Let \(x_\star \in \partial \Omega \cap \big ( B_{3/4}^{n-1} \times (0,+\infty) \big ) \) and~\(\psi^{(x_\star)}\) be the function given in \thref{deZqwlaB} corresponding to the point~\(x_\star\). Next, let~\(A\) be a rigid motion such that~\(Ax_\star = 0\),  \(T_{x_\star}\partial\Omega\) is mapped t~ \(\R^{n-1}\times \{0\}\), and~\(\nu(x_\star) \) is mapped to~\(e_n\). To avoid any confusion, we will always use the variable~\(x\) to denote points in the original (unrotated) coordinates and~\(y\) to denote points in the new rotated coordinates, i.e. \(y=Ax\). By \thref{deZqwlaB}, in the \(y\) coordinates, \begin{align*}
y = (y' , \psi^{(x_\star)} (y') ) \qquad \text{ for all } y\in B_\rho \cap \partial \Omega
\end{align*} where \(y'=(y_1,\dots,y_{n-1})\).

Next, if \(\partial \Omega\) is contained in the closure of \(B_{1+\varepsilon/2} \setminus B_{1-\varepsilon/2}\), we have that \begin{align*}
\bigg ( 1-\frac \varepsilon 2\bigg )^2-\vert y ' \vert^2\leqslant \big( \psi^{(x_\star)} (y') \big )^2 \leqslant \bigg ( 1+\frac \varepsilon 2\bigg )^2-\vert y ' \vert^2
\end{align*} for all \(y' \in B^{n-1}_\rho \). 

Additionally, by Bernoulli's inequality, we have that \begin{align*}
0\leqslant \sqrt{\bigg ( 1+\frac \varepsilon 2\bigg )^2-\vert y ' \vert^2} -\sqrt{1-\vert y ' \vert^2} &\leqslant \frac 1 {2\sqrt{1-\vert y ' \vert^2}} \bigg ( \varepsilon + \frac{\varepsilon^2}4 \bigg ) \leqslant C \varepsilon
\end{align*} with \(C=C(\rho)>0\). Similarly, we also have that \begin{align*}
-C\varepsilon \leqslant \sqrt{\bigg ( 1-\frac \varepsilon 2\bigg )^2-\vert y ' \vert^2} -\sqrt{1-\vert y ' \vert^2} \leqslant 0 . 
\end{align*} Hence, it follows that \begin{align*}
\big \vert | \psi^{(x_\star)} (y')|  -\sqrt{1-\vert y ' \vert^2}  \big \vert &\leqslant C \varepsilon \qquad \text{for all } y' \in B^{n-1}_\rho . 
\end{align*} 

Thus, by interpolation, we have that \begin{align*}
\big \| &|\psi^{(x_\star)} (y')|  -\sqrt{1-\vert y ' \vert^2}  \big \|_{C^1( B^{n-1}_\rho)} \\
 &\leqslant C \big \| |\psi^{(x_\star)} (y')|  -\sqrt{1-\vert y ' \vert^2}  \big \|_{C^\alpha( B^{n-1}_\rho)}\big \| |\psi^{(x_\star)} (y')|  -\sqrt{1-\vert y ' \vert^2}  \big \|_{L^\infty( B^{n-1}_\rho)}^{1-\frac 1 \alpha } \\
 &\leqslant C \varepsilon^{1-\frac1{\alpha}}
\end{align*} using also that \(\partial \Omega \in C^\alpha_{M,\rho}\). 

Note that we have left the \(y\) variable in the equations above to emphasise that we are still using the rotated \(y\) coordinates. Now, returning to the original \(x\) coordinates, we observe that, by the above computation, we have that \(\partial \Omega\) is uniformly close to \(\partial B_1\) in the \(C^1\) sense, so \begin{align*}
\nu(x) \cdot e_n \geqslant C >0 \qquad \text{ for all } x\in \partial \Omega \cap \big ( B_{3/4}^{n-1} \times (0,+\infty) \big )  
\end{align*} provided that \(\varepsilon\) is sufficiently small. Thus, it follows that \(\partial \Omega\) is given by a graph with respect to the~\(e_n\) direction, that is, the claim in~\eqref{Gug035Xo} holds for some \(\psi : B_{3/4}^{n-1} \to (0,+\infty)\). We can see that~\(\psi \in C^\alpha (B_{3/4}^{n-1})\) since~\(\partial \Omega\) is~\(C^\alpha\) and we obtain the claim in~\eqref{WR8AxQVQ} by an identical interpolation argument to the one above. 
\end{proof}

We may now give the proof of \thref{m9q9X2hE}.

\begin{proof}[Proof of \thref{m9q9X2hE}] Without loss of generality, we may assume that \(e=e_1\) and \(\lambda>0\). By~\thref{o1pUKCHX}, it is enough to prove that \begin{align}
\lambda\leqslant C(n,M,\rho,\alpha)(R-r)^{1-\frac 1 \alpha } . \label{EtWD36ww}
\end{align}
 Since \(\lambda \leqslant 2\),  if there exists \(C=C(n,M,\rho,\alpha)>0\) such that \(R-r\geqslant C \) then we are done, so we may assume that \(R-r = \varepsilon\) with \(\varepsilon\) arbitrarily small. Moreover, by rescaling, we may further assume that~\(r=1-\varepsilon/2\) and~\(R=1+\varepsilon/2\). Furthermore, let~\(\psi :B_{3/4}^{n-1}\) be the function given by \thref{GNZAnsDn}. 
 
First, let us consider Case 1 of the method of moving planes. In this case, we obtain a point~\(p=(p_1,\dots,p_n)\in (\partial \Omega \cap \partial \Omega'_\lambda) \setminus \pi_\lambda \) (recall that, in this notation, \(\Omega'_\lambda = \Omega' \cap H_\lambda\), so \(p_1<\lambda\)). Hence,
we have that~\( r^2 \leqslant \vert p \vert^2 \leqslant (r+\varepsilon)^2\) and~\( r^2 \leqslant \vert Q_{\pi_\lambda} (p) \vert^2 \leqslant (r+\varepsilon)^2\) (recall that~$Q_{\pi_\lambda}$ reflects a point across~\(\pi_\lambda\)), from which it follows that \begin{align*}
\lambda (\lambda - p_1 ) \leqslant  \varepsilon (2r+\varepsilon) \leqslant C \varepsilon . 
\end{align*} If \(\lambda - p_1 \geqslant 1/4\) then we are done, so we may assume that \(\lambda - p_1 < 1/4\). 

Moreover, by rotating with respect to~\((x_2,\dots,x_n)\), we may assume without loss of generality that~\(p_2=\dots=p_{n-1} =0\) and~\(p_n>0\). In particular, this implies that \(p'=(p_1,\dots,p_{n-1} ) \in B_{3/4}^{n-1}\). 

Now, on one hand, if \(\psi_\lambda(x') := \psi(x') - \psi (2\lambda-x_1,x_2,\dots,x_{n-1}) \) then \(\psi_\lambda \geqslant 0\) for \(x'\in B_{3/4}^{n-1} \cap \{x_1<\lambda \}\) and \(\psi_\lambda(p')=0\), so \(\partial_1\psi_\lambda (p')=0\). On the other hand, by \thref{GNZAnsDn}, if \(x'' = (x_2,\dots,x_{n-1})\), \begin{align*}
\partial_1 \psi_\lambda(p') &= \partial_1 \big( \psi_\lambda - \sqrt{1-\vert x ' \vert^2}+ \sqrt{1-(2\lambda-x_1)^2 -\vert x''\vert^2 } \big ) (p') \\
&\qquad + \partial_1 \big( \sqrt{1-\vert x ' \vert^2}- \sqrt{1-(2\lambda-x_1)^2 -\vert x''\vert^2 } \big ) (p') \\
&\leqslant C \varepsilon^{1-\frac 1 \alpha}  - \bigg ( \frac{p_1}{\sqrt{1-p_1^2}} + \frac{2\lambda-p_1}{\sqrt{1-(2\lambda-p_1)^2}} \bigg ).
\end{align*} Hence, we have that \begin{equation}\label{elemcalc0}
\frac{p_1}{\sqrt{1-p_1^2}} + \frac{2\lambda-p_1}{\sqrt{1-(2\lambda-p_1)^2}}  \leqslant C \varepsilon^{1-\frac 1 \alpha} .
\end{equation} 

Moreover, we claim that \begin{equation}\label{elemcalc}
\frac{p_1}{\sqrt{1-p_1^2}} + \frac{2\lambda-p_1}{\sqrt{1-(2\lambda-p_1)^2}}  \geqslant\frac{2 \lambda}{\sqrt{1-\lambda^2}} .
\end{equation} 
To prove this, we consider the function
$$ f(t):=\frac{\lambda-t}{\sqrt{1-(\lambda-t)^2}} + \frac{\lambda+t}{\sqrt{1-(\lambda+t)^2}}$$
for~$t\in\left(-\frac12,\frac12\right)$. Notice that~$f$ is even, and therefore we restrict our analysis to the interval~$\left[0,\frac12\right)$.

By a direct calculation,
\begin{eqnarray*}
f'(t)&=& 
-\frac{(\lambda - t)^2}{(1 - (\lambda - t)^2)^{3/2}} - \frac{1}{\sqrt{1 - (\lambda - t)^2}} + \frac1{\sqrt{1 - (\lambda + t)^2}} 
+ \frac{(\lambda + t)^2}{(1 - (\lambda + t)^2)^{3/2}}\\
&=&-\frac{1}{(1 - (\lambda - t)^2)^{3/2} }+\frac{1}{(1 - (\lambda + t)^2)^{3/2}} .
\end{eqnarray*}

Now we define~$\phi(\tau):=(1-\tau)^{-3/2}$. Since~$(\lambda-t)^2\le(\lambda+t)^2\le17/32< 1$, we can write
\begin{eqnarray*}
f'(t)=\int_{ (\lambda - t)^2}^{(\lambda + t)^2} \phi'(\tau)\,d\tau=
\frac{3}{2} \int_{ (\lambda - t)^2}^{(\lambda + t)^2} (1-\tau)^{-5/2}\,d\tau\ge0.
\end{eqnarray*}
Therefore, for all~$t\in\left(0,\frac12\right)$ (and thus for all~$t\in\left(-\frac12,\frac12\right)$), we have that
$$ f(t)\ge f(0)=\frac{2\lambda}{\sqrt{1-\lambda^2}}.$$
Hence, setting~$t=\lambda-p_1\in\left(-\frac12,\frac12\right)$, we obtain the desired inequality in~\eqref{elemcalc}.

In particular, from \eqref{elemcalc} we have that
$$ \frac{p_1}{\sqrt{1-p_1^2}} + \frac{2\lambda-p_1}{\sqrt{1-(2\lambda-p_1)^2}}  \geqslant C \lambda,$$
for some~$C>0$.
Putting together this and~\eqref{elemcalc0}, we deduce that
the claim in~\eqref{EtWD36ww} holds. 

For Case 2 of the method of moving planes, we obtain a point~\(p\in \partial\Omega \cap \pi_\lambda\) at which~\(\partial \Omega\) is orthogonal to~\(\pi_\lambda\). In this case, we have that \(\partial_1\psi_\lambda(p') \geqslant 0\). {F}rom here the proof is identical to Case~1. 
\end{proof}

We will now give an explicit family of domains which demonstrate that the result obtained in \thref{m9q9X2hE}
is sharp. This also serves as a counter-example to \thref{Z08Q7917}. 

\begin{thm}\label{example61}
There exists a smooth \(1\)-parameter family \(\{\Omega_\varepsilon\}\) of open bounded subsets of \(\R^2\) such that \begin{itemize}
\item \(\partial \Omega_\varepsilon \in C^\alpha_{M,1/8}\), with \(\alpha>1\), for some \(M>0\) independent of \(\varepsilon\);
\item \(B_{1-C_1\varepsilon} \subset \Omega_\varepsilon \subset B_{1+C_1\varepsilon}\) for a universal constant \(C_1>0\);
\item if \(\pi_\lambda\) is
the critical hyperplane with respect to~\(e_1\) and~\(\gamma \in (0,1/4)\), then  \begin{align}
\Big| \big\{ x\in \Omega_\varepsilon \triangle \Omega_\varepsilon ' \text{ s.t. } \dist(x,\pi_\lambda) \leqslant \gamma\big\} \Big| \geqslant C_2\varepsilon^{1-\frac 1 \alpha } \label{EtWD36ab}
\end{align} as \(\varepsilon \to 0^+\). The constant \(C_2>0\) depends on \(n\) and \(\gamma\). 
\end{itemize} 
\end{thm}

\begin{figure}[h]
\centering 
\includegraphics{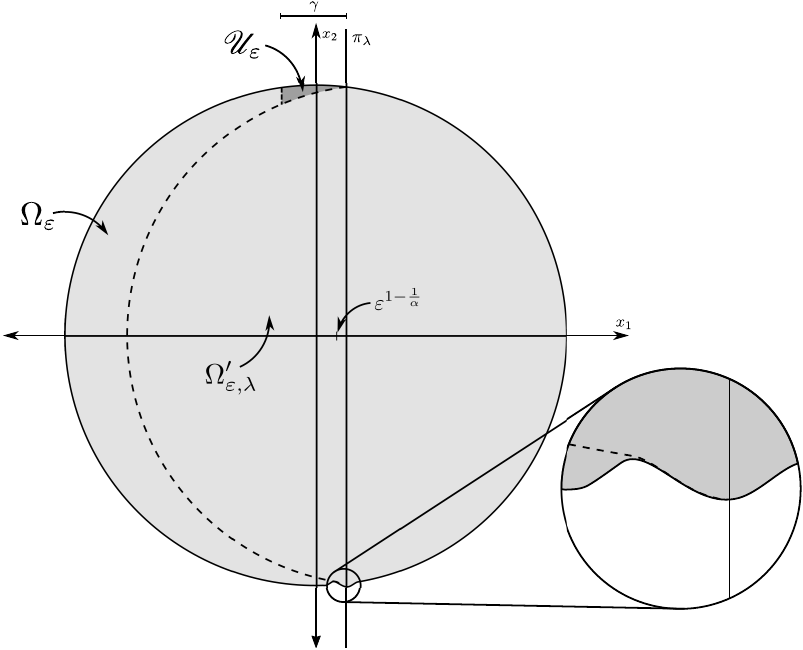}
\caption{Diagram of \(\Omega_\varepsilon\) in Theorem~\ref{example61}.}
\label{Fig1}
\end{figure}

\begin{proof}
We will define \(\Omega_\varepsilon\) by specifying its boundary, see Figure~\ref{Fig1}. 
More precisely, in the region~\(\R^2 \setminus \big ( (0,1/2) \times (-3/2,-1/2 ) \big )  \), let~\(\partial \Omega_\varepsilon = \partial B_1\).
For the definition of~\(\partial \Omega_\varepsilon\) in~\( (0,1/2) \times (-3/2,-1/2 ) \), let~\(\eta~\in C^\infty_0(\R)\) be such that \begin{align*}
\begin{cases}
\eta(\tau) = 2 \tau &\text{in }  ( - 1/4, 1/4 ), \\
\eta =0 &\text{in } \R \setminus (-1,1), \\
\eta(-\tau) =-\eta(\tau) &\text{in } \R, \\
\vert \eta \vert \leqslant 1  &\text{in } \R. 
\end{cases}
\end{align*} 
Let also
\begin{eqnarray*}
&&\eta_\varepsilon(\tau) :=
\varepsilon\eta  \left(\frac{\tau -\varepsilon^{1-\frac 1 \alpha}  } {\varepsilon^{1/\alpha}}  \right) 
\qquad
{\mbox{and}} \qquad
\psi_\varepsilon(\tau):= -\sqrt{1-\tau^2} - \eta_\varepsilon(\tau)
\end{eqnarray*}
and define the remaining portion of \(\partial \Omega_\varepsilon \) by \begin{align*}
\partial \Omega_\varepsilon \cap \Big ( (0,1/2) \times (-3/2,-1/2 ) \Big ) = \big\{ (x,\psi_\varepsilon(x) )  \text{ s.t. } x\in (0,1/2) \big\} .
\end{align*} 

We observe that \(\partial \Omega_\varepsilon\) is smooth and that \begin{align*}
\tau^2 + (\psi_\varepsilon(\tau))^2  &= 1+ 2   \eta_\varepsilon(\tau)  \sqrt{1-\tau^2}  +( \eta_\varepsilon(\tau) )^2 \leqslant 1+ C \varepsilon.
\end{align*} Hence, by Bernoulli's inequality,\begin{align*}
\sqrt{\tau^2 + \psi_\varepsilon(\tau)^2}-1  &\leqslant C\varepsilon ,
\end{align*} and similarly, we can show that \( \sqrt{\tau^2 + \psi_\varepsilon(\tau)^2}-1  \geqslant -C\varepsilon\), so \(B_{1-C\varepsilon} \subset \Omega_\varepsilon \subset  B_{1+C\varepsilon}\). Moreover, we have that \(\| \psi_\varepsilon\|_{C^\alpha((0,1/2))}\leqslant M\) for some \(M>0\) sufficiently large and independent of \(\varepsilon\), so \(\partial \Omega_\varepsilon \in C^\alpha_{M,1/8}\).

Now all that is left to be shown is~\eqref{EtWD36ab}.
First, we claim that the critical parameter~\(\lambda_\varepsilon\) satisfies \begin{align}
\lambda_\varepsilon \geqslant \varepsilon^{1-\frac1\alpha}.  \label{b6gP569L}
\end{align}
To prove this, for all \(\mu \in \R\), we define~\(\psi_{\varepsilon,\mu}(\tau) := \psi_\varepsilon(\tau) - \psi_\varepsilon(2\mu-\tau)\). 
Furthermore, we
consider
$$ \tau_\varepsilon:= \varepsilon^{1-\frac1 \alpha}-\frac14\varepsilon^{\frac1\alpha} \qquad
{\mbox{and}} \qquad \mu_\varepsilon := \varepsilon^{1-\frac1\alpha}.$$ We have that
\begin{eqnarray*}
\psi_{\varepsilon,\mu_\varepsilon}(\tau_\varepsilon) &=& 
- \Big ( \sqrt{1-\tau_\varepsilon^2}  - \sqrt{1-(2\mu_\varepsilon-\tau_\varepsilon) ^2}  \Big ) - \varepsilon \left( 
\eta\left(\frac{\tau_\varepsilon-\varepsilon^{1-\frac1\alpha}}{\varepsilon^{1/\alpha}}\right)-\eta\left(
\frac{2\mu_\varepsilon-\tau_\varepsilon-\varepsilon^{1-\frac1\alpha}}{\varepsilon^{1/\alpha}}\right)\right)
\\&=&- \Big ( \sqrt{1-\tau_\varepsilon^2}  - \sqrt{1-(2\mu_\varepsilon-\tau_\varepsilon) ^2}  \Big ) - 2\varepsilon \bigg ( \frac{\tau_\varepsilon-\varepsilon^{1-\frac1\alpha}}{\varepsilon^{1/\alpha}}-\frac{2\mu_\varepsilon-\tau_\varepsilon-\varepsilon^{1-\frac1\alpha}}{\varepsilon^{1/\alpha}}\bigg ) \\
&=& - \frac{4\mu_\varepsilon(\mu_\varepsilon-\tau_\varepsilon)}{ \sqrt{1-\tau_\varepsilon^2}  +\sqrt{1-(2\mu_\varepsilon-\tau_\varepsilon) ^2}} 
+4\varepsilon^{1-\frac1\alpha} \big (  \mu_\varepsilon- \tau_\varepsilon  \big )  \\
&=&- 4\mu_\varepsilon(\mu_\varepsilon-\tau_\varepsilon)\left(\frac12+o(1)\right)
+4\varepsilon^{1-\frac1\alpha} \big (  \mu_\varepsilon- \tau_\varepsilon  \big ) \\
&=&2\varepsilon^{1-\frac1\alpha} \big (  \mu_\varepsilon- \tau_\varepsilon  \big ) +o\Big(\varepsilon^{1-\frac1\alpha} \big (  \mu_\varepsilon- \tau_\varepsilon  \big )\Big)\\
&=&\frac12\varepsilon+o(\varepsilon)\\
&\ge&\frac14\varepsilon ,  
\end{eqnarray*}
as soon as~$\varepsilon$ is sufficiently small.

Hence, we have that~\(\psi_{\varepsilon,\mu_\varepsilon}(\tau_\varepsilon) >0\), which implies that the reflected region of \(\Omega_{\varepsilon,\mu}\) must have left the region~\(\Omega_\varepsilon\) by the time we have reached~\(\mu=\mu_\varepsilon\), so we conclude that the critical time~\(\lambda_\varepsilon\) satisfies~\(\lambda_\varepsilon > \mu_\varepsilon\),
which gives~\eqref{b6gP569L}.

In fact, since \(\eta\) is zero outside of \((-1,1)\), it follows that \begin{align*}
\lambda_\varepsilon \leqslant C \varepsilon^{1-\frac 1\alpha},
\end{align*}
so, in particular, \(\lambda_\varepsilon \to 0\) as \(\varepsilon \to 0^+\). 

Finally, to complete the proof of~\eqref{EtWD36ab}, we will show that \begin{equation}\label{9u43c7bv9843t69}
\Big| \big\{ x\in \Omega_\varepsilon \triangle \Omega_\varepsilon ' \text{ s.t. } \dist(x,\pi_\lambda) \leqslant \gamma\big\} \Big| \geqslant C \lambda_\varepsilon. 
\end{equation} Indeed, let \(\mathscr U_\varepsilon :=  \big ( \big ( \Omega_\varepsilon \cap \{ \lambda-\gamma <x_1<\lambda \} \big ) \setminus \Omega_{\varepsilon,\lambda}'  \big ) \cap \{x_2>0\}\), see Figure~\ref{Fig1}. Observe that \begin{align*}
\Big| \big ( \Omega_\varepsilon \cap \{ \lambda-\gamma <x_1<\lambda \} \big ) \setminus \Omega_{\varepsilon,\lambda}' \Big|  \geqslant \vert \mathscr U_\varepsilon \vert . 
\end{align*} Moreover, using that \(\lambda_\varepsilon \to 0\), we have that  \begin{eqnarray*}
\vert \mathscr U_\varepsilon \vert &=& \int_{\lambda_\varepsilon-\gamma}^{\lambda_\varepsilon} \Big(\sqrt{1-\tau^2} - \sqrt{1-(2\lambda_\varepsilon-\tau)^2}\Big) \dd \tau \\
&=& \int_{\lambda_\varepsilon-\gamma}^{\lambda_\varepsilon} 
\frac{4\lambda_\varepsilon(\lambda_\varepsilon-\gamma)}{
\sqrt{1-\tau^2} + \sqrt{1-(2\lambda_\varepsilon-\tau)^2}} \dd \tau \\
&\ge&
4\lambda_\varepsilon
\int_{\lambda_\varepsilon-\gamma}^{\lambda_\varepsilon-\frac\gamma2} 
\frac{\lambda_\varepsilon-\gamma}{2} \dd \tau \\
&\ge&\frac{\gamma^2}2  \lambda_\varepsilon ,
\end{eqnarray*} which establishes~\eqref{9u43c7bv9843t69}, and thus completes the proof of Theorem~\ref{example61}.
\end{proof}


\subsection{An application of \thref{m9q9X2hE}: improvement of the exponent in \thref{oAZAv7vy}}\label{subsec:new improvement exponent}

We will now show how to use \thref{m9q9X2hE} to improve the stability exponent in \thref{m9q9X2hE}.
The key step is to use \thref{m9q9X2hE} to achieve the following modified version of \thref{7TQmUHhl}.

\begin{prop}\thlabel{Prop:new per improvement}
Let \(\Omega\) be an open bounded set with~\(C^1\) boundary and satisfying the uniform interior ball condition with radius~\(r_\Omega >0\)
and~$G$ be an open bounded set with~\(C^1\) boundary
such that~$\Omega=G+B_R$
for some~$R>0$. Let~\(f \in C^{0,1}_{\mathrm{loc}}(\R)\) be such that \(f(0)\geqslant 0\)
and \(u\) satisfies~\eqref{problem00} in the weak sense.

For \(e\in \Sph^{n-1}\), let \(\Omega '\) denote the reflection of \(\Omega\) with respect to the critical hyperplane \(\pi_\lambda\).
	
In addition, suppose that \(\partial \Omega \in C^\alpha_{M,\rho}\), with~\(\alpha>1\), 
	for some~\(M>0\) and~\(\rho \in (0,1/4)\), and that 
	\begin{equation}\label{eq:hp per improvement con r e R}
	B_{\rho_i} \subset \Omega \subset B_{\rho_e} \quad  \text{ with } \quad \frac12 \leqslant \rho_i \leqslant \rho_e \leqslant 2    
    \end{equation}
	and
\begin{equation}\label{eq:hp NEWNEW per improvement con r e R}
\rho_e - \rho_i \ge 4^{\frac{ \alpha (s+2)}{ (\alpha -1)(s+1) }} [u]_{\partial G}^{\frac{ \alpha }{ (\alpha -1)(s+1) }} .
\end{equation}	
	
	Then, 
	\begin{align*}
		\vert \Omega \triangle \Omega' \vert \leqslant \widetilde{C}_\star  (\rho_e - \rho_i)^{\frac{\alpha-1}{ \alpha (s+2)} }  \,[u]_{\partial G}^{\frac{1}{s+2}}, 
	\end{align*}
	where $\widetilde{C}_\star$ is some explicit constant depending on~$n$, $s$,
	$\alpha$, $M$, $\rho$, $\diam \Omega$, $R$, $[f]_{C^{0,1}([0,\|u\|_{L^\infty(\Omega)}])}$, and~$\|u\|_{L_s(\R^n)}$.
\end{prop}

\begin{proof}
The proof of Proposition~\ref{Prop:new per improvement} 
is a suitable modification of the one of \thref{7TQmUHhl}. More precisely, we recall~\eqref{ido4985vb6tdfghwafrikewytoiw}
and we use that~$f(0) \ge 0$
to obtain that
\begin{equation}\label{d9i43765v943675y90378659403869yuhgijfjhgflkdjvk7yt8ur}\begin{split}
\Big| \big\{ x \in (\Omega\cap H_\lambda') \setminus \Omega_\lambda' \text{ s.t. } x_1 \delta_{\partial \Omega}^s(x) >\gamma \big\} \Big|
\leqslant\;& \frac {CC_\ast^{-1}\big( f(0)+\|u\|_{L_s(\R^n)}\big)^{-1}} \gamma [u]_{\partial G}\\
\leqslant\;& \frac {CC_\ast^{-1} \|u\|_{L_s(\R^n)}^{-1}} \gamma [u]_{\partial G},
\end{split}\end{equation}
where \(C_\ast\) is as in \thref{zAPw0npM}.

Furthermore, given~$\gamma$, $\beta\in\left(0,\frac14\right]$, we see that
	\begin{eqnarray*}&&
		\Big| \big\{ x \in (\Omega\cap H_\lambda') \setminus \Omega_\lambda' \text{ s.t. } x_1 \delta_{\partial \Omega}^s(x) \leqslant \gamma \big\} \Big| \\&=&
		\Big| \big\{ x \in (\Omega\cap H_\lambda') \setminus \Omega_\lambda' \text{ s.t. } x_1 \delta_{\partial \Omega}^s(x) \leqslant \gamma, x_1< \beta \big\} \Big| \\&&\qquad + \Big| \big\{ x \in (\Omega \cap H_\lambda')
		\setminus \Omega_\lambda' \text{ s.t. } x_1 \delta_{\partial \Omega}^s(x) \leqslant \gamma , x_1 \geqslant \beta \big\} \Big| \\
		&\leqslant& \Big| \big\{ x \in \Omega^+ \text{ s.t. } x_1 < \beta \big\} \Big| +  \left| \left\{ x \in \Omega \text{ s.t. } \delta_{\partial \Omega}(x) \leqslant \left( \frac{\gamma}{\beta} \right)^{\frac 1{s}} \right\} \right| . 
	\end{eqnarray*}
Now, using \thref{m9q9X2hE} we have that
	\begin{align*}
		\Big| \big\{ x \in \Omega^+ \text{ s.t. } x_1< \beta \big\} \Big| &\leqslant 	C \,   (\rho_e - \rho_i)^{1 - \frac{1}{\alpha}}  \beta  ,
	\end{align*}
	for some~$C>0$ depending only on~$n$, $M$, $\rho$ and~$\alpha$.
	
Also, using \cite[Lemma 5.2]{MR4577340} and \cite{MR483992}, we obtain that \begin{align*}
		\left| \left\{ x \in \Omega \text{ s.t. } \delta_{\partial \Omega}(x) \leqslant \left( \frac{\gamma}{\beta} \right)^{\frac{1}{s}} \right\} \right|  &\leqslant
		\frac{2n \vert \Omega \vert }{r_\Omega} 
		\left( \frac{\gamma}{\beta} \right)^{\frac 1{s}}.
	\end{align*} 
Gathering these pieces of information, we obtain that
		\begin{eqnarray*}
	\Big| \big\{ x \in (\Omega\cap H_\lambda') \setminus \Omega_\lambda' \text{ s.t. } x_1 \delta_{\partial \Omega}^s(x) \leqslant \gamma \big\} \Big| \le
	C  (\rho_e - \rho_i)^{1 - \frac{1}{\alpha}} \beta +\frac{2n \vert \Omega \vert }{r_\Omega} \left( \frac{\gamma}{\beta}\right)^{\frac 1{s}}.
	\end{eqnarray*}	
	
{F}rom this and~\eqref{d9i43765v943675y90378659403869yuhgijfjhgflkdjvk7yt8ur}, we now deduce that
	\begin{eqnarray*}
	\vert \Omega \triangle \Omega '\vert &=& 2 \vert (\Omega\cap H_\lambda') \setminus \Omega'_\lambda \vert
	\\&\leqslant& 
		 CC_\ast^{-1} \|u\|_{L_s(\R^n)}^{-1}  \frac{[u]_{\partial G}}{\gamma} 
		 +  C(\rho_e - \rho_i)^{1 - \frac{1}{\alpha}} \beta + \frac{2n \vert \Omega \vert }{r_\Omega}\left( \frac{\gamma}{\beta}\right)^{\frac 1{s}} \\
		  &\leqslant& 
		 C\left[ \frac{[u]_{\partial G}}{\gamma} 
		 +  (\rho_e - \rho_i)^{1 - \frac{1}{\alpha}} \beta + \left( \frac{\gamma}{\beta}\right)^{\frac 1{s}}\right].
	\end{eqnarray*}
	Minimizing the expression in the last line in the variables $(\beta,\gamma)$ gives 
	$$
	\gamma \beta = \frac{[u]_{\partial G}}{(\rho_e - \rho_i)^{1 - \frac{1}{\alpha}}} 
		\qquad {\mbox{ and }}\qquad
	\beta =   \frac{[u]_{\partial G}^{\frac{1}{s+2}} }{(\rho_e - \rho_i)^{ {\frac{(\alpha -1 )(s+1)}{\alpha (s+2) }}  }},
	$$
	that is,
	$$ \gamma = \frac{[u]_{\partial G}^{\frac{s+1}{s+2}}}{(\rho_e - \rho_i)^{\frac{\alpha-1}{\alpha(s+2)} } } 
		\qquad {\mbox{ and }}\qquad
	\beta =   \frac{[u]_{\partial G}^{\frac{1}{s+2}} }{(\rho_e - \rho_i)^{ {\frac{(\alpha -1 )(s+1)}{\alpha (s+2) }}  }}.
	$$
	Notice that the assumption in~\eqref{eq:hp NEWNEW per improvement con r e R} guarantees that, with these choices,
	$\gamma$, $\beta\in\left(0,\frac14\right]$.
	
Thus, we conclude that	$$
	\vert \Omega \triangle \Omega '\vert \leqslant C \, (\rho_e - \rho_i)^{\frac{\alpha-1}{ \alpha (s+2)} }  \,[u]_{\partial G}^{\frac{1}{s+2}}.$$
This completes the proof of Proposition~\ref{Prop:new per improvement}. 
\end{proof}

Proposition~\ref{Prop:new per improvement} leads to the following statement, which is the counterpart
of Proposition~\ref{k4VRS3Fx}. 

\begin{prop} \thlabel{k4VRS3FxBIS} Let \(\Omega\) be an open bounded set with \(C^1\) boundary and satisfying the uniform interior ball condition with radius \(r_\Omega >0\) and suppose that the critical planes \(\pi_{e_i}\) with respect to the coordinate directions \(e_i\) coincide with \(\{x_i=0\}\) for every \(i=1,\dots, n\). Also, given \(e\in \Sph^{n-1}\), denote by \(\lambda_e\) the critical value associated with~\(e\) as in~\eqref{IAhmMsFq}. 

In addition, suppose that \(\partial \Omega \in C^\alpha_{M,\rho}\), with~\(\alpha>1\), 
	for some~\(M>0\) and~\(\rho \in (0,1/4)\),
	and that the assumptions in~\eqref{eq:hp per improvement con r e R} and~\eqref{eq:hp NEWNEW per improvement con r e R}
are satisfied.
	
Assume that \begin{align*}
[u]^{\frac 1{s+2} }_{\partial G} \leqslant \frac {\vert \Omega \vert } {n \widetilde{C}_\star }  
\end{align*}
where \(\widetilde{C}_\star\) is the constant in Proposition~\ref{Prop:new per improvement}.

Then,
\begin{align*}
\vert \lambda_e \vert \leqslant C \, (\rho_e - \rho_i)^{\frac{\alpha-1}{ \alpha (s+2)} } [u]_{\partial G}^{\frac 1 {s+2} }
\end{align*}
for all \(e\in \Sph^{n-1}\),
where~$C$
is some explicit constant depending on~$n$, $s$,
	$\alpha$, $M$, $\rho$, $\diam \Omega$, $R$, $[f]_{C^{0,1}([0,\|u\|_{L^\infty(\Omega)}])}$, and~$\|u\|_{L_s(\R^n)}$.
\end{prop}

We omit the proof of Proposition~\ref{k4VRS3FxBIS}, since it follows the same line as that of Proposition~\ref{k4VRS3Fx}.

Now, up to a translation, we can suppose that the critical planes \(\pi_{e_i}\) with respect to the coordinate directions \(e_i\) coincide with \(\{x_i=0\}\) for every \(i=1,\dots, n\). Hence, following \cite[Proof of Theorem 1.2]{MR3836150}, we can define
\begin{equation}\label{def:new rhoe and rhoi}
\rho_e := \max_{\partial \Omega}|x|  \qquad {\mbox{ and }}\qquad \rho_i := \min_{\partial \Omega}|x| ,
\end{equation}
and work with $\rho_e -\rho_i$ (which clearly gives an upper bound for $\rho(\Omega)$) to achieve the desired stability estimates.

In the following result, we also assume that
\begin{equation}\label{eq:scaling condition mmm...}
\rho_e=1.
\end{equation}
Notice that such an assumption is always satisfied, up to a dilation.

\begin{thm} \thlabel{thm:improvement}
Let \(G\) be an open bounded subset of~\( \R^n\) and~\(\Omega = G+B_R\) for some~\(R>0\). Furthermore, let~\(\Omega\) and~\(G\) have~\(C^1\) boundary,
and let~\(f \in C^{0,1}_{\mathrm{loc}}(\R)\) be such that~\(f(0)\geqslant 0\). Suppose that~\(u\) satisfies~\eqref{problem00} in the weak sense.

In addition, let \(\partial \Omega \in C^\alpha_{M,\rho}\), with~\(\alpha>1\), 
	for some~\(M>0\) and~\(\rho \in (0,1/4)\). 
	
	Let $\rho_e$ and $\rho_i$ be as in~\eqref{def:new rhoe and rhoi} and~\eqref{eq:scaling condition mmm...}.
	
	Then, 
	\begin{align*}
		\rho(\Omega) \leqslant \rho_e -\rho_i &\leqslant \widetilde{C} [u]_{\partial G}^{\frac{\alpha}{1+\alpha(s+1)}} ,
	\end{align*} 
	where $\widetilde{C}$ is some explicit constant depending 
	on~$n$, $s$,
	$\alpha$, $M$, $\rho$, $\diam \Omega$, $R$, $[f]_{C^{0,1}([0,\|u\|_{L^\infty(\Omega)}])}$, and~$\|u\|_{L_s(\R^n)}$.
	\end{thm}

\begin{proof}
Without loss of generality, we can assume that \eqref{eq:hp NEWNEW per improvement con r e R} is satisfied, as otherwise \thref{thm:improvement} trivially holds.

An inspection of the proof of \thref{oAZAv7vy} shows that $\rho_e$ and $\rho_i$ defined in \eqref{def:new rhoe and rhoi} satisfy
\begin{align*}
 	\rho_e - \rho_i &\leqslant C [u]_{\partial G}^{\frac 1 {s+2}} .
\end{align*}
Thus, if~$[u]_{\partial G} \le (2 C)^{-(s+2)} $,
we obtain that~$\rho_e -\rho_i \le 1/2$,
and therefore, in light of the assumption in~\eqref{eq:scaling condition mmm...},
we have that~$1/2\le\rho_i\le\rho_e\le2$, namely~\eqref{eq:hp per improvement con r e R} is satisfied.

Hence, we can employ Proposition~\ref{k4VRS3FxBIS} and obtain that
\begin{equation}
\label{eq:finalstepperimprovement}
\vert \lambda_e \vert \leqslant C \, (\rho_e - \rho_i)^{\frac{\alpha-1}{ \alpha (s+2)} } [u]_{\partial G}^{\frac 1 {s+2} }.
\end{equation}
The desired result follows by reasoning as in the proof of
\thref{oAZAv7vy}, but using \eqref{eq:finalstepperimprovement} instead of~\eqref{eq:NEWNEW OLD lambda estimate}.
\end{proof}


\section{The exponent for a family of ellipsoids} \label{z1jzKMXt}

Throughout this section, we will denote a point \(x=(x_1,\dots ,x_n) \in \R^n\) by \(x=(x_1,x')\) where\footnote{Not to be confused with the notation of the method of moving planes where an apostrophe referred to a reflection across a hyperplane.} \(x'= (x_2,\dots,x_n)\in \R^{n-1}\).

\begin{prop} \thlabel{0LlyfQlx}
Suppose that \(n\geqslant 2\) is an integer, \(s\in (0,1)\), and \(\varepsilon>0\). Define \begin{align}\label{sjiwo555ruwytui333tyeiurtfr}
\Omega_\varepsilon :=  \bigg \{ (x_1,x')\in \R^n \text{ s.t. } \frac{x_1^2}{(1+\varepsilon)^2} + \vert x ' \vert^2 <1 \bigg \}.
\end{align} 

Then, for \(0<\varepsilon<1/4\), there exists a 1-parameter family \(G_\varepsilon \subset \Omega_\varepsilon\) with smooth boundary such that \begin{align*}
G_\varepsilon + B_{1/2} = \Omega_\varepsilon.
\end{align*}

Moreover, let \(u_\varepsilon \in C^2(\Omega_\varepsilon) \cap C^s(\R^n)\) be the unique solution to \begin{align}
\begin{PDE}
(-\Delta)^s u_\varepsilon &=1 &\text{in } \Omega_\varepsilon, \\
u_\varepsilon &=0 &\text{in } \R^n \setminus \Omega_\varepsilon .
\end{PDE} \label{kaPuQECN}
\end{align} 

Then, \begin{align}
\lim_{\varepsilon \to 0^+} \frac{[u_\varepsilon]_{\partial G_{\varepsilon}}}{\rho(\Omega_\varepsilon)} &= s \gamma_{n,s} \bigg ( \frac34 \bigg )^{s-1}  , \label{83Wl82Op}
\end{align}
where
\begin{align}\label{gammaenneesse}
\gamma_{n,s} := \frac{2^{-2s} \Gamma \big ( \frac n2 \big )} {\Gamma \left(\frac{n+2s}2 \right) \Gamma (1+s)}.
\end{align} 
\end{prop}

The existence of the 1-parameter family \(G_\varepsilon\) as in \thref{0LlyfQlx} is an easy corollary of the following geometric lemma.

\begin{lem} \thlabel{jUJs9Upe}
Let \(n\geqslant 2\) be an integer and suppose that \(\Omega\) is an open bounded subset of \(\R^n\) that satisfies the uniform interior ball condition with radius \(r_\Omega>0\). If \(\delta:\Omega \to \R\) is the distance to the boundary of \(\Omega\) and \(\Omega^{\rho} := \{ x\in \Omega \text{ s.t. } \delta(x)>\rho\}\) then for all \(\rho \in (0,r_\Omega)\), \begin{align*}
\Omega^\rho+B_\rho = \Omega.
\end{align*} Moreover, if the boundary of \(\Omega\) is \(C^k\) for some integer \( k \ge 2 \) then,
for all \(\rho \in (0,r_\Omega)\), the boundary of \( \Omega^{\rho}\) is also \(C^k\). 
\end{lem}

\begin{proof}
We will begin by showing that \(\Omega^{\rho}+B_\rho \subset \Omega\). If \(z\in \Omega^{\rho}+B_\rho \) then \(z=x+y \) for some \(x\in \Omega^{\rho}\) and \(y \in B_\rho\). Let \(R>0\) be the largest value possible such that \(B_R(x)\subset \Omega\). Then \(R=\delta(x)\), so \(R >\rho\). Hence, \begin{align*}
\vert z - x\vert = \vert y \vert <\rho < R,
\end{align*} so \(z\in B_R(x) \subset \Omega\). 

Now, let us show that \(\Omega \subset \Omega^{\rho}+B_\rho\). Let \(z\in \Omega\). If \(z \in \Omega^{\rho}\) then \(z=z+0\in \Omega^{\rho}+B_\rho\) and we are done, so we may assume that \(z\in \Omega \setminus \Omega^{\rho}\). Now, let \(R>0\) be the largest value possible such that~\(B_R(z)\subset \Omega\) and let \(p\in \partial \Omega\) be a point at which \(B_R(z)\) touches \(\partial \Omega\). Define the unit vector \begin{align*}
\nu (p) := \frac{p-z}{\vert p - z \vert}. 
\end{align*} Next, since \(B_R(z)\) is an interior touching ball, there exists some~\(\bar z \in \Omega \) such that~\(B_{r_\Omega}(\bar z)\) also touches~\(\partial \Omega\) at~\(p\) and \(p\), \(z\), and \(\bar z\) are collinear. Now, let \(\mu \in (0, \min \{ r_\Omega - \rho , R\})\) and define \begin{align*}
x := z - (\rho - R+\mu ) \nu (p) .
\end{align*} Since \( z = \bar z +(r_\Omega - R) \nu (p)\), it follows that \( \vert x-\bar z\vert = \vert r_\Omega-\rho-\mu \vert < r_\Omega\), so \(x\in B_{r_\Omega}(\bar z)\). Hence, \(x \in \Omega\). Moreover, \(\vert x -p\vert = \rho +\mu > \rho\), so \(x\in \Omega^{\rho}\) and \(\vert z-x\vert = \rho -R +\mu <\rho\), so \(z-x\in B_\rho\). Thus, we have that \(z =x +(z-x)\in \Omega^{\rho} +B_\rho\) as required. This completes the first part of the proof.

The fact that the boundary of \( \Omega^{\rho}\) is \(C^k\) for any \(\rho \in (0,r_\Omega)\) is a simple consequence of \cite[Lemma 14.16]{MR1814364}. Indeed, following the proof of \cite[Lemma 14.16]{MR1814364}, we see that \(\delta\in C^k(\Gamma^{r_\Omega})\) where 
$$ \Gamma^{r_\Omega} := \{ x\in \Omega \text{ s.t. } \delta(x) < r_\Omega\}.$$ Then the boundary of~\(\Omega^{\rho}\) is simply the level set~\(\Omega \cap \{\delta=\rho\}\) which is contained in~\(\Gamma^{r_\Omega}\), so~\(\partial \Omega^{\rho}\) is~\(C^k\). 
\end{proof}

From \thref{jUJs9Upe}, we immediately obtain the proof of the first part of \thref{0LlyfQlx}: 

\begin{cor} \thlabel{FOHX1qV4}
Suppose that \(n\geqslant 2\) is an integer, \(s\in (0,1)\), \(\varepsilon>0\),
and~$\Omega_\varepsilon$ be as in~\eqref{sjiwo555ruwytui333tyeiurtfr}. 

Then, for~\(0<\varepsilon<1/4\), there exists a 1-parameter family \(G_\varepsilon \subset \Omega_\varepsilon\) with smooth boundary such that \begin{align*}
G_\varepsilon + B_{1/2} =  \Omega_\varepsilon.
\end{align*}
\end{cor}

\begin{proof}
It is easy to check that \(\Omega_\varepsilon\) satisfies the uniform interior ball condition with uniform radius given by \(r_{\Omega_\varepsilon}=(1+\varepsilon)^{-1}\). Hence, since \(\varepsilon<1/4\), we have that~\(r_{\Omega_\varepsilon}>1/2\), so if~\(G_\varepsilon := \Omega_\varepsilon^\rho\) with~\(\rho = 1/2\) (i.e. \(G_\varepsilon := \{ x\in \Omega_\varepsilon \text{ s.t. } \dist (x, \partial \Omega_\varepsilon)>1/2\}\)) then \thref{jUJs9Upe} implies that~\(G_\varepsilon\) satisfies all the required properties. 
\end{proof}

The remainder of this section will be spent proving the second part of \thref{0LlyfQlx}. In theory, this is relatively simple: since \(\Omega_\varepsilon\) is an ellipsoid, the solution to the torsion problem~\eqref{kaPuQECN} is known explicitly, see \cite{MR4181195}, and is given by \begin{align*}
u_\varepsilon(x):= \gamma_{n,s,\varepsilon}  \bigg (1 - \frac{x_1^2}{(1+\varepsilon)^2} - \vert x ' \vert^2 \bigg )^s_+ 
\end{align*} where \begin{align*}
\gamma_{n,s,\varepsilon} := \frac{\gamma_{n,s}}{(1+\varepsilon)  {}_2F_1 \big ( \frac{n+2s}2 , \frac 12 ; \frac n2  ; 1-(1+\varepsilon)^2 \big)}
\end{align*} and~$\gamma_{n,s}$ is given by~\eqref{gammaenneesse}.

Here above \({}_2F_1 \) is the Hypergeometric function, see~\cite{MR1225604} for more details. Note that
\({}_2F_1 (a,b;c;0)=1\), so~\(\gamma_{n,s,0}=\gamma_{n,s}\), and~\(\gamma_{n,s}\) is precisely the constant for which~\(u_0 (x) := \gamma_{n,s}(1-\vert x \vert^2)^s_+\) satisfies~\((-\Delta )^su_0=1\) in \(B_1\).

Now, it is easy to check that \(\rho(\Omega_\varepsilon)=\varepsilon\), so to prove the second part of \thref{0LlyfQlx}, one simply has to find the first order expansion of \([u_{\varepsilon}]_{\partial G_\varepsilon}\) in \(\varepsilon\). However, in practice, this is not trivial for a few reasons: one, parametrizations of \(G_\varepsilon\) are algebraically complicated; two, \([u_{\varepsilon}]_{\partial G_\varepsilon}\) is defined in terms of a supremum, which, in general, does not commute well with limits; and three, the supremum is over a quotient whose numerator and denominator both depend on \(\varepsilon\) and whose denominator is not bounded away from zero.

We will now state a technical result that will allow us to postpone addressing these difficulties and to proceed directly to the proof of \thref{0LlyfQlx}. 

\begin{lem} \thlabel{M2fW7tQX} Let \(\varepsilon \in[0,1)\) and define \(a_\varepsilon,b_\varepsilon :[0,+\infty) \to \R\) and \(\phi_\varepsilon : B_1^{n-1} \to \R^n\) as follows \begin{equation}\label{defperf333}\begin{split}
a_\varepsilon (\tau):=\;& 1+\varepsilon - \frac 1 {2 \sqrt{1+ \big ( (1+\varepsilon)^2-1\big ) \tau^2 }} ,\\
b_\varepsilon (\tau):=\;&  1- \frac {1+\varepsilon} {2 \sqrt{1+ \big ( (1+\varepsilon)^2-1\big ) \tau^2 }} ,\\
{\mbox{and}} \qquad \phi_\varepsilon (r):=\;& \Big(a_\varepsilon(\vert r \vert) \sqrt{1-\vert r\vert^2} , b_\varepsilon(\vert r \vert ) r \Big). 
\end{split}\end{equation} 

Then, \begin{align*}
\bigg \vert &\frac{\vert (u_\varepsilon \circ \phi_\varepsilon)(r)-(u_\varepsilon\circ \phi_\varepsilon)(\tilde r ) \vert }{\varepsilon \vert \phi_\varepsilon(r) - \phi_\varepsilon(\tilde r) \vert } - \frac12 s \gamma_{n,s}\left(\frac34\right)^{s-1} \frac{\big \vert  \vert r \vert^2 - \vert \tilde r \vert^2 \big \vert }{\vert \phi_0(r) - \phi_0(\tilde r)\vert}\bigg \vert \leqslant C \varepsilon .
\end{align*} The constant \(C>0\) depends only on \(n\) and \(s\).  
\end{lem}

In the statement of~\thref{M2fW7tQX} and in what follows, we use the notation~$B^{n-1}_1$ for the unit ball in~$\R^{n-1}$.

We will withhold the proof of \thref{M2fW7tQX} until the end of the section. Given \thref{M2fW7tQX}, we can now complete the proof of \thref{0LlyfQlx}.

\begin{proof}[Proof \thref{0LlyfQlx}] The proof of the first statement regarding the existence of \(G_\varepsilon\) is the subject of \thref{FOHX1qV4}, so we will focus on proving~\eqref{83Wl82Op}. Since the ellipsoid is symmetric with respect to reflections across the plane \(\{x_1=0\}\), it follows that \(G_\varepsilon\) is too. Hence, we have that \begin{align*}
[u_\varepsilon]_{\partial G_\varepsilon} &= \sup_{\substack{x,y\in \partial G_\varepsilon^+ \\ x\neq y } } \bigg \{ \frac{\vert u_\varepsilon(x)-u_\varepsilon(y) \vert }{\vert x - y \vert } \bigg \}.
\end{align*} Here, we are using the notation that, given \(A\subset \R^n\), \(A^+ := A \cap \{x_1>0\}\).

Next, the upper half ellipsoid \(\Omega_\varepsilon^+\) can be parametrized by \(\bar \phi_\varepsilon : B_1^{n-1} \to \Omega_\varepsilon^+\) defined
by~\(\bar \phi_\varepsilon (r) := ((1+\varepsilon)\sqrt{1-\vert r \vert^2},r)\). Moreover, the outward pointing normal to~\(B_1\) is given by
\begin{align*}
\nu_\varepsilon(x) &:= \frac1{\sqrt{\frac{x_1^2}{(1+\varepsilon)^4} + \vert x' \vert^2}} \bigg ( \frac{x_1}{(1+\varepsilon)^2} , x' \bigg ),
\end{align*} so a parametrization of \(\partial G_\varepsilon^+\) is \(\phi_\varepsilon :B_1^{n-1} \to \R^n\) defined by~\(\phi_\varepsilon := \bar \phi_\varepsilon-\frac12 (\nu_\varepsilon \circ \bar \phi_\varepsilon) \), that is, \begin{align*}
\phi_\varepsilon(r) &= \big ( (1+\varepsilon)\sqrt{1-\vert r \vert^2}, r \big ) - \frac{1+\varepsilon}{2\sqrt{1+((1+\varepsilon)^2-1)\vert r \vert^2}} \bigg ( \frac{\sqrt{1-\vert r \vert^2}}{1+\varepsilon} ,r  \bigg )\\ &= \big(a_\varepsilon(\vert r\vert ) \sqrt{1-\vert r \vert^2} , b_\varepsilon(\vert r \vert) r\big) 
\end{align*} in the notation of \thref{M2fW7tQX}.

Hence, \begin{align*}
[u_\varepsilon]_{\partial G_\varepsilon} &= \sup_{\substack{r,\tilde r\in B_1^{n-1} \\ r\neq \tilde r } } \left\{  \frac{\vert (u_\varepsilon \circ \phi_\varepsilon)(r)-(u_\varepsilon\circ \phi_\varepsilon)(\tilde r ) \vert }{\vert \phi_\varepsilon(r) - \phi_\varepsilon(\tilde r) \vert } \right\}.
\end{align*} Next, using \thref{M2fW7tQX} and the second triangle inequality in the \(\sup\) norm (over \(r,\tilde r\in B_1^{n-1}\)), we obtain that  \begin{align*}
\bigg \vert \frac{[u_\varepsilon]_{\partial G_\varepsilon}}{\varepsilon} -&\frac12 s \gamma_{n,s}\left(\frac34\right)^{s-1}  \sup_{\substack{r,\tilde r\in B_1^{n-1} \\ r\neq \tilde r } } \bigg \{ \frac{\big \vert  \vert r \vert^2 - \vert \tilde r \vert^2 \big \vert }{\vert \phi_0(r) - \phi_0(\tilde r)\vert} \bigg \}\bigg \vert \\
&\leqslant \sup_{\substack{r,\tilde r\in B_1^{n-1} \\ r\neq \tilde r } } \bigg \vert \frac{\vert (u_\varepsilon \circ \phi_\varepsilon)(r)-(u_\varepsilon\circ \phi_\varepsilon)(\tilde r ) \vert }{\varepsilon \vert \phi_\varepsilon(r) - \phi_\varepsilon(\tilde r) \vert } - \frac12 s \gamma_{n,s}\left(\frac34\right)^{s-1}  \frac{\big \vert  \vert r \vert^2 - \vert \tilde r \vert^2 \big \vert }{\vert \phi_0(r) - \phi_0(\tilde r)\vert}\bigg \vert  \\
&\leqslant C\varepsilon,
\end{align*} so, by sending \(\varepsilon \to 0^+\), we find that \begin{align}
 \lim_{\varepsilon \to 0^+} \frac{[u_\varepsilon]_{\partial G_\varepsilon}}{\varepsilon} &= \frac12 s \gamma_{n,s}\left(\frac34\right)^{s-1} \sup_{\substack{r,\tilde r\in B_1^{n-1} \\ r\neq \tilde r } } \bigg \{ \frac{\big \vert  \vert r \vert^2 - \vert \tilde r \vert^2 \big \vert }{\vert \phi_0(r) - \phi_0(\tilde r)\vert} \bigg \} .\label{vTr8Kp6E}
\end{align}

Finally, we claim that \begin{align}
\sup_{\substack{r,\tilde r\in B_1^{n-1} \\ r\neq \tilde r } } \bigg \{ \frac{\big \vert  \vert r \vert^2 - \vert \tilde r \vert^2 \big \vert }{\vert \phi_0(r) - \phi_0(\tilde r)\vert} \bigg \} = 2. \label{GPAI62Da}
\end{align} Indeed, if \begin{align*}
Q:=\frac{\big (  \vert r \vert^2 - \vert \tilde r \vert^2 \big )^2 }{\vert \phi_0(r) - \phi_0(\tilde r)\vert^2}
\end{align*} then \begin{align*}
Q=   \frac{4\big (  \vert r \vert^2 - \vert \tilde r \vert^2 \big )^2}{\big (\sqrt{1-\vert r \vert^2}-\sqrt{1-\vert \tilde r \vert^2} \big )^2+\vert r - \tilde r \vert^2} \leqslant  \frac{4\big (  \vert r \vert^2 - \vert \tilde r \vert^2 \big )^2 }{(\sqrt{1-\vert r \vert^2}-\sqrt{1-\vert \tilde r \vert^2})^2+\big ( \vert r\vert  - \vert \tilde r \vert \big )^2}.
\end{align*} Next, let \(t,\tilde t \in (0,\pi/2)\) be such that \(\vert r\vert = \cos t\) and \(\vert \tilde r\vert = \cos \tilde t\). Then \begin{align*}
Q &\leqslant  \frac{ 4(  \cos^2t - \cos^2 \tilde t  )^2 }{(\sin t -\sin \tilde t)^2+ ( \cos t   - \cos \tilde t )^2}.
\end{align*} Then, via elementary trigonometric identities, \begin{eqnarray*}&&
(  \cos^2t - \cos^2 \tilde t  )^2= 4 \sin^2  ( t+\tilde t  ) \sin^2 \bigg ( \frac{t-\tilde t}2  \bigg )  \cos^2 \bigg ( \frac{t-\tilde t}2  \bigg )  \\ \text{ and} &&
(\sin t -\sin \tilde t)^2+ ( \cos t   - \cos \tilde t )^2 = 4 \sin^2 \bigg ( \frac{t-\tilde t}2 \bigg ),
\end{eqnarray*} so \begin{align*}
Q \leqslant 4 \sin^2  ( t+\tilde t  )  \cos^2 \bigg ( \frac{t-\tilde t}2  \bigg ) \leqslant 4,
\end{align*} which shows that the left-hand side of~\eqref{GPAI62Da} is less than or equal to the right-hand side of~\eqref{GPAI62Da}. To show that the reverse inequality is also true, consider \(r:=e_1/\sqrt{2}\) and \(\tilde r := ( 1/\sqrt 2 +h)e_1\). Then \begin{align*}
\frac{\big \vert  \vert r \vert^2 - \vert \tilde r \vert^2 \big \vert }{\vert \phi_0(r) - \phi_0(\tilde r)\vert} &= \frac{2\vert h \vert \left(h+\sqrt{2}\right)}{\sqrt{1-h\sqrt 2-\sqrt{1-2 \sqrt{2} h-2 h^2}}} \to 2
\end{align*} as \(h \to 0\) which completes the proof of~\eqref{GPAI62Da}. Thus, recalling that \(\rho(\Omega_\varepsilon) = \varepsilon\),~\eqref{83Wl82Op} follows from~\eqref{vTr8Kp6E} and~\eqref{GPAI62Da}. 
\end{proof}

We will now give the proof of \thref{M2fW7tQX}. We will do this over two lemmata.

\begin{lem} \thlabel{c3Ge7XSV}
Let \(\varepsilon \in[ 0,1)\) and  \(a_\varepsilon\), \(b_\varepsilon\), and \(\phi_\varepsilon\) be as in \thref{M2fW7tQX}. 

Then, for all~\(r,\tilde r\in B_1^{n-1}\) with~\(r\neq \tilde r \), \begin{align}
\bigg \vert \frac{\vert\phi_0(r)-\phi_0(\tilde r)\vert }{\vert \phi_\varepsilon(r)-\phi_\varepsilon(\tilde r)\vert} -1 \bigg \vert \leqslant C \varepsilon .\label{Wy0X8YzO}
\end{align} The constant \(C>0\) depends only on \(n\). 
\end{lem}

\begin{proof}
Observe that \begin{equation}\label{JqHdxdks}\begin{split}
\bigg \vert \frac{\vert\phi_0(r)-\phi_0(\tilde r)\vert }{\vert \phi_\varepsilon(r)-\phi_\varepsilon(\tilde r)\vert} -1 \bigg \vert  & = \bigg \vert \frac{\vert\phi_0(r)-\phi_0(\tilde r)\vert-\vert \phi_\varepsilon(r)-\phi_\varepsilon(\tilde r)\vert}{\vert \phi_\varepsilon(r)-\phi_\varepsilon(\tilde r)\vert}\bigg \vert  \\
&\leqslant \frac{\vert (\phi_0-\phi_\varepsilon)(r)- (\phi_0-\phi_\varepsilon)(\tilde r) \vert}{\vert \phi_\varepsilon(r)-\phi_\varepsilon(\tilde r)\vert}. 
\end{split}\end{equation}
Next, observe that \begin{align*}
\phi_0(r) &= A_\varepsilon(\vert r \vert ) \phi_\varepsilon (r)
\end{align*} where, for each \(\tau\in \R\), \(A_\varepsilon(\tau)\) is the \(n\times n\) matrix given by \begin{align*}
A_\varepsilon(\tau) &:= \begin{pmatrix}
\frac1{2a_\varepsilon(\tau)} & 0 \\ 0 & \frac1{2b_\varepsilon(\tau)} I_{n-1}
\end{pmatrix}.
\end{align*}  Here \(I_k\) denotes the \(k\times k\) identity matrix and we are thinking of \(\phi_\varepsilon\) and~\(\phi_0\) as column vectors.

It will be useful in the future to note that \(a_\varepsilon\) and \(b_\varepsilon\) are strictly increasing for \(\tau\geqslant 0\) (see~\eqref{mAVkC45x} below for \(a_\varepsilon\), the computation for \(b_\varepsilon\) is analogous), so, for all \(\tau \in [0,1)\), \begin{equation}
\begin{split}&
\frac12 +\varepsilon \leqslant a_\varepsilon(\tau) \leqslant 1 +\varepsilon - \frac 1{2(1+\varepsilon)}\leqslant \frac12 +C\varepsilon \\{\mbox{and }}\quad&
\frac12(1-\varepsilon) \leqslant b_\varepsilon(\tau)  \leqslant \frac 12 .
\end{split} \label{ijAZGImZ}
\end{equation} In particular, when \(\varepsilon\) is small, \(A_\varepsilon(\tau)\) is well-defined since \(a_\varepsilon,b_\varepsilon>0\). Hence, it follow from~\eqref{JqHdxdks} that \begin{equation}\label{bgfHZN8C}\begin{split}
&\bigg \vert \frac{\vert\phi_0(r)-\phi_0(\tilde r)\vert }{\vert \phi_\varepsilon(r)-\phi_\varepsilon(\tilde r)\vert} -1 \bigg \vert  \\
\leqslant\;& \frac{\vert (I_n-A_\varepsilon(\vert r \vert )) \phi_\varepsilon(r)- (I_n-A_\varepsilon(\vert \tilde r \vert ) )\phi_\varepsilon(\tilde r) \vert}{\vert \phi_\varepsilon(r)-\phi_\varepsilon(\tilde r)\vert}\\
\leqslant\;& \frac{\vert (I_n-A_\varepsilon(\vert r \vert )) \phi_\varepsilon(r)- (I_n-A_\varepsilon(\vert  r \vert ) )\phi_\varepsilon(\tilde r) \vert}{\vert \phi_\varepsilon(r)-\phi_\varepsilon(\tilde r)\vert} 
+\frac{\vert (A_\varepsilon(\vert \tilde r \vert )-A_\varepsilon(\vert r \vert )) \phi_\varepsilon(\tilde r) \vert}{\vert \phi_\varepsilon(r)-\phi_\varepsilon(\tilde r)\vert}   \\
\leqslant\;& \|I_n-A_\varepsilon(\vert r \vert )\|_{\textrm{OP}} + \frac{\| A_\varepsilon(\vert \tilde r \vert )-A_\varepsilon(\vert r \vert )\|_{\mathrm{OP}}  }{\vert \phi_\varepsilon(r)-\phi_\varepsilon(\tilde r)\vert}
\vert \phi_\varepsilon(\tilde r) \vert,
\end{split}\end{equation}
where \(\| \cdot \|_{\mathrm{OP}}\) denotes the matrix operator norm.

It follows from~\eqref{ijAZGImZ} that \begin{align}
\|I_n-A_\varepsilon(\tau )\|_{\textrm{OP}} &= \max \bigg \{\bigg \vert 1 - \frac 1{2a_\varepsilon(\tau)}\bigg \vert ,  \bigg \vert 1 - \frac 1{2b_\varepsilon(\tau)}\bigg \vert \bigg \}\leqslant C\varepsilon . \label{ZiPmtbEz}
\end{align} Moreover, if \(\| \cdot \|_{\mathrm F}\) denotes the Frobenius norm, then \begin{align*}
\| A_\varepsilon(\vert \tilde r \vert )-A_\varepsilon(\vert r \vert )\|_{\mathrm{OP}}^2 &\leqslant C \| A_\varepsilon(\vert \tilde r \vert )-A_\varepsilon(\vert r \vert )\|_{\mathrm{F}}^2 \\
&= C \left[\left( \frac 1 {a_\varepsilon(\vert r \vert) } - \frac 1 {a_\varepsilon(\vert \tilde r \vert) } \right)^2 +   \left( \frac 1 {b_\varepsilon(\vert r \vert) } - \frac 1 {b_\varepsilon(\vert \tilde r \vert) } \right)^2 \right]\\
&= C \left( \frac{\big (a_\varepsilon(\vert r \vert) - a_\varepsilon(\vert \tilde r \vert) \big )^2}{a_\varepsilon(\vert r \vert)^2 a_\varepsilon(\vert \tilde r \vert)^2} + \frac{\big (b_\varepsilon(\vert r \vert) - b_\varepsilon(\vert \tilde r \vert) \big )^2}{b_\varepsilon(\vert r \vert)^2 b_\varepsilon(\vert \tilde r \vert)^2} \right)\\
&\leqslant  C \Big ( \big (a_\varepsilon(\vert r \vert) - a_\varepsilon(\vert \tilde r \vert) \big )^2 + \big (b_\varepsilon(\vert r \vert) - b_\varepsilon(\vert \tilde r \vert) \big )^2 \Big ), 
\end{align*} where we used~\eqref{ijAZGImZ} to obtain the final inequality.

We also have that \begin{align*}
\vert \phi_\varepsilon(r)-\phi_\varepsilon(\tilde r)\vert^2 &= \Big ( a_\varepsilon(\vert r \vert ) \sqrt{1-\vert r \vert^2}-a_\varepsilon(\vert \tilde r \vert ) \sqrt{1-\vert \tilde r \vert^2} \Big )^2+\vert b_\varepsilon (\vert r \vert ) r - b_\varepsilon(\vert \tilde r \vert ) \tilde r \vert^2 \\
&\geqslant \Big ( a_\varepsilon(\vert r \vert ) \sqrt{1-\vert r \vert^2}-a_\varepsilon(\vert \tilde r \vert ) \sqrt{1-\vert \tilde r \vert^2} \Big )^2+\big ( b_\varepsilon (\vert r \vert ) \vert r \vert  - b_\varepsilon(\vert \tilde r \vert ) \vert  \tilde r \vert \big )^2.
\end{align*}

We claim that \begin{equation}
\begin{split}&
\big \vert a_\varepsilon(\vert r \vert) - a_\varepsilon(\vert \tilde r \vert) \big \vert \leqslant C \varepsilon \big \vert a_\varepsilon(\vert r \vert ) \sqrt{1-\vert r \vert^2}-a_\varepsilon(\vert \tilde r \vert ) \sqrt{1-\vert \tilde r \vert^2} \big \vert  \\ \text{and}\qquad &
\big \vert b_\varepsilon(\vert r \vert) - b_\varepsilon(\vert \tilde r \vert) \big \vert \leqslant C \varepsilon \big \vert  b_\varepsilon (\vert r \vert ) \vert r \vert  - b_\varepsilon(\vert \tilde r \vert ) \vert  \tilde r \vert \big \vert.
\end{split} \label{fERN80H1}
\end{equation} We will prove the inequality in~\eqref{fERN80H1} involving \(a_\varepsilon\), the proof of the inequality involving \(b_\varepsilon\) is entirely analogous. Let \(\bar a_\varepsilon(\tau ) := a_\varepsilon (\tau ) \sqrt{1-\tau^2}\). Then, by Cauchy's Mean Value Theorem, \begin{align}
\bigg \vert \frac{ a_\varepsilon(\vert r \vert) - a_\varepsilon(\vert \tilde r \vert) }{\bar a_\varepsilon(\vert r \vert) - \bar a_\varepsilon(\vert \tilde r \vert) }\bigg \vert &\leqslant \bigg \| \frac{a_\varepsilon'}{\bar a_\varepsilon'} \bigg \|_{L^\infty((0,1))}. \label{w3PCZfMC}
\end{align} On one hand,  \begin{align*}
a'_\varepsilon(\tau) &= \frac12 \varepsilon(\varepsilon+2) \tau \big ( 1 + \big ( (1+\varepsilon)^2-1 \big ) \tau^2  \big )^{-3/2},
\end{align*}so, for all \(\tau \in [0,1)\), \begin{align}
C^{-1} \varepsilon \tau   \leqslant a'_\varepsilon(\tau) \leqslant C \varepsilon \tau. \label{mAVkC45x}
\end{align} On the other hand, using~\eqref{ijAZGImZ} and~\eqref{mAVkC45x}, we have that \begin{align*}
\bar a_\varepsilon '(\tau ) &= a_\varepsilon '(\tau ) \sqrt{1-\tau^2} - \frac{\tau a_\varepsilon(\tau)}{\sqrt{1-\tau^2}} \leqslant C(\varepsilon -1)\tau  <0 . 
\end{align*}Hence, \(\vert \bar a_\varepsilon '(\tau ) \vert =  -a_\varepsilon'(\tau) \geqslant C \tau\), so \begin{align*}
 \bigg \vert \frac{a_\varepsilon'(\tau)}{\bar a_\varepsilon'(\tau)} \bigg \vert &=   \frac{a_\varepsilon'(\tau)}{\vert \bar a_\varepsilon'(\tau) \vert } \leqslant C \varepsilon, \qquad \text{for all } \tau \in (0,1),
\end{align*} which, along with~\eqref{w3PCZfMC}, implies the inequality for \(a_\varepsilon\) in~\eqref{fERN80H1}. 

Gathering these pieces of information, we obtain that
\begin{equation}\begin{split}
\frac{\| A_\varepsilon(\vert \tilde r \vert )-A_\varepsilon(\vert r \vert )\|_{\mathrm{OP}}}{\vert \phi_\varepsilon(r)-\phi_\varepsilon(\tilde r)\vert}
\leqslant\;&\frac{
C \sqrt{  \big ((a_\varepsilon(\vert r \vert) - a_\varepsilon(\vert \tilde r \vert) \big )^2 + \big (b_\varepsilon(\vert r \vert) - b_\varepsilon(\vert \tilde r \vert) \big )^2 }
}{ \sqrt{\Big ( a_\varepsilon(\vert r \vert ) \sqrt{1-\vert r \vert^2}-a_\varepsilon(\vert \tilde r \vert ) \sqrt{1-\vert \tilde r \vert^2} \Big )^2+\big ( b_\varepsilon (\vert r \vert ) \vert r \vert  - b_\varepsilon(\vert \tilde r \vert ) \vert  \tilde r \vert \big )^2}}
\\
\leqslant\;& C \varepsilon 
.\label{R7PfHtVD}
\end{split}\end{equation}

Finally, since \(G_\varepsilon \subset \Omega_\varepsilon \subset B_{1+\varepsilon}\), we have that
\begin{align}
\vert \phi_\varepsilon(\tilde r) \vert &\leqslant 1+\varepsilon \leqslant C. \label{5nLBFgsK}
\end{align} Thus,~\eqref{Wy0X8YzO} follows from~\eqref{bgfHZN8C},~\eqref{ZiPmtbEz},~\eqref{R7PfHtVD}, and~\eqref{5nLBFgsK}.
\end{proof}

\begin{lem} \thlabel{F3lTRgfZ} 
Let \(\varepsilon \in[ 0,1)\) and  \(a_\varepsilon\), \(b_\varepsilon\), and \(\phi_\varepsilon\) be as in \thref{M2fW7tQX}. 

Then, for all~\(r,\tilde r\in B_1^{n-1}\) with~\(r\neq \tilde r \),  \begin{align*}
\bigg \vert \frac{\vert (u_\varepsilon \circ \phi_\varepsilon)(r)- (u_\varepsilon \circ \phi_\varepsilon)(\tilde r)\vert}{\varepsilon\vert \phi_0(r) - \phi_0(\tilde r)\vert} -
\frac12 s \gamma_{n,s}\left(\frac34\right)^{s-1} 
\frac{\big \vert  \vert r \vert^2 - \vert \tilde r \vert^2 \big \vert }{\vert \phi_0(r) - \phi_0(\tilde r)\vert} \bigg \vert \leqslant C \varepsilon.
\end{align*}
The constant \(C>0\) depends only on \(n\) and \(s\). 
\end{lem}

\begin{proof}
Let \begin{align*}
v_\varepsilon(x) &:= \frac 1 {\gamma_{n,s,\varepsilon}} u_\varepsilon(x) = \left ( 1 - \frac{x_1^2}{(1+\varepsilon)^2} - \vert x'\vert^2 \right )^s 
\end{align*} 
and define
$$ \psi_\epsilon(\tau):=\frac{\displaystyle \left ( 1 - \frac{(a_\epsilon(\tau))^2(1-\tau^2)}{(1+\varepsilon)^2} - (b_\epsilon(\tau))^2\tau^2 \right )^s
-\left(\frac34\right)^s}{\epsilon}
+\frac12 s\left(\frac34\right)^{s-1}(1-\tau^2).$$
Exploiting the expressions of~$a_\epsilon$ and $b_\epsilon$ in~\eqref{defperf333}, we compute the derivative of~$\psi_\epsilon$ with respect to~$\tau$ as
\begin{eqnarray*}
\frac{d}{d\tau}\psi_\epsilon(\tau)
&=&
s \left[    -\frac{\tau ^3 (1+\epsilon) \left(2\epsilon + \epsilon^2 \right) \left( 1-\frac{1+\epsilon}{2 \sqrt{1+\tau ^2 \left(2\epsilon +\epsilon^2 \right)}}\right)}{
\left(1+\tau ^2 \left(2\epsilon+ \epsilon^2 \right)\right)^{3/2}}
+\frac{2 \tau  \left(1+\epsilon-\frac{1}{2 \sqrt{1+\tau ^2 \left(2\epsilon + \epsilon^2 \right)}} \right)^2}{(1+\epsilon)^2}\right.
\\&&\quad\left. 
-2 \tau  \left(1- \frac{1+\epsilon}{2 \sqrt{1+\tau ^2 \left(2\epsilon + \epsilon^2 \right)}}  \right)^2
-\frac{ \left(1-\tau ^2\right) \tau  \left(2\epsilon + \epsilon^2 \right) \left(1+\epsilon-\frac{1}{2 \sqrt{1+\tau ^2 \left(2\epsilon + \epsilon^2 \right)}} \right)}{
(1+\epsilon)^2 \left(1+\tau ^2 \left(2\epsilon + \epsilon^2 \right)\right)^{3/2}}\right]
\\&& \times
\left[1-\frac{\left(1-\tau ^2\right) \left(1+\epsilon-\frac{1}{2 \sqrt{1+\tau ^2 \left(2\epsilon + \epsilon^2 \right)}}\right)^2}{(1+\epsilon)^2}-\tau ^2 \left(1-\frac{1+\epsilon }{2 \sqrt{1+
\tau ^2 \left(2\epsilon + \epsilon^2 \right)}}\right)^2\right]^{s-1}.
\end{eqnarray*}

We claim that, for~$\epsilon$ sufficiently small,
\begin{equation}\label{di4o375bv96c34567834567578}
\left\|\frac{d}{d\tau}\psi_\epsilon\right\|_{L^\infty((0,1))}\le C\epsilon,
\end{equation}
for some~$C>0$ depending only on~$s$. Not to interrupt the flow of the argument, we defer the proof of~\eqref{di4o375bv96c34567834567578}
to Appendix~\ref{apptedcalc}.

Now, we notice that
$$ \frac{(v_\varepsilon  \circ \phi_\varepsilon)(r)-\left(\frac34\right)^s}{\epsilon}
+\frac12 s\left(\frac34\right)^{s-1}(1-|r|^2)=\psi_\epsilon(|r|).
$$
Therefore,
\begin{eqnarray*}&&
\bigg \vert \frac{\vert (v_\varepsilon \circ \phi_\varepsilon)(r)- (v_\varepsilon \circ \phi_\varepsilon)(\tilde r)\vert}{\varepsilon\vert \phi_0(r) - \phi_0(\tilde r)\vert} -\frac12 s\left(\frac34\right)^{s-1} \frac{\big \vert  \vert r \vert^2 - \vert \tilde r \vert^2 \big \vert}{\vert \phi_0(r) - \phi_0(\tilde r)\vert} \bigg \vert \\
&\leqslant&  \bigg \vert  \frac{  \varepsilon^{-1}\left[(v_\varepsilon \circ \phi_\varepsilon)(r)-\left(\frac34\right)^s\right]
-\frac12 s\left(\frac34\right)^{s-1} \vert r \vert^2}{\vert \phi_0(r) - \phi_0(\tilde r)\vert} -
\frac{\varepsilon^{-1}\left[(v_\varepsilon \circ \phi_\varepsilon)(\tilde r)-\left(\frac34\right)^s\right]
-\frac12 s\left(\frac34\right)^{s-1} \vert \tilde r \vert^2}{\vert \phi_0(r) - \phi_0(\tilde r)\vert} \bigg \vert   \\
&=&  \bigg \vert  \frac{  \varepsilon^{-1}\left[(v_\varepsilon \circ \phi_\varepsilon)(r)-\left(\frac34\right)^s\right]
+\frac12 s\left(\frac34\right)^{s-1} (1-\vert r \vert^2)}{\vert \phi_0(r) - \phi_0(\tilde r)\vert} -
\frac{\varepsilon^{-1}\left[(v_\varepsilon \circ \phi_\varepsilon)(\tilde r)-\left(\frac34\right)^s\right]
+\frac12 s\left(\frac34\right)^{s-1}(1- \vert \tilde r \vert^2)}{\vert \phi_0(r) - \phi_0(\tilde r)\vert} \bigg \vert   \\
&=& \frac{\vert \psi_\varepsilon ( \vert r \vert ) - \psi_\varepsilon( \vert \tilde r \vert ) \vert}{\vert \phi_0(r) - \phi_0(\tilde r)\vert} .
\end{eqnarray*} Moreover, \begin{align*}
\vert \phi_0( r  ) - \phi_0( \tilde r  ) \vert^2  &= \frac14 \Big ( \big ( \sqrt{1-\vert r \vert^2}-\sqrt{1-\vert \tilde r \vert^2} \big )^2 + \vert r - \tilde r \vert^2 \Big ) \geqslant C \big \vert \vert r \vert - \vert \tilde r \vert \big \vert^2 ,  
\end{align*} so, recalling~\eqref{di4o375bv96c34567834567578}, we conclude that
\begin{equation}\label{OAWOSaAl}\begin{split}&
\bigg \vert \frac{\vert (v_\varepsilon \circ \phi_\varepsilon)(r)- (v_\varepsilon \circ \phi_\varepsilon)(\tilde r)\vert}{\varepsilon\vert \phi_0(r) - \phi_0(\tilde r)\vert} - \frac12 s\left(\frac34\right)^{s-1}  \frac{\big \vert  \vert r \vert^2 - \vert \tilde r \vert^2 \big \vert}{\vert \phi_0(r) - \phi_0(\tilde r)\vert} \bigg \vert \leqslant
C \frac{\vert \psi_\varepsilon ( \vert r \vert ) - \psi_\varepsilon( \vert \tilde r \vert ) \vert}{\big\vert \vert r \vert - \vert \tilde r \vert \big \vert}\\
&\qquad\qquad\qquad
\leqslant  C \left\|\frac{d}{d\tau}\psi_\epsilon\right\|_{L^\infty((0,1))} 
\le  C\epsilon.
\end{split}\end{equation}

Finally, note that \(\gamma_{n,s,\varepsilon}\) depends
smoothly on \(\varepsilon\) for \(\varepsilon\) small and \(\gamma_{n,s,\varepsilon} \to \gamma_{n,s}\) as \(\varepsilon \to 0^+\), which implies that \begin{align}
\vert \gamma_{n,s,\varepsilon} - \gamma_{n,s} \vert &\leqslant C \varepsilon . \label{YH4k6yy6}
\end{align} Combing~\eqref{YH4k6yy6} with~\eqref{OAWOSaAl} gives the final result. 
\end{proof}

We can now prove \thref{M2fW7tQX}. 

\begin{proof}[Proof of \thref{M2fW7tQX}]
By Lemmata~\ref{c3Ge7XSV} and~\ref{F3lTRgfZ}, for all \(r,\tilde r \in B_1^{n-1}\) with \(r \neq \tilde r\), \begin{align*}
\bigg \vert &\frac{\vert (u_\varepsilon \circ \phi_\varepsilon)(r)-(u_\varepsilon\circ \phi_\varepsilon)(\tilde r ) \vert }{\varepsilon \vert \phi_\varepsilon(r) - \phi_\varepsilon(\tilde r) \vert } -
\frac12 s \gamma_{n,s}\left(\frac34\right)^{s-1} \frac{\big \vert  \vert r \vert^2 - \vert \tilde r \vert^2 \big \vert }{\vert \phi_0(r) - \phi_0(\tilde r)\vert}\bigg \vert \\
&\leqslant  \bigg \vert \frac{\vert (u_\varepsilon \circ \phi_\varepsilon)(r)-(u_\varepsilon\circ \phi_\varepsilon)(\tilde r ) \vert }{\varepsilon \vert \phi_0(r) - \phi_0(\tilde r) \vert } \bigg \vert \cdot \bigg \vert  \frac{\vert\phi_0(r)-\phi_0(\tilde r)\vert }{\vert \phi_\varepsilon(r)-\phi_\varepsilon(\tilde r)\vert} -1 \bigg \vert   \\ 
& \qquad +\bigg \vert \frac{\vert (u_\varepsilon \circ \phi_\varepsilon)(r)-(u_\varepsilon\circ \phi_\varepsilon)(\tilde r ) \vert }{\varepsilon \vert \phi_0(r) - \phi_0(\tilde r) \vert } -\frac12 s \gamma_{n,s}\left(\frac34\right)^{s-1}  \frac{\big \vert  \vert r \vert^2 - \vert \tilde r \vert^2 \big \vert }{\vert \phi_0(r) - \phi_0(\tilde r)\vert}\bigg \vert \\
&\leqslant C\varepsilon \bigg (  \bigg \vert \frac{\vert (u_\varepsilon \circ \phi_\varepsilon)(r)-(u_\varepsilon\circ \phi_\varepsilon)(\tilde r ) \vert }{\varepsilon \vert \phi_0(r) - \phi_0(\tilde r) \vert } \bigg \vert   +1 \bigg ).
\end{align*} Moreover, by \thref{F3lTRgfZ}, for \(\varepsilon\) small, \begin{align*}
 \bigg \vert \frac{\vert (u_\varepsilon \circ \phi_\varepsilon)(r)-(u_\varepsilon\circ \phi_\varepsilon)(\tilde r ) \vert }{\varepsilon \vert \phi_0(r) - \phi_0(\tilde r) \vert } \bigg \vert &\leqslant C \bigg (  \frac{\big \vert  \vert r \vert^2 - \vert \tilde r \vert^2 \big \vert }{\vert \phi_0(r) - \phi_0(\tilde r)\vert} +1 \bigg ) \leqslant C
\end{align*} using~\eqref{GPAI62Da}. This completes the proof. 
\end{proof}

\appendix 

\section{Proof of the claim in \eqref{di4o375bv96c34567834567578}}\label{apptedcalc}

We write
\begin{equation}\label{e032jjjjtb58403t4734t4iuytr}
\frac{d}{d\tau}\psi_\epsilon(\tau)=s g_\epsilon(\tau) (h_\epsilon(\tau))^{s-1},\end{equation}
where
\begin{eqnarray*}
g_\epsilon(\tau)&:=& -\frac{\tau ^3 (1+\epsilon) \left(2\epsilon + \epsilon^2 \right) 
\left( 1-\frac{1+\epsilon }{2 \sqrt{1+\tau ^2 \left(2\epsilon + \epsilon^2 \right)}}\right)}{
\left(1+\tau ^2 \left(2\epsilon + \epsilon^2 \right)\right)^{3/2}}
+\frac{2 \tau  \left(1+\epsilon-\frac{1}{2 \sqrt{1+\tau ^2 \left(2\epsilon + \epsilon^2 \right)}} \right)^2}{(1+\epsilon)^2}
\\&&\quad
-2 \tau  \left(1- \frac{1+\epsilon}{2 \sqrt{1+\tau ^2 \left(2\epsilon +\epsilon^2 \right)}}  \right)^2
-\frac{ \left(1-\tau ^2\right) \tau  \left(2\epsilon + \epsilon^2 \right) \left(1+\epsilon-\frac{1}{2 \sqrt{1+\tau ^2 \left(2\epsilon + \epsilon^2 \right)}} \right)}{
(1+\epsilon)^2 \left(1+\tau ^2 \left(2\epsilon + \epsilon^2 \right)\right)^{3/2}}\\
{\mbox{and }}\qquad
h_\epsilon(\tau)&:=&
1-\frac{\left(1-\tau ^2\right) \left(1+\epsilon-\frac{1}{2 \sqrt{1+\tau ^2 \left(2\epsilon + \epsilon^2 \right)}}\right)^2}{(1+\epsilon)^2}-\tau ^2 \left(1-\frac{1+\epsilon }{2 \sqrt{1+
\tau ^2 \left(2\epsilon + \epsilon^2 \right)}}\right)^2.
\end{eqnarray*}

We first prove that
\begin{equation}\label{prima000} |g_\epsilon(\tau)|\le C\epsilon,\end{equation}
for all~$\tau\in(0,1)$, for some constant~$C>0$.

For this, we observe that
\begin{equation}\label{prima111}\begin{split}&
\left| -\frac{\tau ^3 (1+\epsilon) \left(2\epsilon +\epsilon^2 \right) \left( 1-\frac{1+\epsilon}{2 \sqrt{1+\tau ^2 \left(2\epsilon + \epsilon^2 \right)}}\right)}{
\left(1+\tau ^2 \left(2\epsilon + \epsilon^2 \right)\right)^{3/2}}\right|\\
&\qquad\qquad\le
\left|\tau ^3 (1+\epsilon) \left(2\epsilon +\epsilon^2 \right) \left( 1-\frac{1+\epsilon }{2 \sqrt{1+\tau ^2 \left(2\epsilon + \epsilon^2 \right)}}\right)\right|
\\&\qquad\qquad\le 6\epsilon\left|  1-\frac{1+\epsilon}{2 \sqrt{1+\tau ^2 \left(2\epsilon+ \epsilon^2 \right)}} \right|\le 12\epsilon.
\end{split}\end{equation}
Similarly,
\begin{equation}\label{prima222}\begin{split}&
\left|-\frac{ \left(1-\tau ^2\right) \tau  \left(2\epsilon + \epsilon^2 \right) \left(1+\epsilon-\frac{1}{2 \sqrt{1+\tau ^2 \left(2\epsilon + \epsilon^2 \right)}} \right)}{
(1+\epsilon)^2 \left(1+\tau ^2 \left(2\epsilon + \epsilon^2 \right)\right)^{3/2}}\right|\\
&\qquad\le
\left|\left(1-\tau ^2\right) \tau  \left(2\epsilon + \epsilon^2 \right) \left(1+\epsilon-\frac{1}{2 \sqrt{1+\tau ^2 \left(2\epsilon + \epsilon^2 \right)}} \right)\right|
\\&\qquad \le 3\epsilon \left|1+\epsilon-\frac{1}{2 \sqrt{1+\tau ^2 \left(2\epsilon + \epsilon^2 \right)}}\right|\le 9\epsilon.
\end{split}\end{equation}

Additionally,
\begin{equation*}\begin{split}&
2\tau\left[
\frac{ \left(1+\epsilon-\frac{1}{2 \sqrt{1+\tau ^2 \left(2\epsilon + \epsilon^2 \right)}} \right)^2}{(1+\epsilon)^2}
-\left(1- \frac{1+\epsilon}{2 \sqrt{1+\tau ^2 \left(2\epsilon +\epsilon^2 \right)}}  \right)^2\right]\\
=\;&
2\tau\left[\left(1+\epsilon-\frac{1}{2 \sqrt{1+\tau ^2 \left(2\epsilon + \epsilon^2 \right)}} \right)^2
-\left(1- \frac{1+\epsilon}{2 \sqrt{1+\tau ^2 \left(2\epsilon +\epsilon^2 \right)}}  \right)^2
\right]\\&\qquad
+2\tau\left(1+\epsilon-\frac{1}{2 \sqrt{1+\tau ^2 \left(2\epsilon + \epsilon^2 \right)}} \right)^2\left[\frac{1}{(1+\epsilon)^2}-1\right]\\
=\;&
2\tau\left[\left(1+\epsilon-\frac{1}{2 \sqrt{1+\tau ^2 \left(2\epsilon + \epsilon^2 \right)}} \right)^2
-\left(1- \frac{1+\epsilon}{2 \sqrt{1+\tau ^2 \left(2\epsilon +\epsilon^2 \right)}}  \right)^2
\right]\\&\qquad
-2\tau\left(1+\epsilon-\frac{1}{2 \sqrt{1+\tau ^2 \left(2\epsilon + \epsilon^2 \right)}} \right)^2\frac{2\epsilon+\epsilon^2}{(1+\epsilon)^2}.
\end{split}\end{equation*}
We observe that 
\begin{equation*} \begin{split}&
\left(1+\epsilon-\frac{1}{2 \sqrt{1+\tau ^2 \left(2\epsilon + \epsilon^2 \right)}} \right)^2
-\left(1- \frac{1+\epsilon}{2 \sqrt{1+\tau ^2 \left(2\epsilon +\epsilon^2 \right)}}  \right)^2
\\
=\;&\left[\left(1+\epsilon-\frac{1}{2 \sqrt{1+\tau ^2 \left(2\epsilon + \epsilon^2 \right)}}\right)+\left(
1- \frac{1+\epsilon}{2 \sqrt{1+\tau ^2 \left(2\epsilon +\epsilon^2 \right)}} \right)\right]\\
&\qquad\times \left[\left(1+\epsilon-\frac{1}{2 \sqrt{1+\tau ^2 \left(2\epsilon + \epsilon^2 \right)}}\right)-
\left(
1- \frac{1+\epsilon}{2 \sqrt{1+\tau ^2 \left(2\epsilon +\epsilon^2 \right)}} \right)
\right]\\=\;&
\left(2+\epsilon- \frac{2+\epsilon}{2 \sqrt{1+\tau ^2 \left(2\epsilon +\epsilon^2 \right)}}\right)
\left(\epsilon+\frac{\epsilon}{2 \sqrt{1+\tau ^2 \left(2\epsilon + \epsilon^2 \right)}}\right)\\
=\;&\epsilon(2+\epsilon)
\left(1- \frac{1}{2 \sqrt{1+\tau ^2 \left(2\epsilon +\epsilon^2 \right)}}\right)
\left(1+\frac{1}{2 \sqrt{1+\tau ^2 \left(2\epsilon + \epsilon^2 \right)}}\right)\\
=\;&\epsilon(2+\epsilon)
\left(1- \frac{1}{4\big(1+\tau ^2 \left(2\epsilon +\epsilon^2 \right)\big)}\right),
\end{split}\end{equation*}
and therefore
\begin{equation*}\begin{split}&
\left|2\tau \left[\left(1+\epsilon-\frac{1}{2 \sqrt{1+\tau ^2 \left(2\epsilon + \epsilon^2 \right)}} \right)^2
-\left(1- \frac{1+\epsilon}{2 \sqrt{1+\tau ^2 \left(2\epsilon +\epsilon^2 \right)}}  \right)^2\right]\right|
\\ &\qquad \le 6\epsilon
\left|1- \frac{1}{4\big(1+\tau ^2 \left(2\epsilon +\epsilon^2 \right)\big)}\right|\le 12\epsilon .
\end{split}\end{equation*}
Furthermore,
\begin{eqnarray*}&&
\left|-2\tau\left(1+\epsilon-\frac{1}{2 \sqrt{1+\tau ^2 \left(2\epsilon + \epsilon^2 \right)}} \right)^2\frac{2\epsilon+\epsilon^2}{(1+\epsilon)^2}\right|
\le 6\epsilon\left|\left(1+\epsilon-\frac{1}{2 \sqrt{1+\tau ^2 \left(2\epsilon + \epsilon^2 \right)}} \right)^2\right|\le 54\epsilon.
\end{eqnarray*}
Gathering these pieces of information, we obtain that
\begin{eqnarray*}
\left|2\tau\left[
\frac{ \left(1+\epsilon-\frac{1}{2 \sqrt{1+\tau ^2 \left(2\epsilon + \epsilon^2 \right)}} \right)^2}{(1+\epsilon)^2}
-\left(1- \frac{1+\epsilon}{2 \sqrt{1+\tau ^2 \left(2\epsilon +\epsilon^2 \right)}}  \right)^2\right]\right|\le 66\epsilon.
\end{eqnarray*}

{F}rom this, \eqref{prima111} and~\eqref{prima222}
we obtain the claim in~\eqref{prima000}.

We now show that, for~$\epsilon$ sufficiently small,
\begin{equation}\label{djiwecy4i83yci437t539kkkk}
h_\epsilon(\tau)\ge c,
\end{equation}
for all~$\tau\in(0,1)$, for some~$c>0$.

For this, we observe that
\begin{equation*}\begin{split}&1-
\frac{\left(1-\tau ^2\right) \left(1+\epsilon-\frac{1}{2 \sqrt{1+\tau ^2 \left(2\epsilon + \epsilon^2 \right)}}\right)^2}{(1+\epsilon)^2}
-\tau ^2 \left(1-\frac{1+\epsilon }{2 \sqrt{1+
\tau ^2 \left(2\epsilon + \epsilon^2 \right)}}\right)^2\\ =\;&
1-\frac{ \left(1+\epsilon-\frac{1}{2 \sqrt{1+\tau ^2 \left(2\epsilon + \epsilon^2 \right)}}\right)^2}{(1+\epsilon)^2}
+\tau^2\left[\frac{ \left(1+\epsilon-\frac{1}{2 \sqrt{1+\tau ^2 \left(2\epsilon + \epsilon^2 \right)}}\right)^2}{(1+\epsilon)^2}-
\left(1-\frac{1+\epsilon }{2 \sqrt{1+
\tau ^2 \left(2\epsilon + \epsilon^2 \right)}}\right)^2\right]\\
=\;&\frac{ 1+\epsilon }{(1+\epsilon)^2\sqrt{1+\tau ^2 \left(2\epsilon + \epsilon^2 \right)}}-
\frac{ 1}{4(1+\epsilon)^2 \big({1+\tau ^2 \left(2\epsilon + \epsilon^2 \right)}\big)}
\\&\qquad
+\tau^2\left[\frac{ \left(1+\epsilon-\frac{1}{2 \sqrt{1+\tau ^2 \left(2\epsilon + \epsilon^2 \right)}}\right)^2}{(1+\epsilon)^2}-
\left(1-\frac{1+\epsilon }{2 \sqrt{1+
\tau ^2 \left(2\epsilon + \epsilon^2 \right)}}\right)^2\right]\\
=\;&\frac{ 1  }{(1+\epsilon)\sqrt{1+\tau ^2 \left(2\epsilon + \epsilon^2 \right)}}-
\frac{ 1}{4(1+\epsilon)^2 \big({1+\tau ^2 \left(2\epsilon + \epsilon^2 \right)}\big)}
\\ &\qquad
+\tau^2\epsilon(2+\epsilon)\left[
1- \frac{1}{4\big(1+\tau ^2 \left(2\epsilon +\epsilon^2 \right)\big)}
-\left(1+\epsilon-\frac{1}{2 \sqrt{1+\tau ^2 \left(2\epsilon + \epsilon^2 \right)}} \right)^2\frac{1}{(1+\epsilon)^2}
\right]\\
=\;&\frac{ 1}{(1+\epsilon)\sqrt{1+\tau ^2 \left(2\epsilon + \epsilon^2 \right)}}-
\frac{ 1}{4(1+\epsilon)^2 \big({1+\tau ^2 \left(2\epsilon + \epsilon^2 \right)}\big)}
\\ &\qquad
+\tau^2\frac{\epsilon(2+\epsilon)}{(1+\epsilon)^2}\left[\frac{1+\epsilon}{\sqrt{1+\tau ^2 \left(2\epsilon + \epsilon^2 \right)} }
-\frac{1+(1+\epsilon)^2}{4\big({1+\tau ^2 \left(2\epsilon + \epsilon^2 \right)}\big)}
\right]\\
\ge\;&\frac{1}{(1+\epsilon)^2}-\frac{1}{4(1+\epsilon)^2}-\frac{\epsilon(2+\epsilon)}{(1+\epsilon)^2}\left|\frac{1+\epsilon}{\sqrt{1+\tau ^2 \left(2\epsilon + \epsilon^2 \right)} }
-\frac{1+(1+\epsilon)^2}{4\big({1+\tau ^2 \left(2\epsilon + \epsilon^2 \right)}\big)}
\right|\\
\ge\;&\frac1{(1+\epsilon)^2}\left(1-\frac14-\frac{\epsilon(2+\epsilon)(4+4\epsilon+\epsilon^2)}{2}\right)\\
\ge\;&\frac12,
\end{split}\end{equation*}
for~$\epsilon$ sufficiently small. This establishes~\eqref{djiwecy4i83yci437t539kkkk}.

{F}rom~\eqref{e032jjjjtb58403t4734t4iuytr}, \eqref{prima000}
and~\eqref{djiwecy4i83yci437t539kkkk}, we obtain the desired claim 
in~\eqref{di4o375bv96c34567834567578}.

\printbibliography
\vfill
\end{document}